\documentclass[12pt,a4paper,reqno]{amsart}
\usepackage{amsmath,amssymb,amsthm,amsaddr}
\usepackage[margin=2.5cm,top=3.5cm,bottom=3.5cm,footskip=1cm,headsep=1cm]{geometry}
\usepackage[utf8]{inputenc}
\usepackage{graphicx}
\usepackage{caption}
\usepackage{subcaption}
\usepackage{mathtools}
\usepackage{tikz}

\author{H. Egger \and L. Sch\"obel-Kr\"ohn}
\address{Department of Mathematics, TU Darmstadt, Germany}
\email{egger@mathematik.tu-darmstadt.de}
\email{schoebel-kroehn@mathematik.tu-darmstadt.de}

\title[Chemotaxis on networks]{Chemotaxis on networks:\\Analysis and numerical approximation}

\newtheorem{lemma}{Lemma}[section]
\newtheorem{problem}[lemma]{Problem}
\newtheorem{theorem}[lemma]{Theorem}

\theoremstyle{definition}

\newtheorem{remark}[lemma]{Remark}
\newtheorem*{example*}{Example}

\def\dt{\partial_t}
\def\dtt{\partial_{tt}}
\def\div{\operatorname{div}}

\def\dn{\partial_n}
\def\dx{\partial_x}
\def\dxx{\partial_{xx}}
\def\dtau{d_\tau}

\def\E{\mathcal{E}}
\def\G{\mathcal{G}}
\def\V{\mathcal{V}}

\def\RR{\mathbb{R}}

\def\eps{\epsilon}

\def\la{\langle}
\def\ra{\rangle}

\numberwithin{equation}{section}
\numberwithin{table}{section}
\numberwithin{figure}{section}

\begin{document}

\begin{abstract} 
We consider the Keller-Segel model of chemotaxis on one-dimensional networks.
Using a variational characterization of solutions, positivity preservation, conservation of mass, and energy estimates, we establish global existence of weak solutions and uniform bounds. This extends related results of Osaki and Yagi to the network context. We then analyze the discretization of the system by finite elements and an implicit time-stepping scheme. Mass lumping and upwinding are used to guarantee the positivity of the solutions on the discrete level. This allows us to deduce uniform bounds for the numerical approximations and to establish order optimal convergence of the discrete approximations to the continuous solution without artificial smoothness requirements. In addition, we prove convergence rates under reasonable assumptions. Some numerical tests are presented to illustrate the theoretical results.
\end{abstract}

\maketitle

\begin{quote}
\end{quote}

\begin{quote}
\noindent
{\small {\bf AMS-classification (2000):}
35K15, 
35R02, 
65M60, 
92C17  
}
\end{quote}


\section{Introduction}

Back in the 1970s, Keller and Segel \cite{KellerSegel71} introduced their celebrated model of chemotaxis 
describing the collective movement of cellular organisms in response to the distribution of a chemical substance.
According to \cite{ChildressPercus81,HillenPainter09}, the {\em minimal system} given by
\begin{align*}
\dt u - \div (\alpha \nabla u - \chi u \nabla c) &= 0, \\
\dt c - \div(\beta \nabla c) +  \gamma c &= \delta u,
\end{align*}
serves as a prototype for studying many interesting mathematical aspects of chemotaxis.  
In the context of biological applications, $u$ denotes the density of the population of interest and $c$ is the concentration of the chemoattractant. 
The differential equations are usually augmented by homogeneous Neumann boundary conditions $\dn u = \dn c = 0$ 
which leads to global conservation of the population and to preservation of positivity in both variables.

The analysis of this parabolic-parabolic model of chemotaxis is well understood by now. 
Local existence of a unique solution was proven by Yagi in \cite{Yagi97} and global existence was established by Nagai, Senba, and Yoshida \cite{NagaiSenbaYoshida97} under a smallness condition on the initial data. 
For large initial mass, blow-up in finite time was demonstrated in \cite{HerreroVelazquez97,JaegerLuckhaus92} for dimension $d \ge 2$.
Finite-time blow-up cannot occur in dimension $d=1$ and thus solutions for the nonlinear parabolic system exist globally in time \cite{HillenPainter01,OsakiYagi01}. 
We refer to \cite{HillenPainter09,Horstmann03} for an overview about various models of chemotaxis and further theoretical results.

Apart from the analysis, also the numerical approximation of chemotaxis models has attracted significant interest in the literature.
Nakaguchi and Yagi \cite{NakaguchiYagi01} proposed fully discrete schemes obtained by Galerkin approximation in space and Runge-Kutta methods in time, and they established convergence rates under appropriate smoothness assumptions on the true solution. 
Filbet \cite{Filbet06} studied finite volume approximations and proved convergence also for non-smooth solutions.
Higher order schemes have been studied in \cite{ChertockEpshteynHuKurganov17,Epshteyn09}.  
In order to retain the conservation of mass and the positivity of the numerical solution, 
Saito \cite{Saito12} proposed and analyzed an upwind finite element scheme
for multidimensional problems; a slightly different approach was considered \cite{StrehlSokolovKuzminHorstmannTurek13}.
Let us also mention recent work \cite{BorscheKlarPham16,Gosse12,NataliniRibot12} concerning the approximation of hyperbolic models of chemotaxis.

In this paper, we study chemotaxis on one-dimensional networks modeled by a system of partial differential-algebraic equations on finite metric graphs \cite{Mugnolo14}.
A one-dimensional version of the minimal system is imposed on every edge of the graph and complemented by algebraic coupling conditions at the vertices to ensure continuity of the solution and conservation of mass across junctions.
A corresponding system has been proposed and investigated by Borsche et al. \cite{BorscheGoettlichKlarSchillen14}, who considered a positivity preserving finite volume discretization and established the well-posedness of their scheme.
Let us also mention recent work by Camilli and Corrias \cite{CamilliCorrias17},
who considered problems with constant coefficients and proved well-posedness by 
perturbation arguments and using the mapping properties of the heat semi-group on networks. 
Numerical methods for hyperbolic models of chemotaxis on networks were investigated by Bretti et al. \cite{BrettiNataliniRibot14}. 

\medskip 

Before we proceed, let us briefly discuss the main contributions of our manuscript: 
By extending the functional analytic framework of Osaki and Yagi \cite{OsakiYagi01}, 
we consider the chemotaxis problem on networks as a semilinear parabolic system which allows us to establish existence of a local solution by Galerkin approximation, energy estimates, and perturbation arguments. 
Following the ideas of \cite{HillenPotapov04,OsakiYagi01}, we further show that the solution remains positive,  provided that the initial values are positive, and we prove that the total mass of the population is conserved for all time. This yields uniform a-priori estimates for the $L^1$-norm of the density $u$ and allows us to derive sharper energy estimates by which we can show that the solution can at most grow polynomially in time and hence exists globally.
These results can be seen as a natural generalization of those in \cite{HillenPotapov04,OsakiYagi01} to one-dimensional networks. However, we use somewhat different energy estimates in our proofs which allows us to apply our analysis also to problems with discontinuous model parameters and to networks of rather general topology. 
Our method of proof also differs from that in \cite{CamilliCorrias17} and our results are more general, in particular, 
our analysis covers the case of non-constant and discontinuous coefficients.
As preparation for the second part of the manuscript, we also establish higher regularity of solutions.

After having proved the global existence and uniqueness of solutions, we 
turn to their systematic numerical approximation. For the discretization, we here consider 
a Galerkin approximation in space by finite elements combined with an implicit time-stepping scheme.
In order to ensure positivity of the discrete solutions, we employ a mass-lumping strategy and an upwind discretization 
for the convective term. The resulting scheme has a similar structure as that considered 
by Saito \cite{Saito12} for chemotaxis problems in multiple dimensions, but the formulation of 
our scheme is closer to that of the continuous problem, which facilitates the analysis substantially. 
Some alternative but related approaches can be found in \cite{Filbet06} and \cite{StrehlSokolovKuzminHorstmannTurek13}.
Using similar methods of proof as on the analytical level, we derive uniform bounds for the discrete approximations 
and we establish convergence of the numerical solution and order optimal convergence rates under reasonable smoothness assumptions. 
Our analysis is somewhat sharper and more general than that presented in \cite{Saito12}. 
In particular, we do not require a strong restriction on the time step to guarantee the stability of our fully discrete scheme and we obtain convergence in the general case without artificial smoothness assumptions.

\medskip 

The remainder of the manuscript is organized as follows:
In Section~\ref{sec:problem}, we introduce our notation and the problem under investigation, and we give a variational characterization of solutions which will  be the basis for the rest of the manuscript.
In Section~\ref{sec:wellposed}, we establish the existence and uniqueness of solutions and derive uniform bounds that grow at most polynomially in time. In addition, we prove higher regularity of solutions under natural smoothness and compatibility conditions
on the initial data.
The numerical approximation is introduced in Section~\ref{sec:discretization}
and we establish uniform global bounds for the discrete solutions.
In Sections~\ref{sec:convergence} and \ref{sec:rates}, we prove the convergence of discrete solutions to the true
solution and we establish order optimal convergence rates under reasonable smoothness assumptions. 
For illustration of our theoretical findings, we present some numerical tests
in Section~\ref{sec:numerics} and we close the presentation with a short summary 
and a discussion of possible directions for future research.

\section{Preliminaries} \label{sec:problem}

We start by introducing our notation and then formally state the chemotaxis problem on the network to be considered for the rest of the paper.

\subsection{Network topology} 

Let $(\V,\E)$ be a finite directed and connected graph \cite{Berge} with vertices $v \in \V$ and edges $e \in \E$. To any edge $e=(v_1,v_2)$ pointing from vertex $v_1$ to $v_2$, we set $n_e(v_1)=-1$ and $n_e(v_2)=1$. The matrix with entries $N_{ij} = n_{e_j}(v_i)$ then is the incidence matrix of the graph. For any $v \in V$, we denote by $\E(v)=\{e \in \E: e=(v,\cdot) \text{ or } e=(\cdot,v) \}$ the set of edges starting or ending at $v$, and for $e \in \E$ we define $\V(e)=\{v \in \V: e=(v,\cdot) \text{ or } e=(\cdot,v)\}$. We further denote by $\V_b=\{v \in \V : |\E(v)|=1\}$ the set of boundary vertices and call $\V_0=\V \setminus \V_b$ the set of interior vertices. 
A small example illustrating our notation is presented in Figure~\ref{fig:graph}.
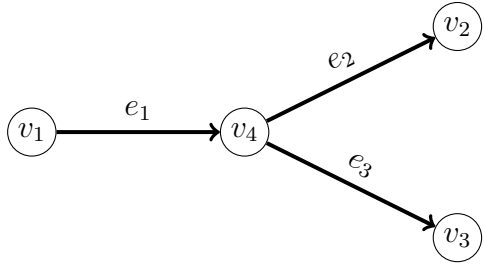
\begin{figure}[ht!]
\begin{minipage}[c]{.3\textwidth}
\hspace*{-2.5em}
\begin{tikzpicture}[scale=.7]
\node[circle,draw,inner sep=2pt] (v1) at (0,2) {$v_1$};
\node[circle,draw,inner sep=2pt] (v4) at (4,2) {$v_4$};
\node[circle,draw,inner sep=2pt] (v2) at (8,4) {$v_2$};
\node[circle,draw,inner sep=2pt] (v3) at (8,0) {$v_3$};
\draw[->,thick,line width=1.5pt] (v1) -- node[above] {$e_1$} ++(v4);
\draw[->,thick,line width=1.5pt] (v4) -- node[above,sloped] {$e_2$} ++(v2);
\draw[->,thick,line width=1.5pt] (v4) -- node[above,sloped] {$e_3$} ++(v3);
\end{tikzpicture}
\end{minipage}
\caption{Graph $\G=(\V,\E)$ with vertices $\V=\{v_1,v_2,v_3,v_4\}$ and edges $\E=\{e_1,e_2,e_3\}$
defined by $e_1=(v_1,v_4)$, $e_2=(v_4,v_2)$, and $e_3=(v_4,v_3)$. 
Here $\V_0=\{v_4\}$, $\V_b=\{v_1,v_2,v_3\}$, $\E(v_1)=\{e_1\}$, $\E(v_2)=\{e_2\}$, $\E(v_3)=\{e_3\}$, and $\E(v_4)=\{e_1,e_2,e_3\}$.
The non-zero entries of the incidence matrix are $n_{e_1}(v_1)=n_{e_2}(v_4)=n_{e_3}(v_4)=-1$ and $n_{e_1}(v_4)=n_{e_2}(v_2)=n_{e_3}(v_3)=1$.\label{fig:graph}} 
\end{figure}

\subsection{Function spaces}

To any edge $e \in \E$ we associate a positive length $\ell_e>0$ and with some abuse of notation, we always identify the topological edge $e$ with the geometric interval $[0,\ell_e]$ in the sequel. Let $\ell$ be the vector with entries $\ell_e$. Following the notation of \cite{Mugnolo14}, we call the triple $\G=(\V,\E,\ell)$ a {\em metric graph}. 
We further denote by 
\begin{align*}
  L^2(\E) = \{v :  v_e = v|_e \in L^2(e) = L^2(0,\ell_e)\}
\end{align*}
the space of square integrable functions on $\E$ which is a Hilbert space when equipped with the natural scalar product
\begin{align*}
\la v,w \ra_\E = \sum_{e \in \E} \la v_e,w_e \ra_{e} = \sum_{e \in \E} \int_0^{\ell_e} v_e w_e dx.
\end{align*}
The corresponding norm is given by $\|v\|_{L^2(\E)} = \la v,v \ra_\E^{1/2}$ and the spaces $L^p(\E)$, $1 \le p \le \infty$ are defined accordingly.
In addition, we will also make use of the function space
\begin{align*}
H^1(\E) = \{w \in L^2(\E) : \dx w_e \in L^2(e) \text{ and }  w_e(v) = w_{e'}(v) \quad \forall e,e' \in \E(v), \ v \in \V_0\}
\end{align*}
consisting of continuous functions with square integrable weak derivatives.
This space is complete when equipped with the norm defined by $\|v\|_{H^1(\E)}^2 = \|v\|_{L^2(\E)}^2 + \|\dx v\|_{L^2(\E)}^2$. 
We denote by $H^1(\E)'$ the dual space of $H^1(\E)$ consisting of continuous linear functionals $l : H^1(\E) \to \RR$. Note that by continuity and density, the scalar product $\la \cdot,\cdot \ra_\E$ can be extended to the duality product on $H^1(\E)' \times H^1(\E)$, for which we use the same symbol. 

For $T>0$ and some Banach space $X$, we denote by $L^p(0,T;X)$ the space of measurable functions with values $v(t) \in X$ and with finite norm $\|v\|_{L^p(0,T;X)}^p = \int_0^T \|v(t)\|_{X}^p$. 
Spaces of differentiable functions in time are denoted by $W^{k,p}(0,T;X)$ and equipped with their natural norms.
As usual, we write $H^k(0,T;X)=W^{k,2}(0,T;X)$ for convenience. 
Let us recall that the embedding of the energy space
\begin{align*}
W(0,T)=L^2(0,T;H^1(\E)) \cap H^1(0,T;H^1(\E)')
\end{align*}
into $ C([0,T];L^2(\E))$ is continuous, which can be proven with similar arguments as on single intervals. 
As a consequence, the evaluation $v(t)$ is well-defined for functions $v \in W(0,T)$ and one has a uniform bound
\begin{align*}
\|v\|_{L^\infty(0,T;L^2(\E))} \le C (\|v\|_{L^2(0,T;H^1(\E))}+\|\dt v\|_{L^2(0,T;H^1(\E)')})
\end{align*}
with a constant $C$ independent of $v$; we refer to \cite{Evans10} for details and further references.

\subsection{Problem statement}

Let $\G=(\V,\E,\ell)$ be a finite directed metric graph as introduced above.
On every edge $e \in \E$, the chemotactic movement shall be described by 
\begin{alignat}{2}
\dt u_e - \dx ( \alpha_e \dx u_e - \chi_e u_e \dx c_e) &= 0, \quad \qquad && e \in \E, \ t>0, \label{eq:sys1}\\
\dt c_e - \dx (\beta_e \dx c_e) + \gamma_e c_e &= \delta_e u_e, \qquad && e \in \E, \ t>0,   \label{eq:sys2}
\end{alignat}
with model parameters $\alpha,\beta,\gamma,\delta,\chi$ to be specified below. 
Recall that $f_e=f|_e$ denotes the restriction of a function $f$ onto the edge $e$. 
In addition to the above equations, we assume that the solution is continuous across vertices, i.e., 
\begin{alignat}{2}
u_e(v) = u_{e'}(v), \qquad c_e(v) = c_{e'}(v), \qquad e,e' \in \E(v), \ v \in \V_0, \ t>0, \label{eq:sys3}
\end{alignat}
and we require that the population and concentration are conserved at all vertices, i.e., 
\begin{alignat}{2} 
\sum_{e \in \E(v)} (\alpha_e(v) \dx u_e(v) - \chi_e(v) u_e(v) \dx c_e(v)) n_e(v) &= 0, \qquad && v \in \V_0 \cup \V_b, \ t>0, \label{eq:sys4}\\
\sum_{e \in \E(v)} \beta_e(v) \dx c_e(v) n_e(v) &= 0, \qquad && v \in \V_0 \cup \V_b, \ t>0. \label{eq:sys5}
\end{alignat}
These conditions imply that no mass is gained or lost at interior vertices $v \in \V_0$ or across the boundary $v \in \V_b$ of the network. 
To complete the definition of our model problem, we finally assume to have knowledge of the initial values 
\begin{align}
u_e(0) = u_{e,0}, \qquad c_e(0) = c_{e,0}, \qquad e \in \E.   \label{eq:sys6}
\end{align}
Any pair of sufficiently regular functions $(u,c)$, for instance, continuously differentiable in time and twice continuously differentiable in space on every edge, will be called a \emph{regular solution} of \eqref{eq:sys1}--\eqref{eq:sys6} on $[0,T]$, 
if it satisfies the equations in a pointwise sense.

\subsection{Variational characterization of solutions}

Throughout our analysis, we will make use of the following weak characterization of regular solutions. 
\begin{lemma} \label{lem:weak} 
Let $(u,c)$ be a regular solution of \eqref{eq:sys1}--\eqref{eq:sys6} on $[0,T]$. 
Then 
\begin{alignat}{2}
\la\dt u(t),v\ra_\E + \la\alpha \dx u(t),\dx v\ra_\E &= \la\chi u(t) \dx c(t), \dx v\ra_\E,  \label{eq:var1}\\
\la\dt c(t),q\ra_\E + \la\beta \dx c(t),\dx q\ra_\E + \la\gamma c(t), q\ra_\E &= \la\delta u(t),q\ra_\E, \label{eq:var2}
\end{alignat}
for all test functions $v,q \in H^1(\E)$ and all points $t \in [0,T]$ in time. 
\end{lemma}
\begin{proof}
Let us start with the second identity. Multiplication of \eqref{eq:sys2} by a test function $q_e$ on every edge $e$, integration over $e$, and summation over all edges $e \in \E$ leads to
\begin{align*}
\la\dt c,q\ra_\E + \la\gamma c,q\ra_\E - \la\delta u,q\ra_\E &= \la\dx (\beta \dx c),q\ra_\E.
\end{align*}
Via integration-by-parts on every edge $e$, the last term can be transformed to
\begin{align*}
\la\dx (\beta \dx c),q\ra_\E &= \sum_{e \in \E} -\la\beta_e \dx c_e, \dx q_e\ra_e + \sum_{e \in \E} \sum_{v \in \V(e)} \beta_e(v) \dx c_e(v) q_e(v) n_e(v).
\end{align*}
Note that $f_e|_{v_1(e)}^{v_2(e)} = f_e(v_2)-f_e(v_1) = \sum_{v \in \V(e)} f_e(v) n_e(v)$ by definition of $n_e(v)$.
Exchanging the order of summation and using the continuity condition \eqref{eq:sys3}, which implies that $q_e(v) = q(v)$ for some $q(v)$ and all $e \in \E(v)$, the last term can be further evaluated as
\begin{align*}
\sum_{e \in \E} \sum_{v \in \V(e)} \beta_e(v) \dx c_e(v) q_e(v) n_e(v)
=\sum_{v \in \V} q(v) \sum_{e \in \E(v)}  \beta_e(v) \dx c_e(v) n_e(v)
=0.
\end{align*}
For the last equality, we made use of the coupling condition \eqref{eq:sys5}.
A combination of the above formulas already yields the second identity of the lemma,
and the first one can be derived with very similar arguments. 
\end{proof}

\begin{remark} \label{rem:weak}
The equations \eqref{eq:var1}--\eqref{eq:var2} also make sense for less regular functions, e.g., 
\begin{align*}
  u &\in L^2(0,T;H^1(\E)) \cap H^1(0,T;H^1(\E)'),  \\
  c &\in L^\infty(0,T;H^1(\E)) \cap H^1(0,T;L^2(\E)). 
\end{align*}
The particular choice of these spaces will become clear from our analysis.
Such a pair of functions $(u,c)$ which satisfies \eqref{eq:var1}--\eqref{eq:var2} for a.a. $t \in [0,T]$ will be called a \emph{weak solution} of problem \eqref{eq:sys1}--\eqref{eq:sys5}. Note that the first term in \eqref{eq:var1} has to be interpreted as a duality product here. 
By standard embedding results \cite{Evans10}, one can see that $u,c \in C([0,T];L^2(\E))$ which allows to 
satisfy the initial values in a reasonable way. 
\end{remark}

\section{Well-posedness} \label{sec:wellposed}

In order to guarantee the existence and uniqueness of solutions of problem \eqref{eq:sys1}--\eqref{eq:sys6},
we make the following assumptions on the parameters and the initial values.
\begin{itemize}
  \item[(A1)] $\alpha,\beta,\gamma,\delta,\chi \in L^\infty(\E)$ such that $\gamma,\delta \ge 0$ as well as 
  $0 < \underline \alpha \le \alpha$ and $0 < \underline \beta \le \beta$ uniformly a.a. on $\E$ for some positive constants $\underline \alpha,\underline \beta$.
  For ease of presentation, we additionally assume that the coefficients are constant on every edge $e \in \E$.
 \item[(A2)] $u_0 \in L^2(\E)$ and $c_0 \in H^1(\E)$ with $u_0 \ge 0$ and $c_0 \ge 0$
\end{itemize}
We will denote by $\overline \alpha, \overline \beta, \overline \gamma, \overline \delta, \overline \chi$ the $L^\infty$ bounds for the coefficients. 
\subsection{Local solvability} \label{sec:local}
Using standard arguments for semilinear parabolic problems, one can now establish the local well-posedness of the problem under consideration.
\begin{theorem} \label{thm:local}
Let (A1)-(A2) hold. Then there exists a time horizon $T>0$, 
depending on the geometry of the graph, on the bounds for the coefficients, and inverse monotonically 
on $\|u_0\|_{L^2(\E)}$, $\|c_0\|_{H^1(\E)}$, such that the system \eqref{eq:sys1}--\eqref{eq:sys6} has a unique local weak solution 
\begin{align*}
  u &\in L^2(0,T;H^1(\E)) \cap H^1(0,T;H^1(\E)'),  \\
  c &\in L^\infty(0,T;H^1(\E)) \cap H^1(0,T;L^2(\E)),
\end{align*}
and the norm of the solution can be bounded by the norm of the initial data.
\end{theorem}
A detailed proof is given in the appendix. 
Let us mention already here that the particular functional analytic setting allows to consider \eqref{eq:sys1}--\eqref{eq:sys6} as a semilinear parabolic system and the result can thus be proven by a fixed-point argument. 
The positivity of the initial values in assumption (A2) is not required for the local existence.

\subsection{Global solutions} \label{sec:global}
As a next step, we now show that the norm of the solution does not blow up in finite time and, therefore, the solution exists globally.
\begin{theorem} \label{thm:global}
Let (A1)--(A2) hold.
Then the local weak solution $(u,c)$ of \eqref{eq:sys1}--\eqref{eq:sys2} with initial values $u(0)=u_0$ and $c(0)=c_0$ 
satisfies $u(t)\ge 0$, $c(t)\ge 0$ for $0 \le t \le T$, and
\begin{align*}
\|u\|_{L^\infty(0,t;L^2(\E))} + \|c\|_{L^\infty(0,t;H^1(\E))} \le P(t).
\end{align*}
Here $P(t)$ is a polynomial in $t$ with coefficients that depend only on the 
bounds for the parameters in (A1), on the geometry of the graph, and on $\|u_0\|_{L^2(\E)}$ and $\|c_0\|_{H^1(\E)}$.
\end{theorem}
Similarly as in \cite{HillenPotapov04,OsakiYagi01}, our proof is based on conservation and positivity preservation 
of solutions, which we state explicitly as a preparatory result.
\begin{lemma} \label{lem:cons_and_pos}
Let (A1)--(A2) hold and let $(u,c)$ be the local weak solution guaranteed by Theorem~\ref{thm:local}.
Then 
\begin{align*}
\int_\E u(t) dx = \int_\E u_0 dx, \qquad 0 \le t \le T.
\end{align*}
Moreover, the solution is positive, i.e., $u(t) \ge 0$ and $c(t) \ge 0$ on $\E$ for a.a. $0 \le t \le T$.
\end{lemma}
\begin{proof}
The first assertion follows by testing \eqref{eq:var1} with $v \equiv 1$.  
Now let $u^- = \min(u,0)$ denote the negative part of $u$. 
Testing equation~\eqref{eq:var1} with $v = u^-$, we conclude that
\begin{align*}
\frac{1}{2} \frac{d}{dt} \|u^-(t)\|_{L^2(\E)}^2 + \underline \alpha \|\dx u^-(t)\|_{L^2(\E)}^2 
&\le \overline\chi \|\dx c(t)\|_{L^2(\E)} \|u^-(t)\|_{L^\infty(\E)} \|\dx u^-(t)\|_{L^2(\E)} = (*).
\end{align*}
From Theorem~\ref{thm:local}, we deduce that $\overline \chi \|\dx c(t)\|_{L^2(\E)} \le C$ for a.a. $0 \le t \le T$. 
Using Lemma~\ref{lem:interpolation}, we thus obtain 
\begin{align*}
(*) &\le C \left(\eps \|\dx u^-(t)\|^2_{L^2(\E)} + C_\eps \|u^-(t)\|_{L^2(\E)}^2\right),
\end{align*}
where $C_\eps$ only depends on $\eps$ and $C_G$. Choosing $\eps=\underline\alpha/(2C)$ allows to absorb the 
first term in the left hand side of the energy estimate, and by Lemma~\ref{lem:gronwall}, we deduce that
\begin{align*} 
\|u^-(t)\|_{L^2(\E)}^2 
\le e^{2 C' t} \|u^-(0)\|_{L^2(\E)}^2 = 0;
\end{align*}
in the last identity we used that $u(0)=u_0 \ge 0$. 
This shows that $u(t) \ge 0$ on its domain of existence. 
The non-negativity of $c$ can then be derived with similar arguments.
\end{proof}

\subsection*{Proof of Theorem~\ref{thm:global}}

The positivity of the solution is already guaranteed by Lemma~\ref{lem:cons_and_pos}.
The remaining proof of the a-priori estimate then proceeds in several steps.

\smallskip

\noindent {\em Step~1:} 
As a direct consequence of Lemma~\ref{lem:cons_and_pos}, we obtain that 
\begin{align} \label{eq:aux1}
\|u(t)\|_{L^1(\E)} = \|u_0\|_{L^1(\E)} 
=: M
\end{align}
for all $0 \le t \le T$. Using this identity, we further deduce from \eqref{eq:var2}, 
by testing with $q=1$ and integration over time, that
\begin{align} \label{eq:aux2}
\|c(t)\|_{L^1(\E)} 
= \|c_0\|_{L^1(\E)} + t \overline \delta M = :P_1(t).    
\end{align}
Note that $P_1(t)$ is a polynomial in $t$ whose coefficients depend continuously on the data.

\smallskip

\noindent {\em Step~2:}
Testing \eqref{eq:var2} with $q=c(t)$ yields 
\begin{align*}
\frac{1}{2} \frac{d}{dt} &\|c(t)\|_{L^2(\E)}^2 + \underline \beta \|\dx c(t)\|_{L^2(\E)}^2 
 \le \overline \delta \|u(t)\|_{L^1(\E)} \|c(t)\|_{L^\infty(\E)} \\
&\le \overline \delta M C_G(\|c(t)\|_{L^1(\E)}^{1/3} \|\dx c(t)\|_{L^2(\E)}^{2/3} + \|c(t)\|_{L^1(\E)})
 \le P_2(t) + \frac{\underline \beta}{2} \|\dx c(t)\|_{L^2(\E)}^2.
\end{align*}
Here we used \eqref{eq:aux1} and Lemma~\ref{lem:interpolation} for the second estimate, and employed
\eqref{eq:aux2} and Young's inequality for the third estimate.
Note that $P_2(t)$ is again a polynomial of $t$ with coefficients depending continuously on the problem data.
By some elementary computations and integration with respect to time, we further obtain 
\begin{equation}\label{eq:aux3}
 \|c(t)\|_{L^2(\E)}^2 + \int_0^t \|\dx c(s)\|_{L^2(\E)}^2 ds \le P_3(t),
\end{equation}
with polynomial $P_3(t)$ depending only on $P_2(t)$ and $\underline \beta$.

\smallskip

\noindent {\em Step~3:} 
Testing equation~\eqref{eq:var2} with $q=\dt c(t)$ yields the estimate
\begin{align*}
\|\dt c(t)\|_{L^2(\E)}^2 &+ \frac{1}{2} \frac{d}{dt} \|\beta^{1/2} \dx c(t)\|_{L^2(\E)}^2 
 \le \overline \delta \|u(t)\|_{L^2(\E)} \|\dt c(t)\|_{L^2(\E)} \\
&\le \overline \delta C_G (M^{2/3} \|\dx u(t)\|^{1/3}_{L^2(\E)} + M) \|\dt c(t)\|_{L^2(\E)} \\
&\le  C_2 + C_3 \|\dx u(t)\|_{L^2(\E)}^{2/3} +\frac{1}{2} \|\dt c(t)\|_{L^2(\E)}^2,
\end{align*}
with $C_2=C_2(C_G,M,\overline \delta)$ and $C_3=C_3(C_G,M, \overline \delta)$. Here we used Lemma~\ref{lem:interpolation} 
and \eqref{eq:aux1} in the second estimate, and Young's inequality for the third. 
By integration in time, we get
\begin{equation}\label{eq:aux4}
\|\dx c(t)\|_{L^2(\E)}^2 + \int_0^t \|\dt c(s)\|_{L^2(\E)}^2 ds
\le P_4(t) + C_4 \int_0^t \|\dx u(s)\|^{2/3} ds,
\end{equation}
with constant $C_4=C_4(C_G,M,\underline \beta, \overline \beta, \overline \delta)$ and polynomial $P_4(t)$ whose coefficients 
again depend continuously on the data. 
By squaring the previous estimate, we further get 
\begin{equation}\label{eq:aux5}
\|\dx c(t)\|_{L^2(\E)}^4 
\le P_5(t) + C_5 t^{4/3} \left(\int_0^t \|\dx u(s)\|_{L^2(\E)}^2 ds\right)^{2/3}.
\end{equation}
with $P_5(t)=2 P_4^2(t)$ and $C_5=2 C_4^2$. 
In the derivation of this estimate, we used H\"older's inequality to bound 
the term $\int_0^t \|u(s)\|^{2/3} ds \le t^{2/3} (\int_0^t \|u(s)\|^2 ds)^{1/3}$ from above.

\smallskip

\noindent {\em Step~4:}
Testing equation~\eqref{eq:var1} with $v=u(t)$ leads to
\begin{align*}
\frac{1}{2} \frac{d}{dt} \|u(t)\|_{L^2(\E)}^2 &+ \underline \alpha \|\dx u(t)\|_{L^2(\E)}^2 
 \le \overline\chi \|u(t)\|_{L^\infty(\E)} \|\dx c(t)\|_{L^2(\E)} \|\dx u(t)\|_{L^2(\E)}\\
&\le \overline\chi C_G (M^{1/3} \|\dx u(t)\|_{L^2(\E)}^{2/3} + M) \|\dx c(t)\|_{L^2(\E)} \|\dx u(t)\|_{L^2(\E)} \\
&\le  C_5 \left(\|\dx c(t)\|_{L^2(\E)}^6 + \|\dx c(t)\|_{L^2(\E)}^2 \right)+ \frac{\underline\alpha}{2}\|\dx u(t)\|_{L^2(\E)}^2,
\end{align*}
with constant $C_5=C_5(\underline\alpha,\overline\chi, C_G, M)$.
From \eqref{eq:aux4} and \eqref{eq:aux5}, we can deduce that
\begin{align*}
\int_0^t &\|\dx c(s)\|_{L^2(\E)}^6 ds
 \le \int_0^t \left(P_5(s) + C_5 s^{4/3} \left(\int_0^s \|\dx u(r)\|_{L^2(\E)}^2 dr\right)^{2/3}\right) \|\dx c(s)\|_{L^2(\E)}^2 ds \\
&\le P_5(t)\int_0^t \|\dx c(s)\|_{L^2(\E)}^2 ds + C_5 \left(\int_0^t \|\dx u(s)\|_{L^2(\E)}^2 ds\right)^{2/3} \left(\int_0^t s^{4/3} \|\dx c(s)\|_{L^2(\E)}^2 ds\right).
\end{align*}
Together with the estimate \eqref{eq:aux3} and using Young's inequality, one can then see that
\begin{align*}
\int_0^t \|\dx c(s)\|_{L^2(\E)}^6 ds
&\le \underline\beta^{-1} P_5(t) P_3(t)  + C_5 \left(\int_0^t \|\dx u(s)\|_{L^2(\E)}^2ds\right)^{2/3} 
  \left(t^{4/3} \underline\beta^{-1} P_{3}(t)\right) \\
&\le P_6(t) + \frac{\underline\alpha}{4 C_5} \int_0^t \|\dx u(s)\|^2 ds.
\end{align*}
The coefficients of the polynomial $P_6(t)$ again depend continuously on the problem data. 

\smallskip

\noindent {\em Step~5:}
Inserting the last expression in the first estimate of Step~4, slightly rearranging the terms, 
and integrating with respect to time now yields
\begin{align} \label{eq:aux6}
\|u(t)\|_{L^2(\E)}^2 + \underline\alpha \int_0^t \|\dx u(s)\|_{L^2(\E)}^2 ds &\leq P_7(t)
\end{align}
with polynomial $P_7(t)$ whose coefficients depend continuously on the data. 
This is the required estimate for $u$.
A combination with \eqref{eq:aux3} and \eqref{eq:aux4} yields the bounds for $c$. 
\hfill \qed

\subsection{Higher regularity}

We now state some bounds for the solution in stronger norms which are obtained under the assumption that 
the initial values have higher regularity and satisfy the usual compatibility conditions. 
We thus assume in the following that
\begin{itemize}
 \item[(A3)]  $\alpha \dx u_0 + \chi \dx c_0 u_0  \in H_0(\div;\E)$ 
and $\dx(\beta \dx c_0)-\gamma c_0 + \delta u_0 \in H^1(\E)$, 
\end{itemize}
where $H_0(\div;\E) = \big\{w \in H^1_{pw}(\E) : \sum_{e \in \E(v)} n_e(v) w_e(v) = 0 \ \forall v \in \V \big\}$
is the space of regular fluxes and $H_{pw}^k(\E) = \{ w \in L^2(\E) : w|_e \in H^k(e) \ \forall e \in \E\}$ the space 
of piecewise smooth functions with appropriate regularity on the individual edges $e \in \E$. 
Under these assumptions, one can show that the solution $(u,c)$ in fact enjoys higher regularity. 
\begin{theorem}\label{thm:reg}
Let (A1)--(A3) hold. Then the solution $(u,c)$ of Theorem~\ref{thm:global} satisfies 
\begin{align} 
\|u\|_{L^{2}(0,T;H_{pw}^{2}(\E))} + \|\dt u\|_{L^2(0,T;H^1(\E))} + \|\dtt u\|_{L^2(0,T;H^1(\E)')}^2 &\le C(T), \label{eq:reg1}\\ 
\|c\|_{L^{\infty}(0,T;H_{pw}^2(\E))} + \|\dt c\|_{L^\infty(0,T;H^1(\E))} + \|\dtt c\|_{L^2(0,T;L^2(\E))} &\le C(T), \label{eq:reg2}
\end{align}
\end{theorem}
A complete proof is presented in the appendix. 
Let us note that this is exactly the additional regularity that can be expected, see \cite[Sec~7.1.3]{Evans10}, and 
which allows us to establish order optimal convergence rates for the numerical approximation in Section~\ref{sec:rates}.

\section{Discretization} \label{sec:discretization}

We now turn to the systematic discretization of the chemotaxis problem on networks 
by a finite element method in space and an implicit time stepping scheme. 

\subsection{Notation}

Let $[0,\ell_e]$ be the interval represented by the edge $e$ 
and let $T_h(e) = \{T\}$ be a uniform mesh of $e$ with subintervals $T$ of length $h_T=h_e$. 
The global mesh is then defined by $T_h(\E) = \{T_h(e) : e \in \E\}$ and $h=\max_e h_e$ is the global mesh size. 
Furthermore, let $x_j$, $j=1,\ldots,N$ be the vertices of the mesh $T_h(\E)$.
We denote by 
\begin{align*}
 P_k(T_h(\E)) &= \{v \in L^2(\E) : v|_e \in P_k(T_h(e)), \ e \in \E\}, 
\end{align*}
and $P_k(T_h(e)) = \{v \in L^2(e) : v|_T \in P_k(T), \ T \in T_h(e)\}$ the spaces of piecewise polynomials on $T_h(\E)$ and $T_h(e)$, respectively, and by $P_k(T)$ the space of polynomials of degree $\le k$ on the element $T$. 
Note that $P_k(T_h(e)) \subset L^2(e)$, but in general $P_k(T_h(e)) \not\subset H^1(e)$. 

For the approximation of the population density $u$ and the concentration $c$ in space,
we consider the finite element space
\begin{align} \label{eq:space}
V_h = P_{1}(T_h(\E)) \cap H^1(\E)
\end{align}
of continuous and piecewise linear functions over the mesh $T_h(\E)$. 
We further denote by $\pi_h : L^2(\E) \to V_h$ the standard $L^2$-orthogonal projection onto $V_h$, defined by 
\begin{align*} 
\la \pi_h v,v_h\ra_E = \la v,v_h\ra_\E \qquad \forall v_h \in V_h.
\end{align*}
Recall that $\pi_h$ is uniformly bounded on $L^2(\E)$ and on $H^1(\E)$. 
Moreover, 
\begin{align} \label{eq:l2-proj}
\|\pi_h v - v\|_{H^s(\E)} \le C h^{k-s} \|v\|_{H_{pw}^k(\E)} 
\end{align}
for all $v \in H_{pw}^k(\E) \cap H^s(\E)$ with $0 \le s \le 1$ and $0 \le k \le 2$.  
These estimates are readily proven by the standard approximation error estimates for the $L^2$-projection, 
inverse inequalities, and interpolation arguments; see \cite{BrennerScott08} for details.

Let us now turn to the time discretization.
We choose a time step size $\tau>0$ and set $t^n=n\tau$ for $n \ge 0$. 
For any given sequence $\{a_n\}_{n \ge 0}$, we denote by 
\begin{align} \label{eq:diffquot}
 \dtau a^n = \frac{1}{\tau}(a^n - a^{n-1}) 
\end{align}
the backward difference quotient with respect to the given time discretization.
Let us note that for any scalar product $\la\cdot,\cdot\ra$ with associated norm $\|\cdot\|$, there holds
\begin{align*} 
\frac{1}{2} \dtau \|a^n\|^2 
= \frac{1}{2 \tau}\|a^n\|^2 - \frac{1}{2 \tau}\|a^{n-1}\|^2 
= \la\dtau a^n, a^n\ra - \frac{\tau}{2} \|\dtau a^n\|^2,
\end{align*}
which can be verified by some basic calculations. As a direct consequence, one has 
\begin{align} \label{eq:identity}
\frac{1}{2 \tau}\|a^n\|^2 
\le \frac{1}{2 \tau}\|a^{n-1}\|^2 + \la \dtau a^n, a^n\ra, 
\end{align}
which will be frequently used to derive discrete energy estimates in our analysis below.

\subsection{Auxiliary results}

In order to guarantee the positivity of solutions also on the discrete level, 
some modifications of the standard Galerkin approach will be required. 
In the following, we introduce the main ingredients and present some basic results.

\subsection*{Quasi-interpolation}

For the approximation of the initial values, we will use a
quasi-interpolation operator $\widetilde \pi_h : L^2(\E) \to V_h$, which is defined by 
\begin{align} \label{eq:clement}
(\widetilde \pi_h v)(x_i) = \frac{1}{|\omega_{x_i}|} \int_{\omega_{x_i}}  v(x) dx,
\end{align}
where $\omega_{x_i} = \bigcup T\in T_h(\E) : x_i \in T$ is the element patch around the vertex $x_i$ of the mesh $T_h(\E)$.
Let us recall that the operator $\widetilde \pi_h : L^2(\E) \to V_h \subset L^2(\E)$ is linear and continuous with 
\begin{align} \label{eq:clement1}
\|\widetilde \pi_h v  - v\|_{H^s(\E)} \le C h^{k-s} \|v\|_{H_{pw}^k(\E)}
\end{align}
for all $v \in H_{pw}^k(\E) \cap H^s(\E)$ with $0 \le s \le k \le 1$; see \cite{BrennerScott08,Clement75} for details.  
Moreover, the operator is positivity-preserving, i.e., 
\begin{align} \label{eq:clement2}
\widetilde \pi_h v \ge 0 \qquad \text{if} \qquad v \ge 0,
\end{align}
which follows directly from the construction \eqref{eq:clement} via local averaging.

\subsection*{Mass-lumping}
It is well-known that some sort of mass lumping is required to ensure a discrete maximum 
principle for the finite element approximation of parabolic problems, see e.g. \cite{Thomee06}. 
To this end, we define for $u,v \in H^1_{pw}(\E)$ the \emph{lumped} scalar product 
\begin{align*}
\la u,v\ra_{h,\E} = \sum_{T \in T_h} \int_T I_T(u v) dx,  
\end{align*}
where $I_T(w) \in P^1(T)$ denotes the linear interpolation of the function $w \in H^1(T)$. 
This corresponds to using numerical integration on every element $T \in T_h$ by the trapezoidal rule. 
One can verify by simple calculations that the induced norm 
\begin{align*}
	\|v_h\|_h = \la v_h,v_h \ra_{h,E}^{1/2}, \quad v_h \in V_h,
\end{align*}
is equivalent to the standard $L^2$-norm on the finite element space $V_h$, i.e., 
\begin{align}\label{eq:equiv}
\frac{1}{\sqrt 3} \|v_h\|_{h} \le \|v_h\|_{L^2(\E)} \le \|v_h\|_h, \quad \forall v_h\in V_h.
\end{align}

\subsection*{Upwinding}

As a final ingredient for our discretization scheme, we now describe the upwind technique for the convective term.
Let $\chi \dx c_h^{n-1} \in P_0(T_h) \subseteq L^2(\E) $ be given. 
Then for any $u \in H^1(\E)$, we define $\widehat \pi_h^{n-1} u \in P_0(T_h)$ 
on every element $T = [x_1,x_2]$ by 
\begin{align}\label{eq:upw}
\widehat \pi_h^{n-1} u |_{T} = 
\begin{cases} 
u(x_1), & \text{if } \chi \dx c_h^{n-1} \ge 0, \\
u(x_2), & \text{else}.
\end{cases}
\end{align}
Note that for any $n \ge 1$, the operator $\widehat \pi_h^{n-1} : H^1(\E) \to P_0(T_h)$ is linear and bounded, i.e., 
\begin{equation}\label{eq:upwBounded}
\| \widehat \pi_h^{n-1} u \|_{L^{\infty}(\E)} \leq \|u\|_{L^{\infty}(\E)} \leq C \|u_h\|_{H^{1/2+\eps}(\E)}.
\end{equation}
Moreover, the piecewise constant function $\widehat \pi_h^{n-1} u$ interpolates $u$ at the vertices $x_j$ 
and, therefore, standard Taylor estimates yield 
\begin{align} \label{eq:upwind2}
\| \widehat \pi_h^{n-1} u - u\|_{L^\infty(\E)} \le C h^{1-1/p} \|\dx u\|_{L^p(\E)}, 
\end{align}
for all $u \in H^1(\E)$ with a uniform constant $C$ independent of the mesh.

\subsection{Definition of the discretization scheme}

We are now in the position to formulate our numerical approximation scheme 
for the weak formulation of problem \eqref{eq:sys1}--\eqref{eq:sys6}.
\begin{problem} \label{prob:disc}
Set $u_h^0=\widetilde \pi_h u_0$, $c_h^0= \widetilde \pi_h c_0$. 
Then for $n \ge 1$, find $(u_h^n,c_h^n) \in V_h \times V_h$ 
with
\begin{alignat}{2}
&\la\dtau u_h^n,v_h\ra_{h,\E} + \la\alpha \dx u_h^n, \dx v_h\ra_\E = \la \chi \widehat \pi_h^{n-1} u_h^n \dx c_h^{n-1}, \dx v_h\ra_\E, \qquad && \forall v_h \in V_h, \label{eq:disc1}\\
&\la \dtau c_h^n,q_h\ra_{h,\E} + \la \beta \dx c_h^n,\dx q_h\ra_\E + \la \gamma c_h^n,q_h\ra_{h,\E} = \la \delta u_h^n,q_h\ra_{h,\E}, && \forall q_h \in V_h. \label{eq:disc2}
\end{alignat}
\end{problem}
We will show below that the discrete solution is well-defined, that the scheme is positivity preserving,
and that the total population is conserved for all time.
For our analysis, we will need some auxiliary results which we state next.

\subsection{Algebraic properties}

Let us start with summarizing some algebraic properties of the discretization scheme 
\eqref{eq:disc1}--\eqref{eq:disc2}.
We denote by $\{\phi_i : i=1,\ldots,N\}$ the nodal basis of the finite element space $V_h$ defined by 
\begin{align} \label{eq:nodal}
\phi_i \in V_h : \phi_i(x_j) = \delta_{i,j}.
\end{align}
Now let $\underline v=[v_1,\ldots,v_N]$ be the coordinate vector of the function $v_h = \sum_{i=1}^N v_i \phi_i \in V_h$ 
with respect to this basis given by $v_i = v_h(x_i)$.
From the particular form of the basis functions $\phi_i$, one can directly deduce that 
\begin{align} \label{eq:pos}
v_h \ge 0 \quad \Leftrightarrow \quad \underline v \ge 0. 
\end{align}
Moreover, the discretized system \eqref{eq:disc1}--\eqref{eq:disc2} can be written in algebraic form as 
\begin{align*}
M(1) \dtau \underline u^n + K(\alpha) \underline u^n &= C^{n-1} \underline u^n, \\
M(1) \dtau \underline c^n + K(\beta)  \underline c^n + M(\gamma) \underline c^n &= M(\delta) \underline u^n,
\end{align*}
with system matrices defined by 
\begin{align*}
M(\eta)_{ij} = \la \eta \phi_j, \phi_i\ra_{h,\E},
\qquad 
K(\eta)_{ij} = \la \eta \dx \phi_j, \dx \phi_i\ra_{\E} 
\end{align*}
and 
\begin{align*}
C^{n-1}_{ij} = \la \chi \widehat \pi_h^{n-1} \phi_j \dx c_h^{n-1}, \dx \phi_i\ra_{\E}. 
\end{align*}
From the definition of the matrices, one can immediately deduce the following properties. 
\begin{lemma} \label{lem:algebraic}
(a) For any $\eta \in L^\infty(\E)$, the matrix $M=M(\eta)$ is diagonal. 
If $\eta>0$, then $M_{ii}>0$ for all $i=1,\ldots,N$.
(b) For $\eta \in L^\infty(\E)$ with $\eta > 0$, the matrix $K=K(\eta)$ satisfies 
$K_{ii} > 0$, $K_{ij} \le 0$ for $j \ne i$, and $\sum_{j=1}^N K_{ij} = \sum_{i=1}^N K_{ij} = 0$. 
(c) Let $\chi \dx c_h^{n-1} \in P_0(T_h)$. Then the matrix $C=C^{n-1}$ satisfies 
$C_{ii} \le 0$, $C_{ij} \ge 0$ for all $j \neq i$, and $\sum_{j=1}^N C_{ij} = 0$. 
\end{lemma}
\begin{proof}
Note that the entries of all system matrices are defined by integrals over the network $\E$ 
which can be split into integrals over individual elements $T \in T_h(\E)$ respectively.

\smallskip 

\noindent (a) 
By definition of the matrix $M$, one has $M_{ij} = \sum_{T \in T_h} M_{ij}^{(T)}$
with element contributions given by $M_{ij}^{(T)} = \int_T I_T (\eta \phi_j \phi_i) dx = |T| \sum_{s\in \{k,l\}} \eta(x_s) \phi_j(x_s) \phi_i(x_s)$ for $T=[x_k,x_l]$. The properties in (a) then follow directly from \eqref{eq:nodal}. 

\smallskip 

\noindent (b) 
In a similar manner, we can decompose $K=K(\eta)$ as $K=\sum_{T \in T_h} K^{(T)}$ 
with the non-zero entries of $K^{(T)}$ for the element $T=[x_k,x_l]$ given by 
\begin{align*}
K_{ij}^{(T)} = 
\begin{cases} 
\frac{\int_T \eta dx}{|T|^2}, & i=j \in \{k,l\}, \\
-\frac{\int_T \eta dx}{|T|^2}, & i \ne j \in \{k,l\},
\end{cases}
\end{align*}
and $K_{i,j}^{(T)}=0$ else. The properties in (b) then follow by summation over all elements.

\smallskip 

\noindent (c) 
The matrix $C^{n-1} = \sum_{T \in T_h} C^{(T)}$ can again be split into element contributions.
Now let $a_T=\chi \dx c_h^{n-1}|_T$ for $T=[x_k,x_l]$. Then the non-zero entries of $C_{ij}^{(T)}$ 
are given by 
\begin{align*}
[C_{ij}^{(T)}]_{ij} = 
\min(a_T,0) \begin{pmatrix}
\frac{1}{2} & 0 \\
-\frac{1}{2} & 0
\end{pmatrix} 
+ \max(a_T,0)\begin{pmatrix}
0 & \frac{1}{2} \\
0 & -\frac{1}{2}
\end{pmatrix} 
\end{align*}
for $i,j \in \{k,l\}$ and $C_{ij}^{(T)}=0$ else. The properties in assertion (c) then follow 
directly from these observations by summation over all elements.
\end{proof}

\subsection{Well-posedness, conservation, and positivity}

As a direct consequence of the algebraic properties stated in the previous lemma, we obtain 
the following result.
\begin{lemma} \label{lem:disc_wellposed}
Let (A1)--(A2) hold. 
Then Problem~\ref{prob:disc} has a unique solution $(u_h^n,c_h^n)_{n \ge 0}$. 
Moreover, $\int_\E u_h^n dx = \int_\E u_h^0 dx$ and $u_h^n \ge 0$, $c_h^n \ge 0$ for all $n \ge 0$.
\end{lemma}
\begin{proof}
Let $\underline u^{n}$, $\underline c^n$ denote the coordinate vectors for the functions $u_h^n$ and $c_h^n$. 
Using (A2), \eqref{eq:clement2}, \eqref{eq:pos}, and the definition of the initial values, 
one can see that $\underline u^0 \ge 0$ and $\underline c^0 \ge 0$.
Moreover, the algebraic system defining the solution at iteration $n$ can be rewritten as 
\begin{align}
\left[\tfrac{1}{\tau} M(1) + K(\alpha) - C^{n-1}\right] \underline u^n &= \tfrac{1}{\tau} M(1) \underline u^{n-1}, \label{eq:linsys1}\\
\left[\tfrac{1}{\tau} M(1) + K(\beta) + M(\gamma)\right] \underline c^n &= \tfrac{1}{\tau} M(1) \underline c^{n-1} + M(\delta) \underline u^{n}. \label{eq:linsys2}
\end{align}
From the algebraic properties stated in Lemma~\ref{lem:algebraic} and the assumptions (A1) on the model parameters,
one can deduce that the system matrices $S_u = \frac{1}{\tau} M(1) + K(\alpha) - C^{n-1}$ and $S_c = \frac{1}{\tau} M(1) + K(\beta) + M(\gamma)$ are strictly diagonally dominant. 
This shows that the two linear systems \eqref{eq:linsys1} and \eqref{eq:linsys2} can be solved uniquely.
Existence of a unique discrete solution $(u_h^n,c_h^n)$ for all $n \ge 0$ then follows by induction. 
From Lemma~\ref{lem:algebraic}, one can further deduce that $M(1)$ and $M(\delta)$ have positive entries 
and that the system matrices $S_u$ and $S_c$ are M-matrices; see \cite{Plemmons77} for details.
By the inverse positivity of M-matrices one can thus infer that $\underline u^n \ge 0$ and $\underline c^n \ge 0$,
if $\underline u^{n-1} \ge 0$ and $\underline c^{n-1} \ge 0$.
Positivity of the discrete solution $(u_h^n,c_h^n)$ for all $n \ge 0$ then follows by induction and 
noting that $\underline u^n \ge 0$, $\underline c^n \ge 0$; see above.
The conservation property finally follows by testing \eqref{eq:disc1} with $v_h=1$. 
\end{proof}

\begin{remark} \rm 
A related algebraic argument was used by Saito in \cite{Saito12} to establish positivity of the discrete solution. 
In that work, however, a strong smallness condition on the time step size $\tau$ was required to ensure that the system 
matrices are row-wise diagonally dominant. Here we use the fact that the matrices $S_u$ and $S_c$ are column-wise diagonally dominant 
for any $\tau>0$ which, together with the sign conditions on the entries, is sufficient to prove the M-matrix property. 
This allows us to avoid the strong smallness condition on $\tau$. The same argument should allow to improve the analysis for the method of \cite{Saito12}.
\end{remark}

\subsection{Uniform a-priori bounds}

With similar reasoning as on the continuous level, we are now able to derive uniform bounds 
also for the discrete solutions.
\begin{theorem} \label{thm:disc}
Let (A1)--(A2) hold and $(u_h^n,c_h^n)_{n \ge 0}$ be the solution of Problem~\ref{prob:disc}. 
Then 
\begin{align*}
\max_{k \le n}\|u_h^k\|_{L^2(\E)}^2 + \sum_{k=1}^n \tau \|\dx u_h^k\|^2_{L^2(\E)} 
\le P(t^n), \\
\max_{k \le n}\|\dx c_h^k\|_{L^2(\E)}^2 + \sum_{k=1}^n \tau \|\dtau c_h^k\|^2_{L^2(\E)} \le P(t^n),
\end{align*}
with a polynomial $P(t)$ whose coefficients only depend on the problem data.
\end{theorem}
\begin{proof}
The proof of Theorem~\ref{thm:global} applies almost verbatim. In particular, the constants and polynomial functions appearing in the proof are of the same form as those appearing in Theorem~\ref{thm:global}.
For convenience of the reader, we repeat the main steps.
All estimates will hold uniformly for all time steps $t^n \le T$.

\smallskip

\noindent {\em Step~1:}
As a direct consequence of Lemma~\ref{lem:disc_wellposed}, we obtain 
\begin{align*}
\|u_h^n\|_{L^1(\E)} = \|u_h^0\|_{L^1(\E)} 
=: M_h. 
\end{align*}
Without loss of generality we may set $M_h=M$ by altering either of the two constants in the respective proofs.
Testing \eqref{eq:disc2} with $q_h=1$ and using the positivity of $c_h^n$, we get
\begin{align*}
\frac{1}{\tau}\|c_h^n\|_{L^1(\E)} + \underline\gamma \|c_h^n\|_{L^1(\E)} \le \frac{1}{\tau}\|c_h^{n-1}\|_{L^1(\E)} + \overline\delta \|u_h^n\|_{L^1(\E)}. 
\end{align*}
From the uniform bound on $\|u_h^n\|_{L^1(\E)}$ and induction, we thus deduce that 
\begin{align*}
\|c_h^n\|_{L^1(\E)} \le  \|c_h^0\|_{L^1(\E)} + t^n \overline \delta M_h =: P_1(t^n).
\end{align*}
Note that the polynomial $P_1(t)$ has the same form as that in the proof of Theorem~\ref{thm:global}.

\smallskip

\noindent {\em Step~2:}
Testing \eqref{eq:disc2} with $q_h=c_h^n$ and using \eqref{eq:identity} leads to 
\begin{align*}
\frac{1}{2} \dtau \|c_h^n\|_{L^2(\E)}^2 + \underline \beta \|\dx c_h^n\|^2_{L^2(\E)} 
\le \overline\delta \|u_h^n\|_{L^1(\E)} \|c_h^n\|_{L^\infty(\E)} 
\le P_2(t^n) + \tfrac{\underline\beta}{2} \|\dx c_h^n\|^2_{L^2(\E)},
\end{align*}
with the same polynomial $P_2(t)$ as in the proof of Theorem~\ref{thm:global}.
Summation over $n$ yields
\begin{align*}
\|c_h^n\|^2_{L^2(\E)} + \underline \beta \sum_{k=1}^n \tau \|\dx c_h^k\|_{L^2(\E)}^2 
\le P_3(t^n). 
\end{align*}

\smallskip

\noindent {\em Step~3:}
By testing equation \eqref{eq:disc2} with $q_h=\dtau c_h^n$ and proceeding with the same arguments 
as in the proof of Theorem~\ref{thm:global}, one arrives at 
\begin{align*}
\|\dx c_h^n\|_{L^2(\E)}^2  + \sum_{k=1}^n \tau \|\dtau c_h^n\|_{L^2(\E)}^2
\le P_4(t^n) + C_4 \sum_{k=1}^n \tau \|\dx u_h^k\|_{L^2(\E)}^{2/3}.
\end{align*}
Squaring this estimate and some elementary estimates further lead to
\begin{align*}
\|\dx c_h^n\|_{L^2(\E)}^4 \le P_5(t^n) + C_5  (t^n)^{4/3}\big( \sum_{k=1}^n \tau \|\dx u_h^k\|_{L^2(\E)}^2 \big)^{2/3}.
\end{align*}
The constant $C_5$ and the polynomial $P_5$ are again of the same form as those in Theorem~\ref{thm:global}.

\smallskip

\noindent {\em Step~4:}
We now turn to the estimates for $u_h^n$. 
By testing \eqref{eq:disc1} with $v_h=u_h^n$, using \eqref{eq:identity}, 
and proceeding in the same manner as in the proof of Theorem~\ref{thm:global}, we get 
\begin{align*}
\frac{1}{2} \dtau \|u_h^n\|_{L^2(\E)}^2 + \underline\alpha \|\dx u_h^n\|_{L^2(\E)}^2 
\le  C_5 \left(\|\dx c_h^{n-1}\|_{L^2(\E)}^6 + \|\dx c_h^{n-1}\|_{L^2(\E)}^2 \right) + \frac{\underline\alpha}{2} \|\dx u_h^n\|^2_{L^2(\E)}.
\end{align*}
From the previous estimates for $c_h^n$, we may then further deduce that 
\begin{align*}
\sum_{k=1}^n \tau \|\dx c_h^{k-1}\|_{L^2(\E)}^6 \le P_6(t^n) + \frac{\underline\alpha}{4 C_5} \sum_{k=1}^n \tau \|\dx u_h^n\|_{L^2(\E)}^2,
\end{align*}
and this bound holds with the same polynomial $P_6(t)$ as in the proof of Theorem~\ref{thm:global}.

\smallskip

\noindent {\em Step~5:}
A combination of the estimates in Step~4 now leads to 
\begin{align*}
\|u_h^n\|_{L^2(\E)}^2 + \underline \alpha \sum_{k=1}^n \tau \|\dx u_h^k\|_{L^2(\E)}^2 \le P_7(t^n),
\end{align*}
which yields the desired bound for $u_h^n$. 
The corresponding estimates for $c_h^n$ then follow by inserting 
this bound into the estimates of Step~2 and Step~3. 
\end{proof}

Let us emphasize that all arguments used for the analysis of the problem on the continuous level 
carry over to the discrete setting almost verbatim. The reason for this is that by its variational character,
the proposed method with mass lumping, upwinding, and implicit time integration inherits all important structures of 
the continuous problem.

\section{Convergence} \label{sec:convergence}

Based on the uniform bounds of the discrete solution provided by Theorem~\ref{thm:disc}, we now establish the convergence of the discretization scheme towards the unique solution of the continuous problem without additional regularity assumptions. 
For a convenient presentation of our results, we will interpret the discrete functions $(w_h^n)_{n \ge 0}$ as 
(piecewise) continuous functions of time. Such extensions will be denoted with double subscripts $w_{h,\tau}$.
\begin{theorem}
Let (A1)--(A2) hold and let $(u,c)$ and $(u_h^n,c_h^n)_{n \ge 0}$ denote the weak solutions of problem \eqref{eq:sys1}--\eqref{eq:sys6} and the discrete solution of Problem~\ref{prob:disc}, respectively. 
Furthermore, let $(u_{h,\tau},c_{h,\tau})$ be the piecewise linear interpolation of $(u_h^n,c_h^n)_{n \ge 0}$ in time. Then
\begin{alignat*}{2}
\|u - u_{h,\tau}\|_{L^2(0,T;L^2(\E))} + \|c - c_{h,\tau}\|_{L^2(0,T;L^2(\E))} \to 0, \qquad h,\tau \to 0.
\end{alignat*}
\end{theorem}
\begin{proof}
The proof is based on standard arguments for the analysis of nonlinear parabolic problems, see e.g. \cite{Roubicek13}. 
We therefore only sketch the main arguments.

\smallskip 

\noindent 
{\em Step~1a:}
Let $(\overline u_{h,\tau},\overline c_{h,\tau})$ be piecewise constant in time with values given by $\overline u_{h,\tau}(t) = u_h^n$ and $\overline c_{h,\tau}(t) = c_h^n$ for $t^{n-1} < t \le t^n$. 
Furthermore, let $\widehat c_{h,\tau}$ and $\widehat \pi_{h,\tau}$ be piecewise constant in time with values $\widehat c_{h,\tau}(t) = c_h^{n-1}$ and $\widehat \pi_{h,\tau}(t) = \pi_h^{n-1}$ for $t^{n-1} < t \le t^n$. Then from the variational characterization \eqref{eq:disc1}--\eqref{eq:disc2}, one can deduce that
\begin{alignat*}{3}
&\la \dt  u_{h,\tau}(t), v_h\ra_{h,\E} + \la \alpha \dx \overline u_{h,\tau}(t), \dx v_h\ra_{\E} 
   = \la \chi \widehat\pi_{h,\tau}(t) \overline u_{h,\tau}(t) \dx \widehat c_h(t), \dx v_h\ra_\E,\\
&\la \dt  c_{h,\tau}(t),q_h\ra_{h,\E} + \la \beta \dx \overline c_{h,\tau}(t), \dx q_h\ra_\E 
   + \la \gamma \overline c_{h,\tau}(t), q_h\ra_{h,\E} = \la \delta \overline u_{h,\tau}(t), q_h\ra_{h,\E},
\end{alignat*}
for all test functions $v_h \in V_h$ and $q_h \in V_h$, and for a.a. $0 \le t \le T$. 

\smallskip 

\noindent 
{\em Step~1b:}
From the a-priori estimates of Theorem~\ref{thm:disc} and the weak compactness of bounded sets in reflexive Banach spaces \cite{Roubicek13}, one may deduce that 
there exist limit functions $u^* \in L^2(0,T;H^1(\E)) \cap H^1(0,T;H^1(\E)')$ and $c^* \in L^2(0,T;H^1(\E)) \cap H^1(0,T;L^2(\E))$ with
\begin{alignat*}{2}
 u_{h,\tau} &\rightharpoonup u^* \qquad && \text{in } L^2(0,T;H^1(\E)) \cap H^1(0,T;H^1(\E)') \\
  \overline u_{h,\tau} &\rightharpoonup u^* \qquad && \text{in } L^2(0,T;H^1(\E))
\end{alignat*}
as well as 
\begin{alignat*}{2}
 c_{h,\tau} &\rightharpoonup c^* \qquad && \text{in } L^2(0,T;H^1(\E)) \cap H^1(0,T;L^2(\E)) \\ 
 c_{h,\tau} &\rightharpoonup^* c^* \qquad && \text{in } L^{\infty}(0,T;H^1(\E)) \\
\overline c_{h,\tau} &\rightharpoonup c^* && \text{in } L^2(0,T;H^1(\E)) \\
\widehat c_{h,\tau} &\rightharpoonup c^* && \text{in } L^2(0,T;H^1(\E)).
\end{alignat*}
This implies that $\dt u_{h,\tau} \rightharpoonup \dt u^*$ and $\dt c_{h,\tau} \rightharpoonup \dt c^*$ in $L^2(0,T;H^1(\E)')$; see \cite{Roubicek13} for details. 

\smallskip 

It remains to show that $(u^*, c^*)$ satisfies \eqref{eq:var1}--\eqref{eq:var2}
and that the initial values satisfy $u_{h,\tau}(0) \rightharpoonup u^*(0) = u_0$ and 
$c_{h,\tau}(0) \rightharpoonup c^*(0) = c_0$. 
We start with verifying \eqref{eq:var2}.

\smallskip

\noindent 
{\em Step~2a:}
Recall that $\pi_h : L^2(\E) \to V_h$ denotes the $L^2$-orthogonal projection onto $V_h$.
Then
\begin{align*}
\int_{t'}^{t''} \la \dt c_{h,\tau}(t), v\ra_{h,\E} dt
 &= \int_{t'}^{t''} \la \dt c_{h,\tau}(t), v\ra_{\E} dt+ \int_{t'}^{t''} \la \dt c_{h,\tau}(t), \pi_h v - v\ra_{\E} dt \\
        &+ \int_{t'}^{t''}  \la \dt c_{h,\tau}(t), \pi_h v\ra_{h,\E} - \la \dt c_{h,\tau}(t), \pi_h v\ra_\E dt  
        = (i) + (ii) + (iii).
\end{align*}
By the weak convergence of $\dt c_{h,\tau}$ in $L^2(0,T;H^1(\E)')$, we infer that $(i) \to \int_{t'}^{t''} \la \dt c^*(t), v\ra_{\E} dt$ 
for all $v \in H^1(\E)$ as $h,\tau \to 0$. 
From the  density of $\{V_h\}_{h>0}$ in $H^1(\E)$, the approximation properties of the $L^2$-projection $\pi_h$, 
and the quasi-uniformity of the mesh $T_h(\E)$, 
one can see that $\|\pi_h v - v\|_{H^1(\E)} \to 0$ with $h \to 0$. The uniform boundedness of $\dt c_{h,\tau}$ in the norm of $L^2(0,T;H^1(\E)')$ therefore yields that $(ii) \to 0$ for all $v \in H^1(\E)$ when $h,\tau \to 0$. 
Now let $L_{h,\tau} : H^1(\E) \to \RR$ be linear operators defined by 
$$
L_{h,\tau} v = \int_{t'}^{t''}  \la \dt c_{h,\tau}(t), \pi_h v\ra_{h,\E} - \la \dt c_{h,\tau}(t), \pi_h v\ra_\E dt  = (iii).
$$ 
Then by the uniform bounds for the discrete solution $(u_h^n,c_h^n)_{n \ge 0}$
and the quasi-uniformity of the mesh $T_h(\E)$, one can see that
\begin{align*}
|L_{h,\tau} v| \le C \|\dt c_{h,\tau}\|_{L^2(t',t'',L^2(\E))} h^s \|v\|_{H^k_{pw}(\E)}
        \le C' \|\dt c_{h,\tau}\|_{L^2(t',t'',H^1(\E)')} h^{k-1}\|v\|_{H^k_{pw}(\E)}, 
\end{align*}
for $1 \le k \le 2$ and all $v \in H^1(\E) \cap H^k_{pw(\E)}$. 
Recall that $H^k_{pw}=\{v \in L^2(\E) : v|_e \in H^k(e)\}$ is the space of piecewise smooth functions. 
Thus, the family $\{L_{h,\tau}\}_{h,\tau>0}$ of linear operators is uniformly bounded on $H^1(\E)$ and 
$L_{h,\tau} v \to 0$ for $v \in H^{2}_{pw}(\E) \cap H^1(\E)$ which is dense in $H^1(\E)$. 
Hence, $L_{h,\tau} v \to 0$ for all $v \in H^1(\E)$ by the Banach-Steinhaus theorem \cite{Heuser82}. 
In summary, we thus have shown that 
\begin{align*}
\int_{t'}^{t''} \la \dt c_{h,\tau}(t) , \pi_h v\ra_{\E,h} dt \to \int_{t'}^{t''}\la \dt c^*(t), v\ra_{\E} dt
\end{align*}
for all $0 \le t'< t'' \le T$ and all $v \in H^1(\E)$ as $h,\tau \to 0$. 
It then follows from Lebesgue's differentiation theorem \cite{Evans10} that 
\begin{align*} 
\la \dt c_{h,\tau}(t), \pi_h v\ra_{\E,h} \to \la \dt c^*(t), v\ra_{\E}
\end{align*}
for all $v \in H^1(\E)$ and for a.a. $0 \le t \le T$. 

\smallskip

\noindent 
{\em Step~2b:}
In a similar manner, one can show that 
\begin{align*}
&\la \beta \dx \overline c_{h,\tau}(t),\dx \pi_h v\ra_{\E}  \to \la \beta \dx c^*(t), \dx  v\ra_{\E}, \\
&\la \gamma \overline c_{h,\tau}(t),\pi_h v\ra_{\E,h}  \to \la \gamma c^*(t), v\ra_{\E} , \\
&\la \delta \overline u_{h,\tau}(t),\pi_h v\ra_{\E,h} \to \la \delta u^*(t), v\ra_{\E},
\end{align*}
for all $v \in H^1(\E)$ and for a.a. $0 \le t \le T$. 
This shows that $(u^*,c^*)$ satisfies equation \eqref{eq:var2}. 

\smallskip 

We next turn to the verification of identity \eqref{eq:var1}. 

\smallskip

\noindent 
{\em  Step~3a:}
With the very same arguments as above, one can show that
\begin{align*}
&\la \dt u_{h,\tau}(t), \pi_h v\ra_{\E,h} \to \la \dt u^*(t), v\ra_{\E} , \\
& \la \alpha \dx \overline u_{h,\tau}(t),\dx \pi_h v\ra_{\E} \to \la \alpha \dx u^*(t), \dx  v\ra_{\E} , 
\end{align*}
for all $v \in H^1(\E)$ and for a.a. $0 \le t \le T$ as $h,\tau \to 0$. 
It thus remains to establish the convergence for the convective term, which we do next.

\smallskip 

\noindent 
{\em Step~3b:}
Let us define linear operators $B_{h,\tau}: H^1(\E) \to \RR$ by
\begin{align*}
B_{h,\tau} v = \int_{t'}^{t''} \la \chi \widehat \pi_{h,\tau}(t) \overline u_{h,\tau}(t) \dx \widehat c_{h,\tau}(t) , \pi_h v\ra_\E dt.  
\end{align*}
Then from the uniform bounds for the discrete solution $(u_h^n,c_h^n)_{n \ge 0}$,
the $H^1$-stability of the $L^2$-projection $\pi_h$ on the quasi-uniform mesh $T_h(\E)$, 
and \eqref{eq:upwBounded}, one can deduce that 
\begin{align*}
|B_{h,\tau} v| \le C \|\overline u_{h,\tau}\|_{L^2(0,T,L^\infty(\E))} \|\widehat c_{h,\tau}\|_{L^\infty(0,T;H^1(\E))} \|v\|_{H^1(\E)}.
\end{align*}
This shows that the family of operators $\{B_{h,\tau}\}_{h,\tau>0}$ is uniformly bounded for all $h,\tau >0$. 
As a next step, we decompose 
\begin{align*}
B_{h,\tau} v 
&= \int_{t'}^{t''} \la \chi \dx \widehat c_{h,\tau}(t) \overline u_{h,\tau}(t), \dx \pi_h v \ra_\E dt \\
&\qquad + \int_{t'}^{t''} \la \chi \dx \widehat c_{h,\tau}(t) (\widehat \pi_{h,\tau}(t) \overline u_{h,\tau}(t) - \overline u_{h,\tau}(t)), \dx \pi_h v \ra_\E dt 
= (i) + (ii). 
\end{align*}
By the Aubin-Lions lemma \cite{Roubicek13}, we know that $\overline u_{h,\tau} \to u^*$ in $L^2(0,T;L^2(\E))$.
Moreover, $\dx \pi_h v \to \dx v$ in $L^2(\E)$ and $\dx \widehat c_{h,\tau} \rightharpoonup \dx c^*$ in $L^2(0,T;L^2)$. 
This implies that
\begin{align*} 
(i) \to \int_{t'}^{t''} \la \chi \dx c^*(t) u^*(t), \dx v\ra_\E dt 
\end{align*}
for all $0\le t' \le t'' \le T$ and for any test function $v \in H^1(\E)$. 
From the estimate \eqref{eq:upwind2}, we therefore deduce that $\|\widehat \pi_{h,\tau} \overline u_{h,\tau} - \overline u_{h,\tau}\|_{L^2(t',t'';L^\infty(\E))} \le C h^{1/2} \|\dx \overline u_{h,\tau}\|_{L^2(t',t'';L^2(\E))}$. 
Using the $H^1$-stability of the $L^2$-projection $\pi_h$ and the uniform a-priori bounds for 
the discrete solution $(u_h^n,c_h^n)_{n \ge 0}$, one can then see that 
\begin{align*}
|(ii)| \le C h^{1/2} \to 0 \qquad \text{with } h,\tau \to 0.
\end{align*}
In summary, we thus have shown that
\begin{align*} 
B_{h,\tau} v \to \int_{t'}^{t''} \la \chi u^*(t) \dx c^*(t) , \dx v\ra_\E dt \qquad \text{with } h,\tau \to 0
\end{align*}
for all $v \in H^1(\E)$. With the assertions of Step~3a, this shows that $(u^*,c^*)$ solves \eqref{eq:var1}. 

\smallskip

\noindent 
{\em Step~4:}
From the estimates for the quasi-interpolation operator $\widetilde \pi_h$ stated in Section~\ref{sec:problem}, 
one can deduce that $\widetilde \pi_h v \to v$ in $L^2(\E)$ for all $v \in L^2(\E)$. 
Together with the continuity of the trace mapping for the space $W(0,T)$, this yields $u^*(0)=u_0$ and $c^*(0)=c_0$.

\smallskip 

\noindent 
{\em Step~5:}
In summary, we have shown that $(u^*,c^*)$ is a weak solution of \eqref{eq:sys1}--\eqref{eq:sys6},
and from the uniqueness stated in Theorem~\ref{thm:global}, we infer that $u^* = u$ and $c^* = c$.
\end{proof}

\section{Error estimates} \label{sec:rates}

Under suitable smoothness assumptions on the true solution, we can now also derive quantitative convergence rates. The aim of this section is to prove the following result.
\begin{theorem}\label{thm:error}
Let (A1)--(A3) hold and assume that the solution $(u,c)$ of \eqref{eq:sys1}--\eqref{eq:sys6} satisfies \eqref{eq:reg1}--\eqref{eq:reg2}. Moreover, let $(u_h^n,c_h^n)_{n \ge 0}$ be the solution of Problem~\ref{prob:disc} and denote by $(u_{h,\tau}, c_{h,\tau})$ its linear piecewise interpolation in time. Then 
\begin{align*}
\|u - u_{h,\tau}\|_{L^\infty(0,T;L^2(\E))} + \|u - u_{h,\tau}\|_{L^2(0,T;H^1(\E))} \le C (h+\tau), \\
\|c - c_{h,\tau}\|_{L^\infty(0,T;L^2(\E))} + \|c - c_{h,\tau}\|_{L^2(0,T;H^1(\E))} \le C (h+\tau) , 
\end{align*}
with constant $C$ that only depends on the bounds for the coefficients, the time horizon $T$, the geometry of the network, and the norm of the solution $(u,c)$ in \eqref{eq:reg1}--\eqref{eq:reg2}. 
\end{theorem}
In the usual way \cite{Varga71,Wheeler73}, we decompose the errors in the two solution components via
\begin{align}
u - u_{h,\tau} &= (u - R_h u) + (R_h u - u_{h,\tau}), \label{eq:split1}\\
c - c_{h,\tau} &= (c - R_h c) + (R_h c - c_{h,\tau}), \label{eq:split2}
\end{align}
into approximation and discrete error components. For our analysis, 
we choose the operator $R_h: H^1(\E) \to V_h$ as the $H^1$-orthogonal projection onto $V_h$
defined by 
\begin{align} \label{eq:h1}
\la \dx R_h u, \dx v_h\ra_\E + \la R_h u, v_h \ra_\E = \la \dx u, \dx v_h\ra_\E + \la u, v_h \ra_\E \qquad \forall v_h \in V_h.
\end{align}
In the following two sections, we derive bounds for the two error contributions of the individual solution components, 
and we then complete the proof of the above theorem. 

\subsection{Approximation error} 

We start with summarizing some elementary properties of the $H^1$-projection $R_h : H^1(\E) \to V_h$ defined
above. These are well-known for a single edge $e$, see e.g. \cite{BrennerScott08}, and can be generalized easily
to the network setting.
\begin{lemma} \label{lem:h1}
For any $u \in H^1(\E)$, we have $\|R_h u\|_{H^1(\E)} \le \|u\|_{H^1(\E)}$ and 
\begin{align*}
\|u - R_h u\|_{L^2(\E)} + h^{1/2} \|u - R_h u\|_{L^\infty(\E)} + h \|u - R_h u\|_{H^1(\E)}&\le C h \|\dx u\|_{L^2(\E)}
\end{align*}
with uniform constant $C$.
If $u \in H^2_{pw}(T_h(\E)) \cap H^1(\E)$, then 
\begin{align*}
\|u - R_h u\|_{L^2(\E)} + h^{1/2} \|u - R_h u\|_{L^\infty(\E)} + h \|u - R_h u\|_{H^1(\E)} \le C h^2 \|\dxx' u\|_{L^2(\E)}. 
\end{align*}
Here $\|\dxx' u\|_{L^2(\E)} = (\sum_T \|\dxx u\|_{L^2(T)}^2)^{1/2}$ denotes the norm of the broken derivative $\dxx' u$.
\end{lemma}
As a direct consequence of these estimates, we obtain the following bounds for the approximation error contributions in the above error splitting. 
\begin{lemma} \label{lem:6.3}
Let the assumptions of Theorem~\ref{thm:error} hold. Then
\begin{align*}
\|u - R_h u\|_{L^\infty(0,T;L^2(\E))} + \|u - R_h u\|_{L^2(0,T;H^1(\E))} \le C(u) h, \\
\|c - R_h c\|_{L^\infty(0,T;L^2(\E))} + \|c - R_h c\|_{L^2(0,T;H^1(\E))} \le C(c) h , 
\end{align*}
with $C(w) = C (\|\dx w\|_{L^\infty(0,T;L^2(\E))} + \|\dxx' w\|_{L^2(0,T;L^2(\E))})$  for $w=u,c$. 
\end{lemma}

For the proof of the corresponding estimates for the discrete error components
\begin{align*}
e_h^n=R_h u(t^n)-u_h^n \qquad \text{and} \qquad  d_h^n=R_h c(t^n)-c_h^n,
\end{align*}
we require a number of auxiliary results which are stated and proved in the next section.

\subsection{Auxiliary results}

As a preliminary step, we now state estimates for some terms that will arise in the analysis of the discrete error components $e_h^n$, $d_h^n$ below. 
\begin{lemma}\label{lem:errAux1}
Let $u_h,v_h \in V_h$, then for any $\eps > 0$,
\begin{align*}
\left|\la u_h,v_h\ra_{h,\E} - \la u_h,v_h\ra_\E\right| \le C h \|u_h\|_{H^1(\E)} \|v_h\|_{L^2(\E)} \leq \frac{C}{\eps} h^2 \|u_h\|_{H^1(\E)}^2 + \eps \|v_h\|_{L^2(\E)}^2. 
\end{align*}
\end{lemma}
\begin{proof}
Let us note that the numerical integration is exact, if the product $u_h v_h$ is piecewise linear. 
By the Bramble-Hilbert lemma and scaling arguments, one can then see that
\begin{align*}
\left|\la u_h,v_h\ra_{h,\E} - \la u_h,v_h\ra_\E\right| \le C h^2 \|\dxx' (u_h v_h)\|_{L^1(\E)}.
\end{align*}
Since $u_h,p_h \in V_h \subset P_1(T_h)$, one can further compute
\begin{align*}
\|\dxx' (u_h v_h)\|_{L^1(\E)} \le 2 \|\dx u_h \dx v_h\|_{L^1(\E)} \le 2 \|\dx u_h\|_{L^2(\E)} \|\dx v_h\|_{L^2(\E)}
\end{align*}
Via an inverse inequality, the second term can be bounded 
by $\|\dx v_h\|_{L^2(\E)} \le C h^{-1} \|v_h\|_{L^2(\E)}$. 
The result then follows by combination of the last two inequalities, summation over all elements, and application of the Cauchy-Schwarz inequality.
\end{proof}

Using the properties of the $H^1$-projection $R_h$, we can derive the following bounds.
\begin{lemma}\label{lem:errAux2}
Let $w \in H^1(\E)$. Then for any $\eps > 0$,
\begin{align*}
\left|\la R_h w, v_h\ra_{h,\E} - \la w,v_h\ra_{h,\E}\right| \le C h \|w\|_{H^1(\E)} \|v_h\|_h \le \frac{C}{\eps} h^2 \|w\|_{H^1(\E)}^2  + \eps \|v_h\|_h^2. 
\end{align*}
\end{lemma}
\begin{proof}
By application of Lemma~\ref{lem:h1}, the Cauchy-Schwarz inequality, and the norm equivalence estimates \eqref{eq:equiv}, 
we obtain 
\begin{align*}
|\la R_h w - w, v_h\ra_{h,\E}|
&\le C \|R_h w - w\|_{L^2(\E)} \|v_h\|_h 
 \le C' h \|w\|_{H^1(\E)} \|v_h\|_h. 
\end{align*}
The assertion of the lemma now follows via Young's inequality.
\end{proof}

As a next step, we derive an estimate for the errors introduced through mass lumping and the approximation of the time derivatives by finite differences. 
\begin{lemma}
Let $w \in L^\infty(0,T;L^2(\E)) \cap L^2(0,T;H^1(\E))$. Then for any $\eps>0$, 
\begin{align*}
&\left|\la \dtau R_h w(t^n), v_h\ra_{h,\E} - \la \dt w(t^n),v_h\ra_{\E}\right| 
 \le \eps \|v_h\|_h^2 + \eps \|\dx v_h\|_{L^2(\E)}^2 \\
&\qquad + \frac{C \tau }{\eps} \|\dtt w\|_{L^2(t^{n-1},t^n;H^1(\E)')}^2  
 + \frac{C h^2}{\eps \tau} \left( \|\dt w\|_{L^2(t^{n-1},t^n;H^1(\E))}^2 +\|\dxx' w\|_{L^2(t^{n-1},t^n;L^2(\E))}^2 \right).
\end{align*}
\end{lemma}
\begin{proof}
We start with splitting the error by 
\begin{align*}
\la \dtau R_h &w(t^n), v_h\ra_{h,\E} - \la \dt w(t^n),v_h\ra_{\E}  \\
&= \left(\la \dtau R_h w(t^n), v_h\ra_{h,\E} - \la \dtau w(t^n), v_h\ra_{\E}\right) 
  + \la \dtau w(t^n) - \dt w(t^n), v_h\ra_{\E}
  = (i) + (ii).
\end{align*}
By Lemmas~\ref{lem:errAux1} and \ref{lem:errAux2}, we readily obtain
\begin{align*}
|(i)| 
&\le C h \|\dtau w(t^n)\|_{H^1(\E)} \|v_h\|_h 
 \le \frac{C h^2}{\eps\tau} \int_{t^{n-1}}^{t^n} \|\dt w(t)\|^2_{H^1(\E)} dt + \frac{\eps}{2} \|v_h\|_h^2.
\end{align*}
In the second step, we used the fundamental theorem of calculus the Cauchy-Schwarz,  and Young's inequality. 
The second term can be further estimated by
\begin{align*}
|(ii)|
&\le  \|\dtau w(t^n) - \dt w(t^n)\|_{H^1(\E)'} \|v_h\|_{H^1(\E)} \\
&\le \frac{C}{\eps} \tau \int_{t^{n-1}}^{t^n} \|\dtt w(t)\|^2_{H^1(\E)'} dt + \frac{\eps}{2} \|v_h\|_h^2 + \frac{\eps}{2} \|\dx v_h\|_{L^2(\E)}^2.
\end{align*}
Here we used Taylor expansion and the Cauchy-Schwarz inequality for the first term, the norm equivalence \eqref{eq:equiv} for the second, and applied Young's inequality to split the product. A combination of the two estimates for $(i)$ and $(ii)$ yields the assertion.
\end{proof}

As a last step in our preliminary considerations, we now derive a bound for the error introduced through the upwinding strategy.
\begin{lemma}
Let $c_h^{n-1} \in V_h$ be given, and let $\widehat \pi_h^{n-1}$ denote the corresponding upwind operator as defined in \eqref{eq:upw}. Then for all $\eps>0$, 
\begin{align*}
&\left|\la \chi u(t^n)  \dx c(t^n) , \dx v_h \ra_\E - \la \chi \widehat \pi_h^{n-1} u_h^n \dx c_h^{n-1} , \dx v_h\ra_\E\right| \\
& \quad \le \frac{C}{\eps} \left(\tau \|\dt c\|^2_{L^2(t^{n-1},t^n;H^1(\E))} 
          + \frac{h^2}{\tau} \|\dxx' c\|^2_{L^2(t^{n-1},t^n;L^2(\E))}+ \tau \|\dt u\|^2_{L^2(t^{n-1},t^n;H^1(\E))} \right.\\
& \quad \quad \left. + \frac{h^3}{\tau} \|\dxx' u\|^2_{L^2(t^{n-1},t^n;L^2(\E))} + \|\dx d_h^{n-1}\|^2_{L^2(\E)} + \|e_h^n\|^2_h\right) + \eps(\|\dx e_h^n\|^2_{L^2(\E)} + \|\dx v_h\|^2_{L^2(\E)}).
\end{align*}
\end{lemma}
\noindent 
Recall that $e_h^n=R_hu(t^n)-u_h^n$ and $d_h^n=R_hc(t^n)-c_h^n$ are the discrete error contributions.
\begin{proof}
By application of the Cauchy-Schwarz and Young's inequality, we obtain 
\begin{align*}
\la \chi u(t^n) &\dx c(t^n), \dx v_h \ra_\E - \la \chi \widehat \pi_h^{n-1} u_h^n \dx c_h^{n-1} , \dx v_h\ra_\E \\
&\le \frac{C}{\eps} \| u(t^n) \dx c(t^n) - \widehat \pi_h^{n-1} u_h^n \dx c_h^{n-1} \|_{L^2(\E)}^2 + \eps \|\dx v_h\|_{L^2(\E)}^2.
\end{align*}
Using triangle inequalities, the first term can be further estimated by
\begin{align*}
\| u(t^n) &\dx c(t^n) - \widehat \pi_h^{n-1} u_h^n \dx c^{n-1} \|_{L^2(\E)} \\
&\le \|u(t^n) [\dx c(t^n) - \dx c(t^{n-1})] \|_{L^2(\E)} 
   + \|u(t^n) [\dx c(t^{n-1}) - \dx R_h c(t^{n-1})] \|_{L^2(\E)} \\
&  \qquad + \|u(t^n) [\dx R_h c(t^{n-1}) - \dx c_h^{n-1}] \|_{L^2(\E)} 
   + \|[u(t^n) - R_h u(t^n)] \dx c_h^{n-1} \|_{L^2(\E)}  \\
&  \qquad  + \|[R_h u(t^n) - \widehat\pi_h^{n-1} R_h u(t^n)] \dx c_h^{n-1} \|_{L^2(\E)} 
   + \|\widehat \pi_h^{n-1} [R_h u(t^n) - u_h^n] \dx c_h^{n-1} \|_{L^2(\E)} \\
&= (i) + (ii) + (iii) + (iv) + (v) + (vi).
\end{align*}
Before turning to the estimation of the individual terms, let us make a preliminary observation. 
From the embedding of $H^1(\E)$ in $L^\infty(\E)$, the bounds for $\|u(t)\|_{H^1(\E)}$ in \eqref{eq:reg1}, and the estimates of Theorem~\ref{thm:disc}, we obtain
\begin{align} \label{eq:ubound}
\|u(t)\|_{L^\infty(\E)} \le C \qquad \text{and} \qquad \|\dx c_h^n\|_{L^2(\E)} \le C,
\end{align}
and these bounds hold uniformly for all $0 \le t\le T, n\ge 0 $. 
With the aid of H\"older's inequality, Taylor estimates, and the Cauchy-Schwarz inequality, we can then estimate
the first term in the above error expansion by
\begin{align*}
(i) &\le \|u(t^n)\|_{L^\infty(\E)}  \|\dx c(t^n) - \dx c(t^{n-1})\|_{L^2(\E)} 
     \le  C \tau^{1/2} \|\dt c\|_{L^2(t^{n-1},t^n;H^1(\E))}.
\end{align*}
For the second term, we introduce temporal averages $\overline c^{n} = \frac{1}{\tau} \int_{t^{n-1}}^{t^n} c(t) dt$ for the function $c$. Then by the H\"older and triangle inequalities, and using the bound \eqref{eq:ubound}, we get
\begin{align*}
(ii) &\le \|u(t^n)\|_{L^\infty(\E)} \|\dx c(t^{n-1}) - \dx R_h c(t^{n-1})\|_{L^2(\E)}  
      \le C \left(\|\dx c(t^{n-1}) - \dx \overline c^{n}\|_{L^2(\E)} \right.\\
     & \qquad \qquad \qquad \qquad \qquad 
       \left. + \|\dx \overline c^{n} - \dx R_h \overline c^n\|_{L^2(\E)} 
        + \|\dx R_h \overline c^{n} - \dx R_h  c(t^{n-1})\|_{L^2(\E)} \right)\\
     &\le C (\tau \|\dt c\|_{L^2(t^{n-1},t^n;H^1(\E))}^2 + h^2/\tau \|\dxx' c\|_{L^2(t^{n-1},t^n;L^2(\E))}^2)^{1/2}. 
\end{align*}
For the third term in the above estimate, we get 
\begin{align*}
(iii) \le \|u(t^n)\|_{L^{\infty}(\E)} \|\dx d_h^{n-1}\|_{L^2(\E)} \le C \|\dx d_h^{n-1}\|_{L^2(\E)}.
\end{align*}
Using H\"older's inequality, the uniform bounds for $\|\dx c_h^n\|_{L^2(\E)}$ provided by Theorem~\ref{thm:disc}, 
and similar reasoning as in the estimate of the term $(ii)$, we obtain
\begin{align*}
(iv) &\le  \|u(t^n) - R_h u(t^n)\|_{L^\infty(\E)} \|\dx c_h^{n-1}\|_{L^2(\E)} \\
     &\le C (\|u(t^n) - \overline u^{n}\|_{L^\infty(\E)} + \|\overline u^{n} - R_h \overline u^{n}\|_{L^\infty(\E)} + \|R_h \overline u^{n} - R_h u(t^n)\|_{L^\infty(\E)}) \\
     &\le C (\tau \|\dt u\|^2_{L^2(t^{n-1},t^n;H^1(\E))} + h^3/\tau \|\dxx' u\|^2_{L^2(t^{n-1},t^n;L^2(\E))})^{1/2}.
\end{align*}
In the last step, we employed Taylor estimates, the Cauchy-Schwarz inequality, and the properties of the $H^1$-projection. By the estimate \eqref{eq:upwind2} for the upwind projection and similar arguments as before, 
the fifth term can be estimated by 
\begin{align*}
(v) &\le  C (\tau \|\dt u\|^2_{L^2(t^{n-1},t^n;H^1(\E))} + h^2/\tau \|\dx u\|^2_{L^2(t^{n-1},t^n;H^1_{pw}(\E))})^{1/2}.
\end{align*}
Using the interpolation inequality \eqref{interpolationL2} and the uniform bounds \eqref{eq:ubound} 
for the discrete solution, the last term can finally be bounded by
\begin{align*}
(vi) &\le \|e_h^n\|_{L^\infty(\E)} \|\dx c_h^{n-1}\|_{L^2(\E)}  
      \le C(\eps \|\dx e_h^n\|^2_{L^2(\E)} + C'/\eps \|e_h^n\|_h^2)^{1/2}. 
\end{align*}
The assertion of the lemma now follows by squaring the estimates of the terms (i)--(vi), and 
combination with the first inequality of the proof.
\end{proof}

\subsection{Estimates for the discrete error}

We are now in the position to derive the following bounds for the discrete error contributions. 
\begin{lemma} \label{lem:6.8}
Let the assumptions of Theorem~\ref{thm:error} hold. 
Then the discrete error components $e_h^n = R_h u(t^n) - u_h^n$ and $d_h^n = R_h c(t^n) - c_h^n$ satisfy the bounds
\begin{align*}
\|e_h^n\|^2_h + \sum_{k=1}^n \|\dx e_h^n\|^2_{L^2(\E)} \le C(\tau^2 + h^2), \\
\|d_h^n\|^2_h + \sum_{k=1}^n \|\dx d_h^n\|^2_{L^2(\E)} \le C(\tau^2 + h^2).
\end{align*}
The constants only depend on the parameters of the problem, the geometry of the graph,
and the norms of the solution components appearing in \eqref{eq:reg1}--\eqref{eq:reg2}.
\end{lemma}
\begin{proof}
From the variational characterization of $u$ and $u_h^n$, one can see that
\begin{align*}
\la \dtau e_h^n,& v_h\ra_{h,\E} + \la \alpha \dx e_h^n,\dx v_h\ra_\E \\
    &= \la \dtau R_h u(t^n), v_h \ra_{h,\E} - \la \dt u(t^n), v_h \ra_\E 
     + \la \dx R_h u(t^n) - \dx u(t^n), \dx v_h\ra_\E    \\
    &\qquad \qquad \qquad + \la \chi [u(t^n) \dx c(t^n)  - \widehat\pi_h^{n-1} u_h^n\dx c_h^{n-1}] , \dx v_h \ra_\E.
\end{align*}
As in the proof of the preceding lemma, we have for all $\eps>0$,
\begin{align*}
|\la \dx R_h u(t^n) &- \dx u(t^n), \dx v_h\ra_\E| 
\le  \eps \|\dx v_h\|^2_{L^2(\E)} \\
&+ \frac{C}{\eps} \left(\tau \|\dt u\|_{L^2(t^{n-1},t^n;H^1(\E))}^2 
+ h^2/\tau \|\dxx' u\|_{L^2(t^{n-1},t^n;L^2(\E))}^2\right). 
\end{align*}
The remaining terms on the right hand side can be estimated by the previous lemmas. 
Testing with $v_h = e_h^n$ and using some elementary computations, one thus obtains 
\begin{align*}
&\frac{1}{2}\dtau \|e_h^n\|_h^2 + \underline \alpha \|\dx e_h^n\|^2_{L^2(\E)} 
\le \eps  \|\dx e_h^n\|^2_{L^2(\E)} \\
&  + \frac{C}{\eps} \left(\tau \|\dt c\|^2_{L^2(t^{n-1},t^n;H^1(\E))} 
+ \frac{h^2}{\tau} \|\dxx' c\|^2_{L^2(t^{n-1},t^n;L^2(\E))}
+ \tau \|\dtt u\|^2_{L^2(t^{n-1},t^n;H^1(\E)')} \right.\\
& \quad \left. + (\tau +\frac{h^2}{\tau}) \|\dt u\|^2_{L^2(t^{n-1},t^n;H^1(\E))}
+ \frac{h^2}{\tau} \|\dxx' u\|^2_{L^2(t^{n-1},t^n;L^2(\E))} + \|\dx d_h^{n-1}\|^2_{L^2(\E)} + \|e_h^n\|^2_h \right).
\end{align*}
Choosing $\eps = \underline \alpha/2$ allows to absorb the first term on the right hand side by the left hand side
of the inequality. Multiplication by $2\tau$, summation over $n$, and using the bounds \eqref{eq:reg1}--\eqref{eq:reg2} further yields
\begin{align}\label{eq:aux7}
\|e_h^n\|_h^2 + \sum_{k=1}^n \tau  \|\dx e_h^n\|^2_{L^2(\E)} 
\le C \left(\sum_{k=1}^n \tau \|\dx d_h^{n-1}\|^2_{L^2(\E)} + \sum_{k=1}^n \tau \|e_h^{n}\|^2_h + \tau^2 + h^2 \right).
\end{align}
From the variational equations characterizing $c$ and $c_h^n$, we obtain that 
\begin{align*}
\la \dtau d_h^n, q_h\ra_{h,\E} 
   &+ \la \beta \dx d_h^n, \dx q_h\ra_{\E} + \la \gamma d_h^n, q_h\ra_{h,\E} =\la \delta e_h^n, q_h\ra_{h,\E} \\
   &+ \big(\la \dtau R_h c(t^n), q_h\ra_{h,\E} - \la \dt c(t^n), q_h\ra_\E \big)
    + \big( \la \beta (\dx R_h c(t^n) - \dx c(t^n)), \dx v_h \ra_\E \big)\\
   &+ \big(\la \gamma R_h c(t^n), q_h\ra_{h,\E} - \la \gamma c(t^n), q_h\ra_\E \big)
    + \big( \la \delta u(t^n), q_h\ra_\E - \la \delta R_h u(t^n), q_h\ra_{h,\E} \big).  
\end{align*}
As before, we denote by $\overline c^n = \tfrac{1}{\tau} \int_{t^{n-1}}^{t^n} c(t) dt$ the temporal averages of the function $c$ on the subinterval $[t^{n-1},t^n)$. We then have for all $\eps>0$, 
\begin{align*}
|\la \gamma &R_hc(t^n), q_h \ra_{h,\E} - \la \gamma c(t^n),q_h \ra_{\E} | \le | \la \gamma R_h(c(t^n)- \overline c^n),q_h \ra_{h,\E}| \\
&\qquad+ | \la \gamma R_h \overline c^n, q_h \ra_{h,\E} - \la \gamma R_h \overline c^n, q_h \ra_{\E} | 
+ |\la \gamma (R_h \overline c^n - \overline c^n), q_h \ra_{\E}| 
+ |\la \gamma(\overline c^n - c(t^n)), q_h \ra_{\E}| \\
&\le \frac{C}{\eps}\left(\tau \| \dt c\|^2_{L^2(t^{n-1},t^n;L^2(\E))} + \frac{h^2}{\tau} \|\dx c\|^2_{L^2(t^{n-1},t^n;L^2(\E))}\right)
 + \eps \|q_h\|_h^2.
\end{align*}
Here, we have used the equivalence of the norms \eqref{eq:equiv} for the first term, 
Lemma~\ref{lem:errAux1} for the second term, as well as  Taylor estimates and the properties of the $H^1$-projector.
In the same way, we obtain for all $\eps >0$,
\begin{align*}
|\la \delta R_h &u(t^n), q_h \ra_{h,\E} - \la \delta u(t^n),q_h \ra_{\E} | \\
& \le \frac{C}{\eps} \left(\tau \| \dt u\|^2_{L^2(t^{n-1},t^n;L^2(\E))} + \frac{h^2}{\tau} \|\dx u\|^2_{L^2(t^{n-1},t^n;L^2(\E))} \right) 
+ \eps \|q_h\|_h^2.
\end{align*}
The remaining terms on the right hand side can now be estimated by the results of the previous section.
Choosing $q_h = d_h^n$ and some elementary computations, we thus arrive at 
\begin{align*}
&\frac{1}{2} \dtau \|d_h^n\|^2_h + \underline \beta \|\dx d_h^n\|^2_{L^2(\E)} \le \eps \|\dx d_h^n\|^2_{L^2(\E)} 
  + \frac{C}{\eps} \|d_h^n\|^2_h + \|e_h^n\|^2_h \\
& \quad + \frac{C}{\eps} \left(\tau \|\dtt c\|^2_{L^2(t^{n-1},t^n;H^1(\E)')} + (\tau+\frac{h^2}{\tau}) \|\dt c\|^2_{L^2(t^{n-1},t^n;H^1(\E))}   + \tau \| \dt u\|^2_{L^2(t^{n-1},t^n;L^2(\E))} \right.\\
& \left.  \qquad \qquad \quad 
  + \frac{h^2}{\tau} \left( \| \dxx' c\|^2_{L^2(t^{n-1},t^n;L^2(\E))}
  +  \|\dx c\|^2_{L^2(t^{n-1},t^n;L^2(\E))} + \|\dx u\|^2_{L^2(t^{n-1},t^n;L^2(\E))}\right)  \right).
\end{align*}
With $\eps=\underline \beta/2$, the first term on the right hand side can be absorbed in the left hand side. 
Multiplying by $2 \tau$, summing over $n$, and using the bounds \eqref{eq:reg1}--\eqref{eq:reg2}, we obtain
\begin{align}\label{eq:aux8}
\|d_h^n\|^2_h + \sum_{k=1}^n \tau \|\dx d_h^n\|_{L^2(\E)}^2 
\le C \left(  \sum_{k=1}^n \tau\|d_h^n\|^2_h + \sum_{k=1}^n \tau\|e_h^n\|^2_h +\tau^2 + h^2\right).
\end{align}
An application of the discrete Gronwall lemma \eqref{eq:discrete_gronwall} further yields
\begin{align*}
\|d_h^n\|^2_h + \sum_{k=1}^n \tau \|\dx d_h^n\|_{L^2(\E)}^2 
\le C' \left( \sum_{k=1}^n \tau\|e_h^n\|^2_h +\tau^2 + h^2\right).
\end{align*}
Inserting this estimate into \eqref{eq:aux7} and applying the discrete Gronwall lemma \eqref{eq:discrete_gronwall} once more, 
we can conclude that 
\begin{align}
\|e_h^n\|_h^2 + \sum_{k=1}^n \tau  \|\dx e_h^n\|^2_{L^2(\E)} 
\le C (\tau^2 + h^2).
\end{align}
This is the required estimate for the first component of the discrete error. 
Using this estimate in \eqref{eq:aux8} yields the bound for the second component of the discrete error.
\end{proof}

\subsection{Proof of Theorem~\ref{thm:error}}
The error splitting \eqref{eq:split1}--\eqref{eq:split2}, we directly obtain 
\begin{align*}
\|u - u_{h,\tau}\| &\le \|u - R_h u\| + \|R_h u - u_{h,\tau}\|, \\
\|c - c_{h,\tau}\| &\le \|c - R_h c\| + \|R_h c - c_{h,\tau}\|.
\end{align*}
The assertion of the theorem now follows by simply estimating the individual terms 
with the help of the bounds provided by Lemma~\ref{lem:6.3} and Lemma~\ref{lem:6.8}

\section{Numerical illustration}\label{sec:numerics}

In this section, we illustrate the theoretical results of the paper by some numerical tests. %

\subsection{Tripod network}
Let us start with verifying the convergence rates obtained in Section~\ref{sec:discretization}. To this end, we study the chemotaxis problem on a network consisting of three pipes meeting at a junction; see Figure~\ref{fig:graph}. We choose the edge lengths $l_i = l_{e_i} = 1$ and the parameters $\alpha_i,\chi_i,\gamma_i,\delta_i = 1$, $\beta_i = 0.1$ for $i=1,2,3$. For the initial values we let 
\begin{align*} 
u_{i,0}(x) = 4, \quad i=1,2,3; \qquad c_{i,0}(x) = 0, \quad i=1,2;\ c_{3,0}(x) = 1-\cos(\pi x), \quad x\in [0,1].
\end{align*}
In the following, we present the time-evolution of the two concentrations with discretization parameters $h=2^{-4}, \tau =2^{-7}$.
This setup leads to the occurrence of a peak of both concentrations at vertex $v_4$. The time-evolution is illustrated in figures \ref{fig:evolutionC} and \ref{fig:evolutionU}.
\begin{figure}[ht!]
\begin{subfigure}{.30\textwidth}
  \centering
  \includegraphics[width=.9\linewidth]{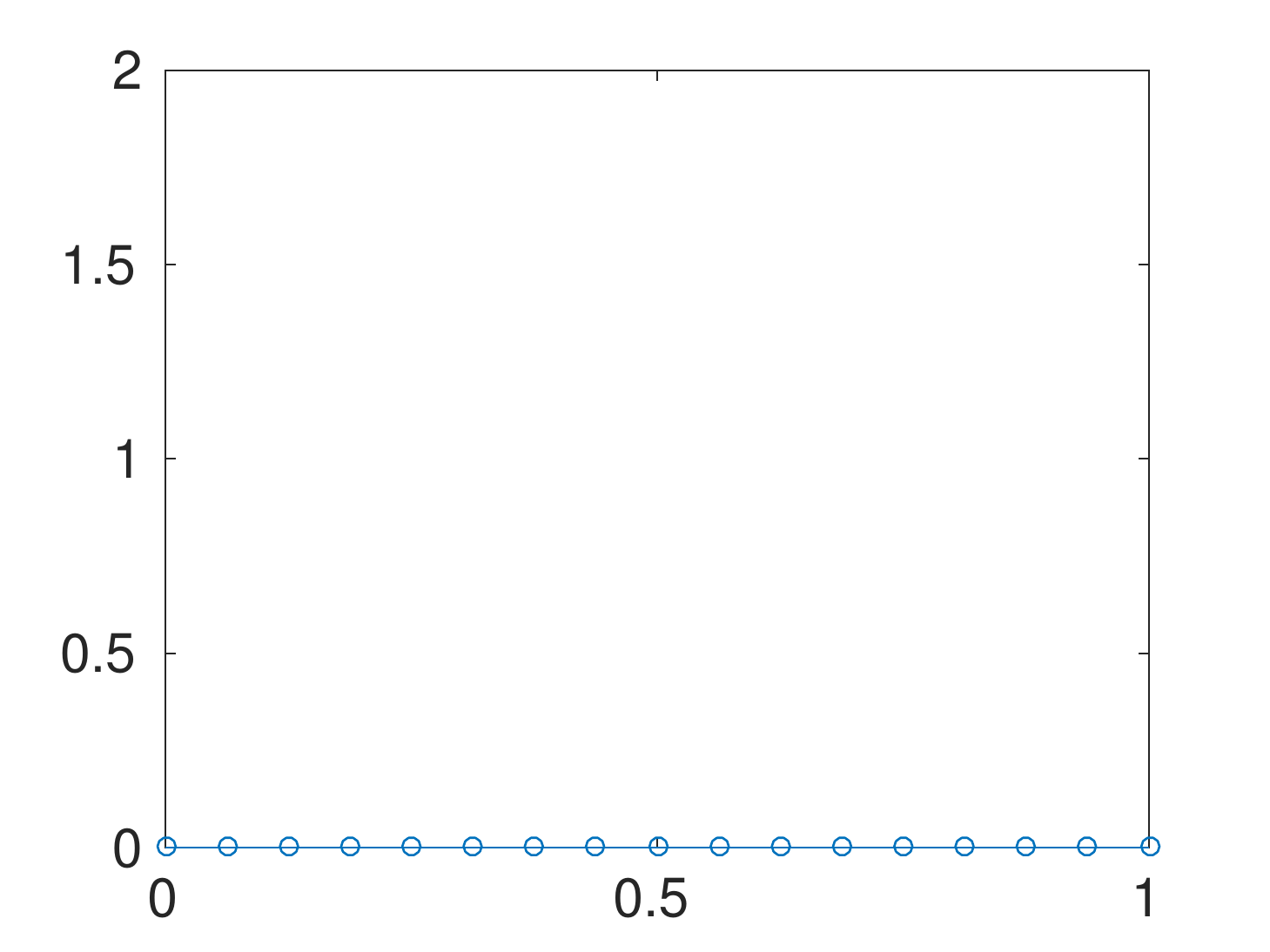}
  \caption*{$e_1$ at $t=0$}
  \label{fig:C1_t=0}
\end{subfigure}%
\begin{subfigure}{.30\textwidth}
  \centering
  \includegraphics[width=.9\linewidth]{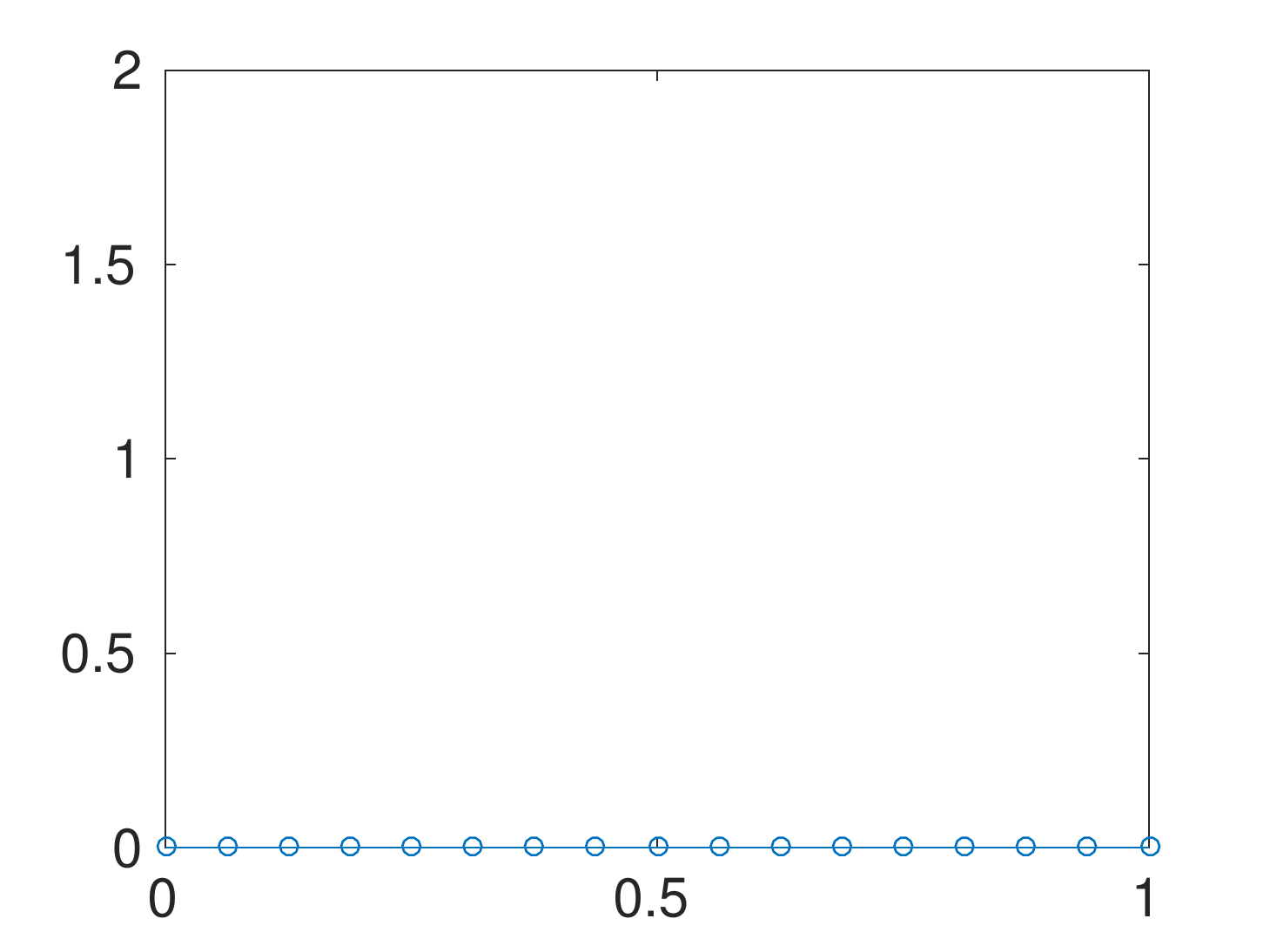}
  \caption*{$e_2$ at $t=0$}
  \label{fig:C2_t=0}
\end{subfigure} 
\begin{subfigure}{.30\textwidth}
  \centering
  \includegraphics[width=.9\linewidth]{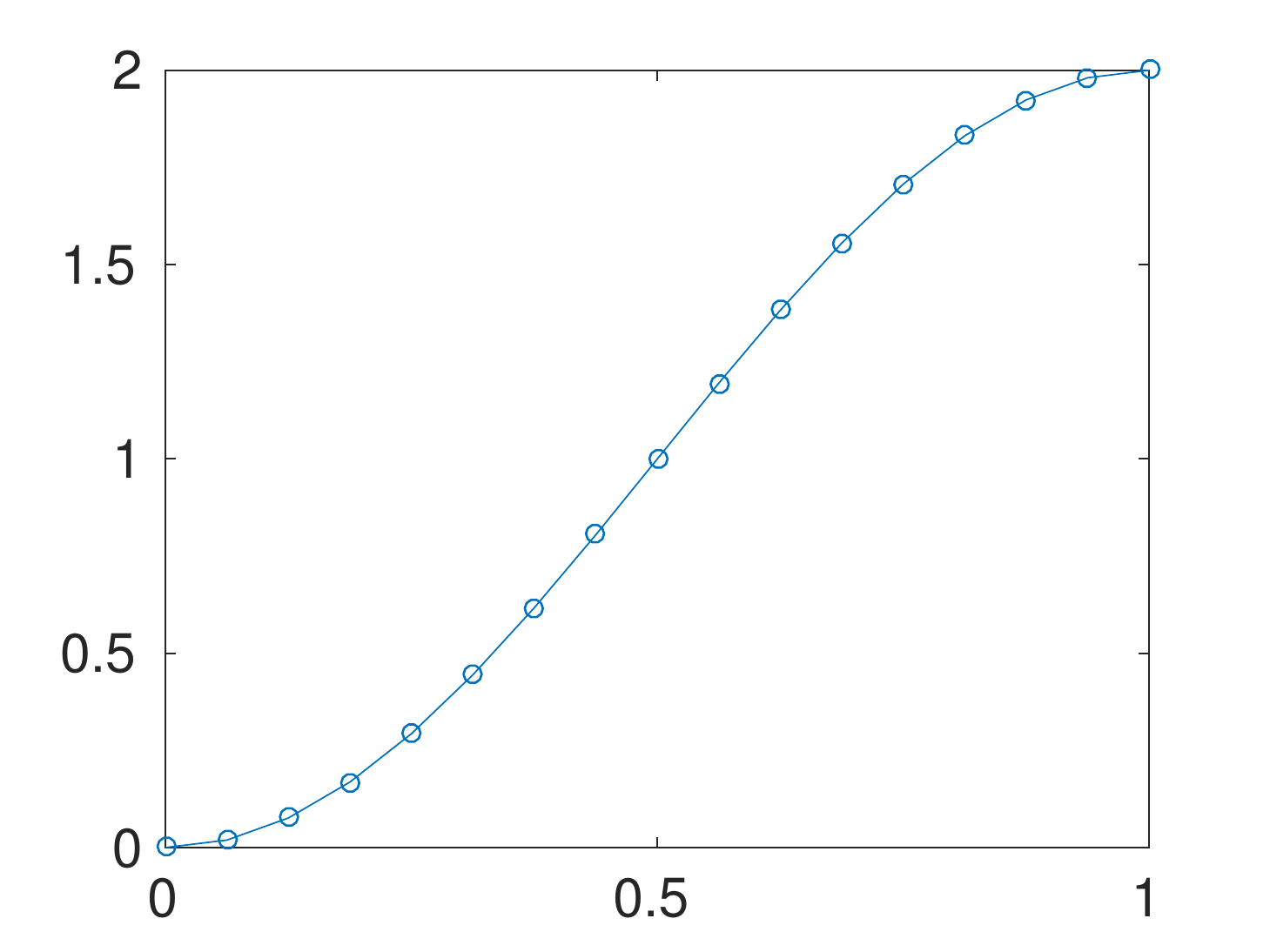}
  \caption*{$e_3$ at $t=0$}
  \label{fig:C3_t=0}
\end{subfigure} 
\begin{subfigure}{.30\textwidth}
  \centering
  \includegraphics[width=.9\linewidth]{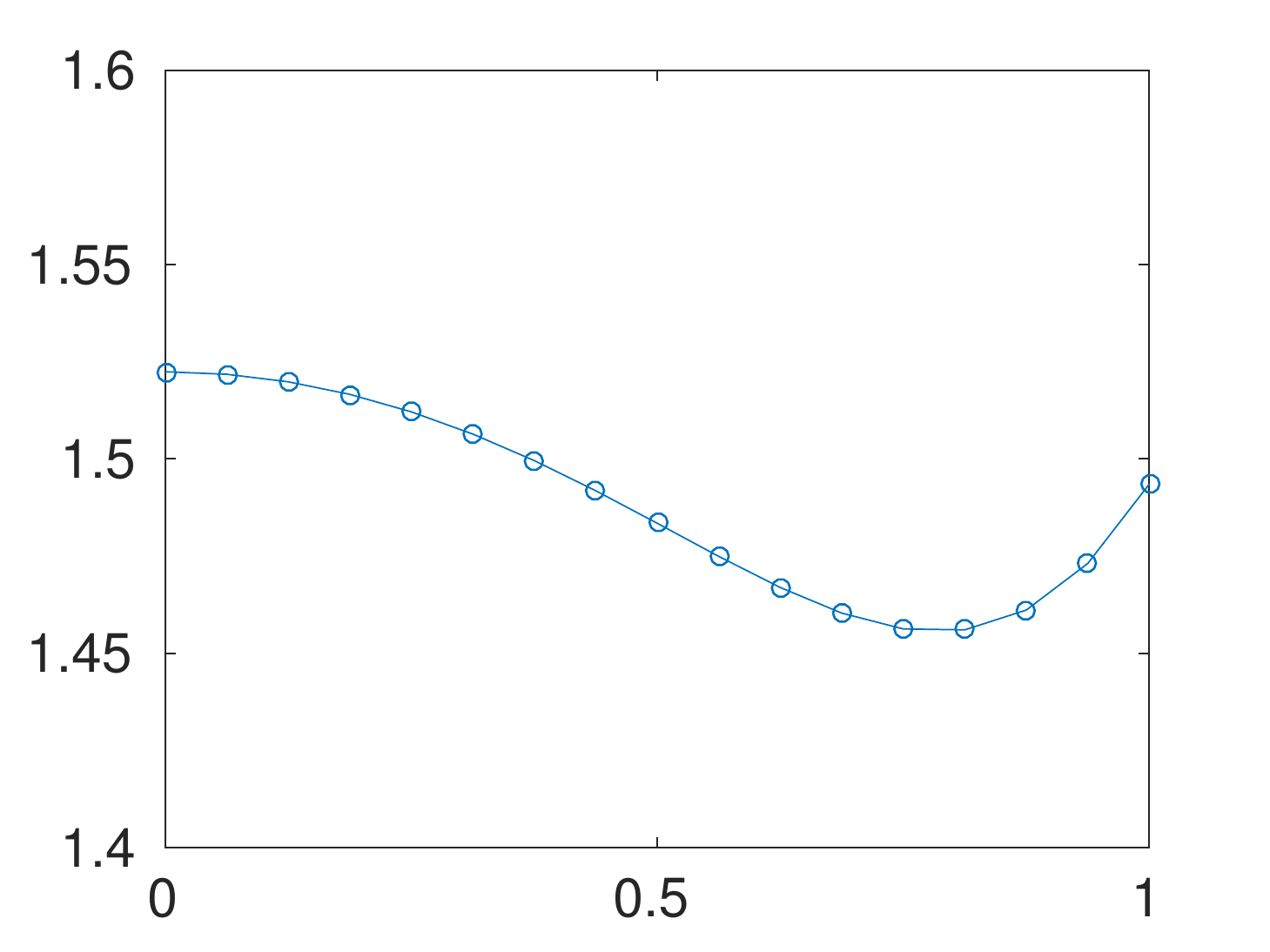}
  \caption*{$e_1$ at $t=0.5$}
  \label{fig:C1_t=0.5}
\end{subfigure}%
\begin{subfigure}{.30\textwidth}
  \centering
  \includegraphics[width=.9\linewidth]{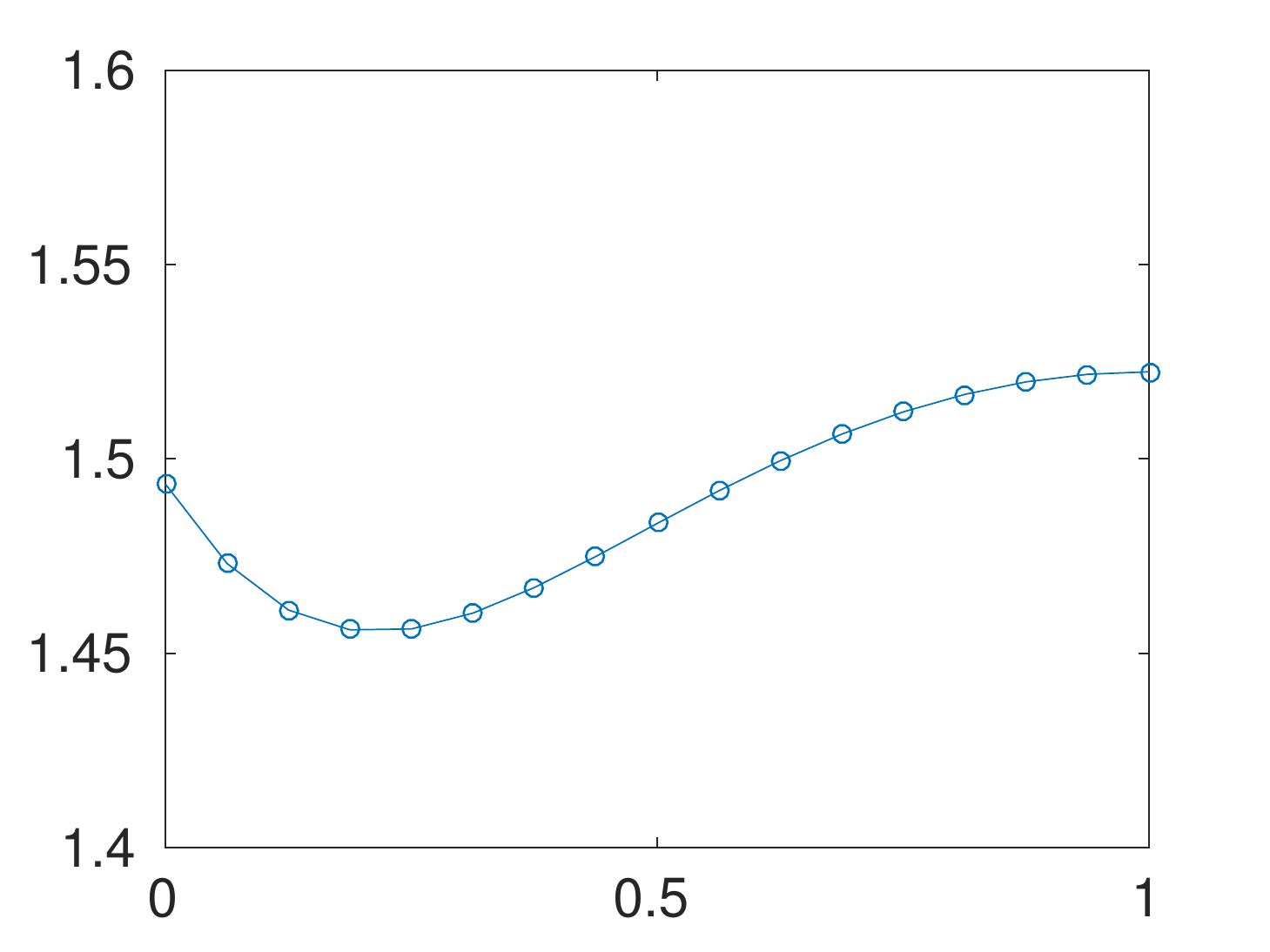}
  \caption*{$e_2$ at $t=0.5$}
  \label{fig:C2_t=0.5}
\end{subfigure} 
\begin{subfigure}{.30\textwidth}
  \centering
  \includegraphics[width=.9\linewidth]{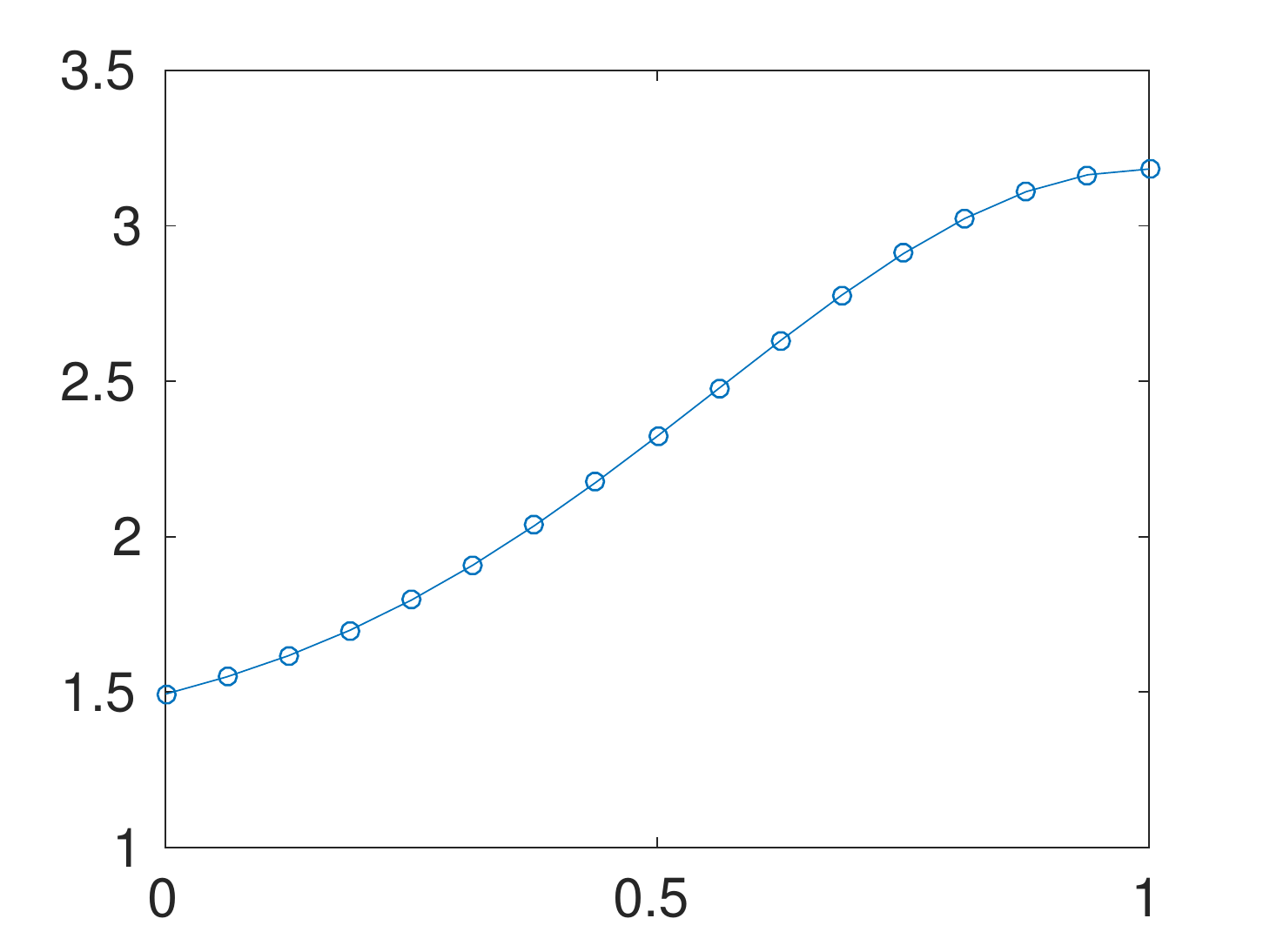}
  \caption*{$e_3$ at $t=0.5$}
  \label{fig:C3_t=0.5}
\end{subfigure} 
\begin{subfigure}{.30\textwidth}
  \centering
  \includegraphics[width=.9\linewidth]{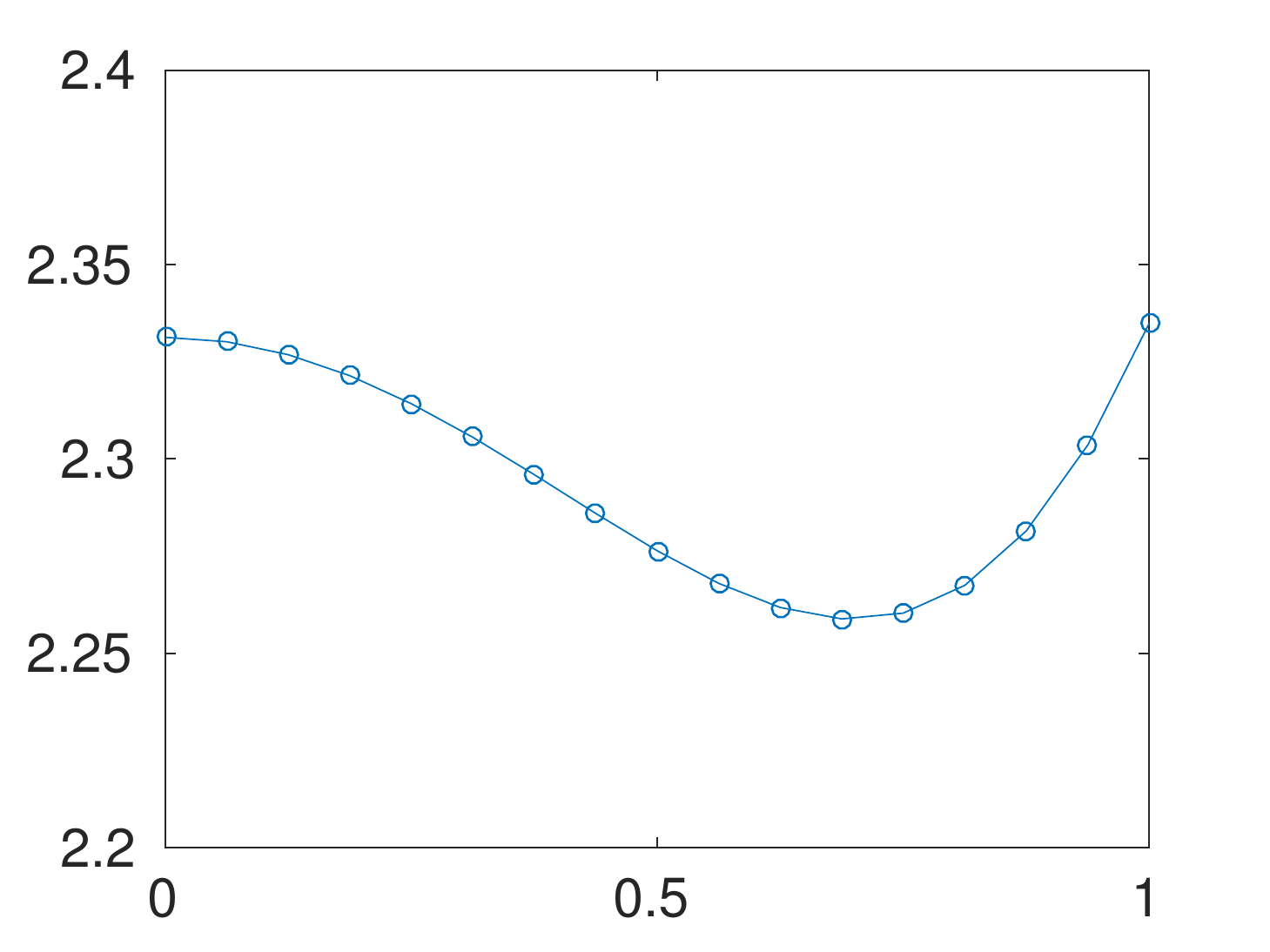}
  \caption*{$e_1$ at $t=1$}
  \label{fig:C1_t=1}
\end{subfigure}
\begin{subfigure}{.30\textwidth}
  \centering
  \includegraphics[width=.9\linewidth]{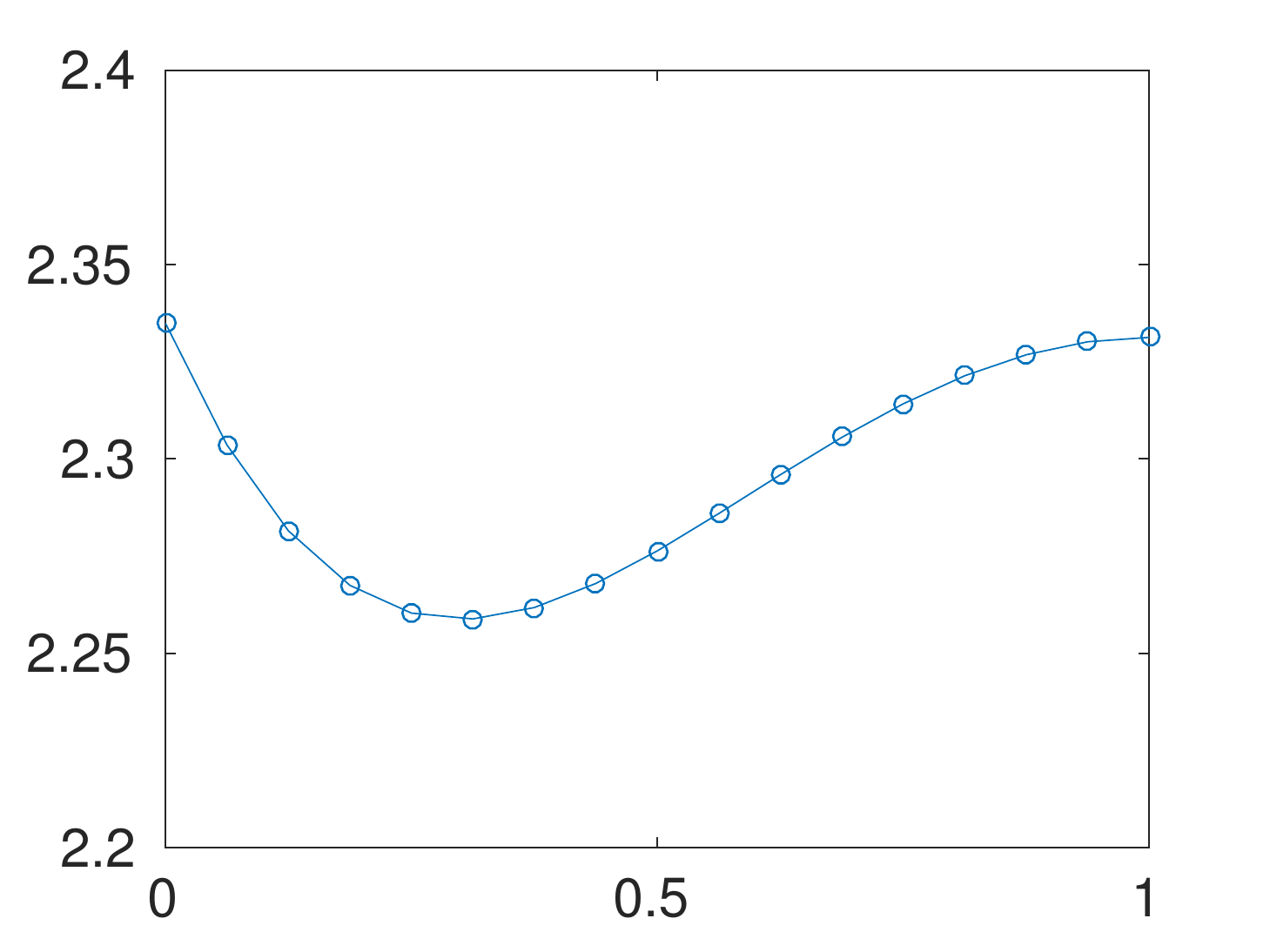}
  \caption*{$e_2$ at $t=1$}
  \label{fig:C2_t=1}
\end{subfigure}
\begin{subfigure}{.30\textwidth}
  \centering
  \includegraphics[width=.9\linewidth]{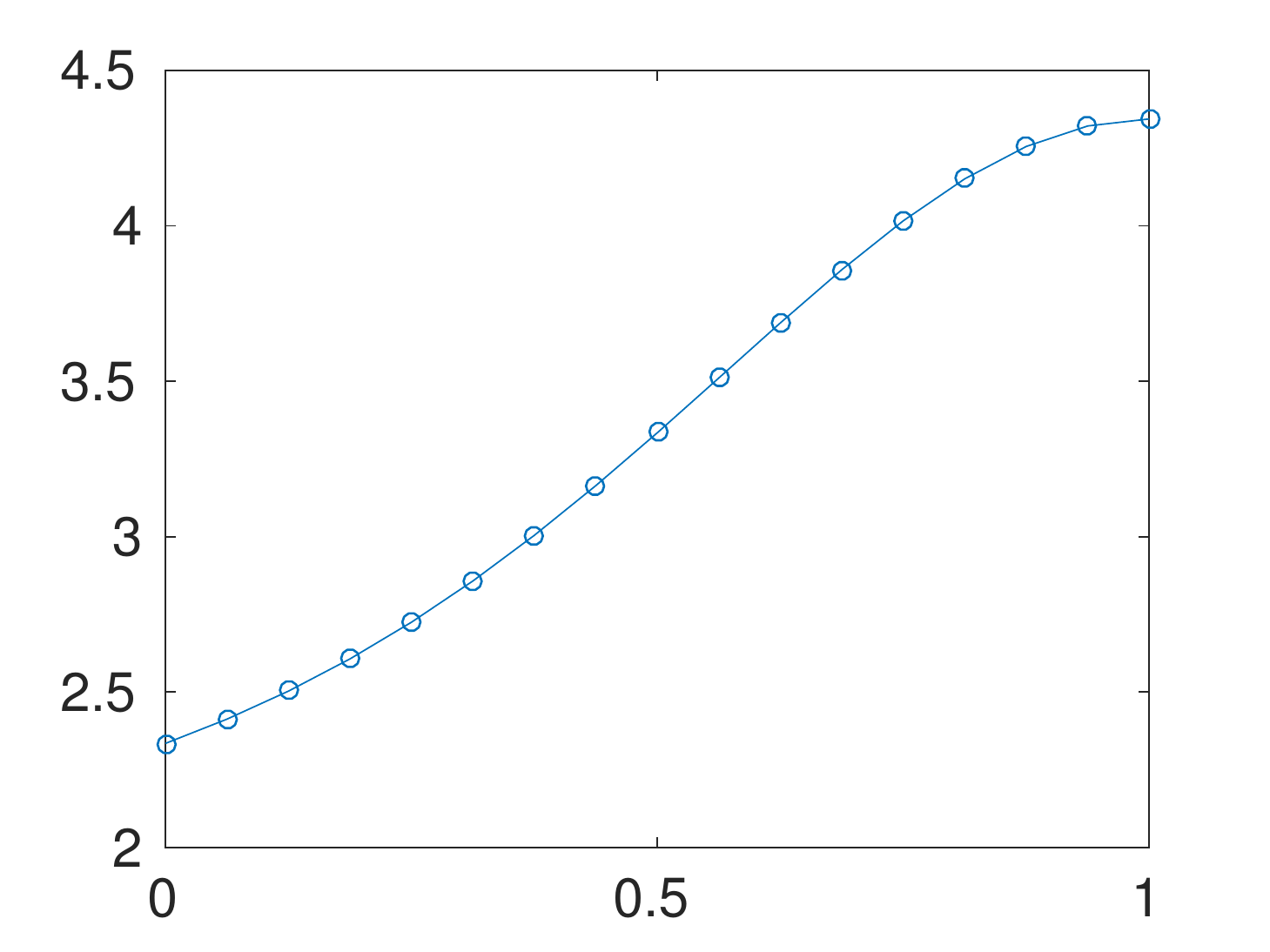}
  \caption*{$e_3$ at $t=1$}
  \label{fig:C3_t=1}
\end{subfigure}
\caption{Snapshots of the concentration $c_h(t)$ of the chemoattractant for the pipes $e_i$ 
at times $t=0,0.5,1$ obtained with meshsize $h=2^{-4}$.}
\label{fig:evolutionC}
\end{figure}

\begin{figure}[ht!]
\begin{subfigure}{.32\textwidth}
  \centering
  \includegraphics[width=.9\linewidth]{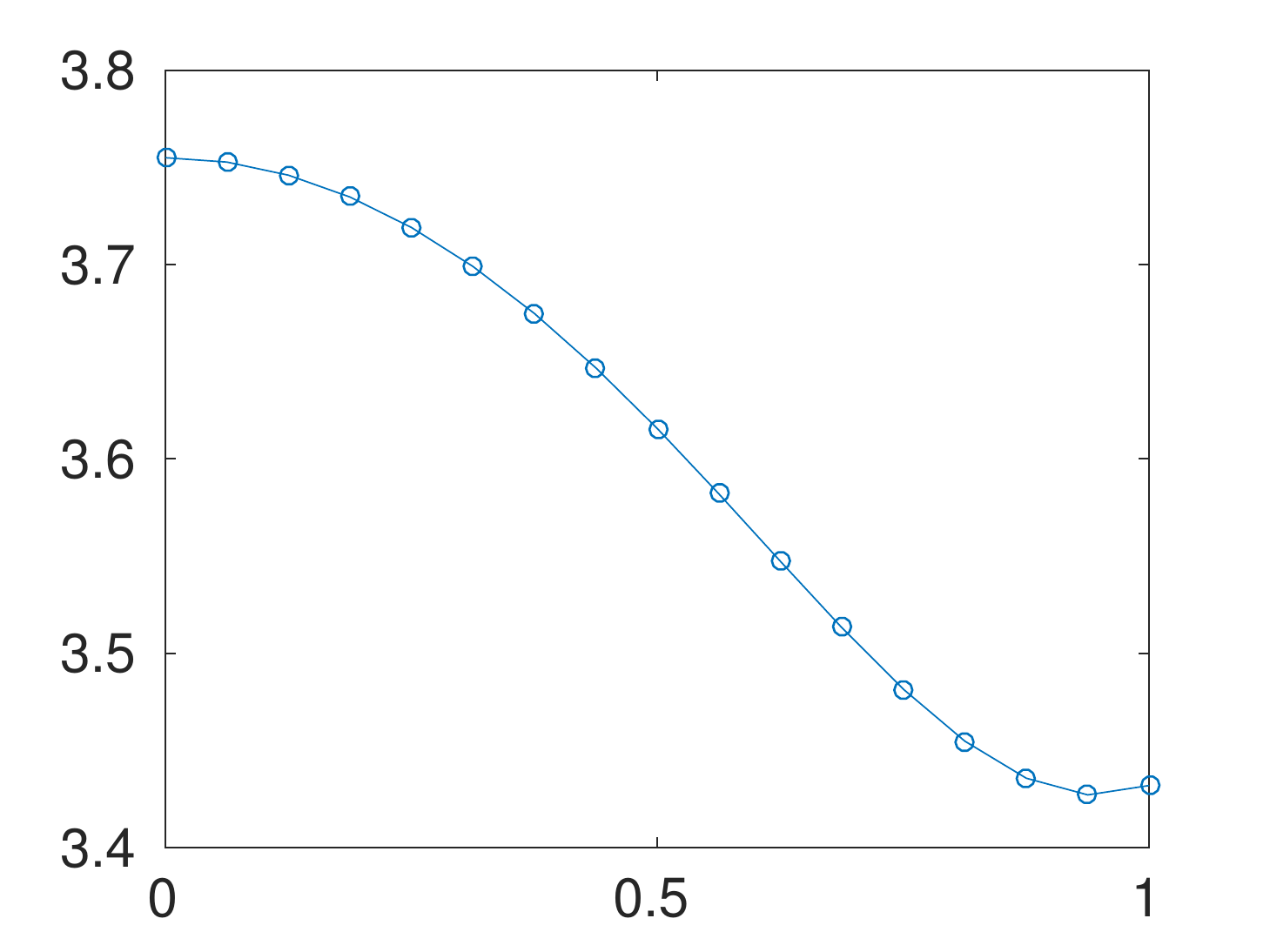}
  \caption*{$e_1$ at $t=0.5$}
  \label{fig:U1_t=0.5}
\end{subfigure}%
\begin{subfigure}{.32\textwidth}
  \centering
  \includegraphics[width=.9\linewidth]{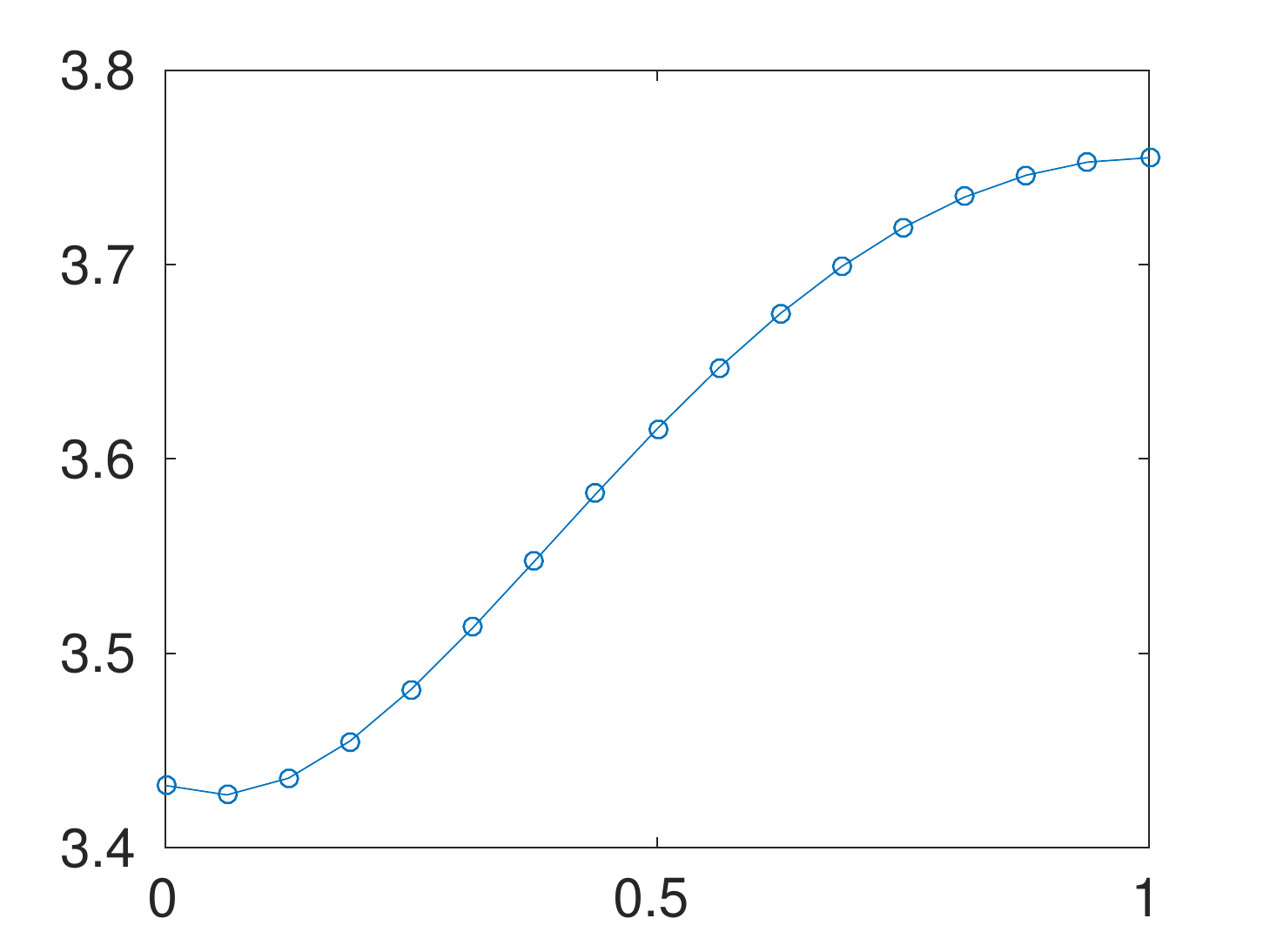}
  \caption*{$e_2$ at $t=0.5$}
  \label{fig:U2_t=0.5}
\end{subfigure} 
\begin{subfigure}{.32\textwidth}
  \centering
  \includegraphics[width=.9\linewidth]{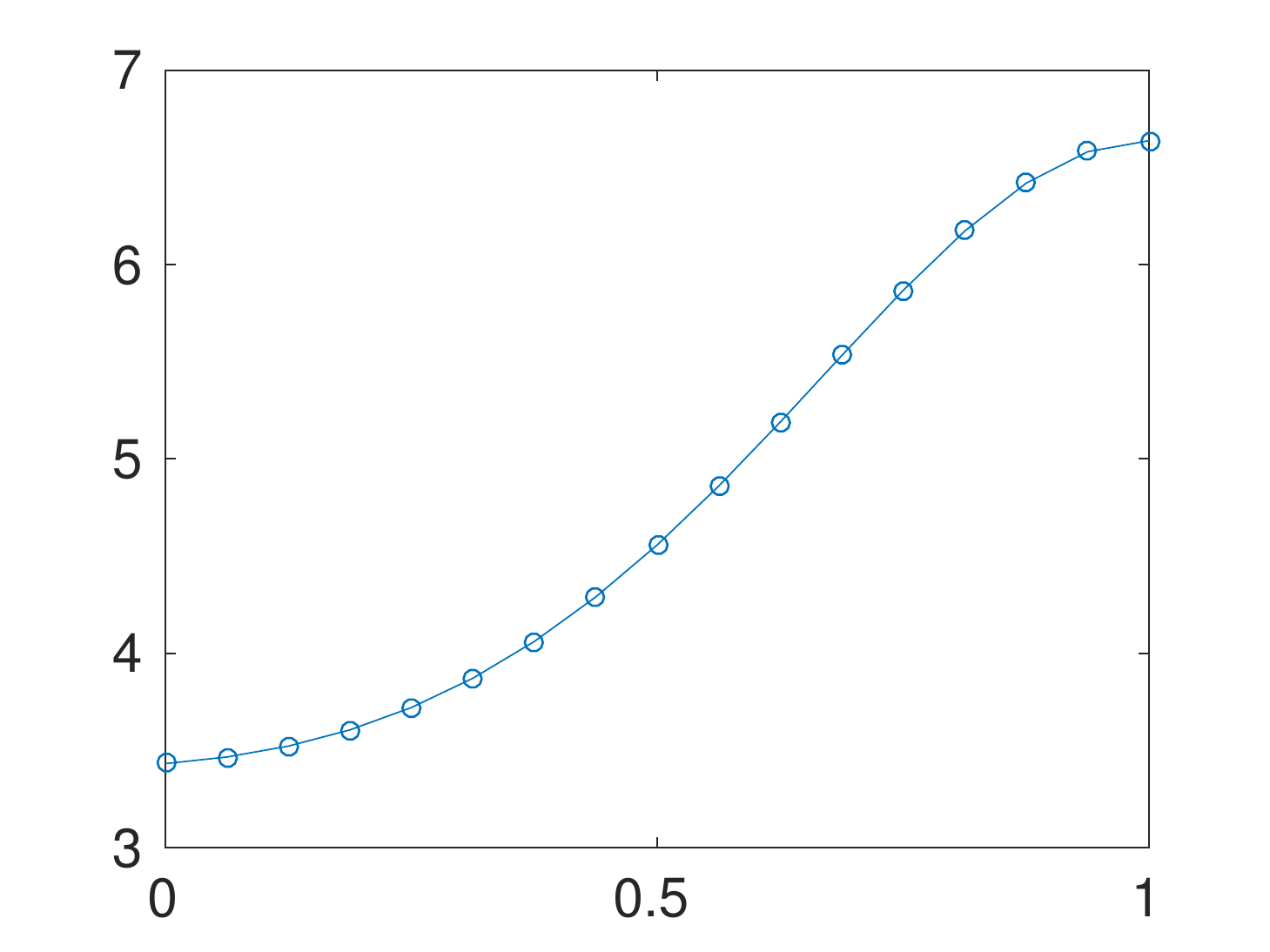}
  \caption*{$e_3$ at $t=0.5$}
  \label{fig:U3_t=0.5}
\end{subfigure} 
\begin{subfigure}{.32\textwidth}
  \centering
  \includegraphics[width=.9\linewidth]{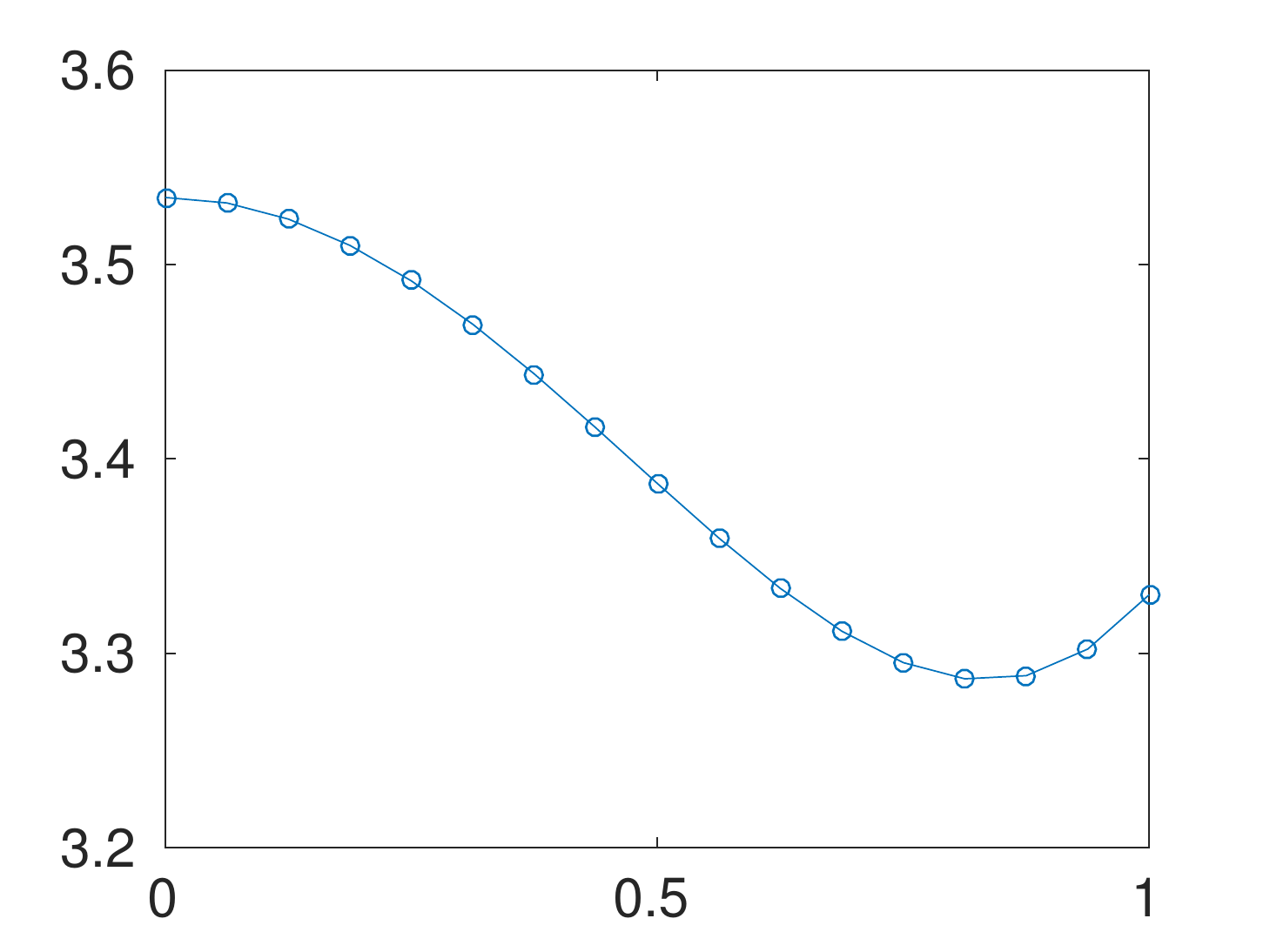}
  \caption*{$e_1$ at $t=1$}
  \label{fig:U1_t=1}
\end{subfigure}
\begin{subfigure}{.32\textwidth}
  \centering
  \includegraphics[width=.9\linewidth]{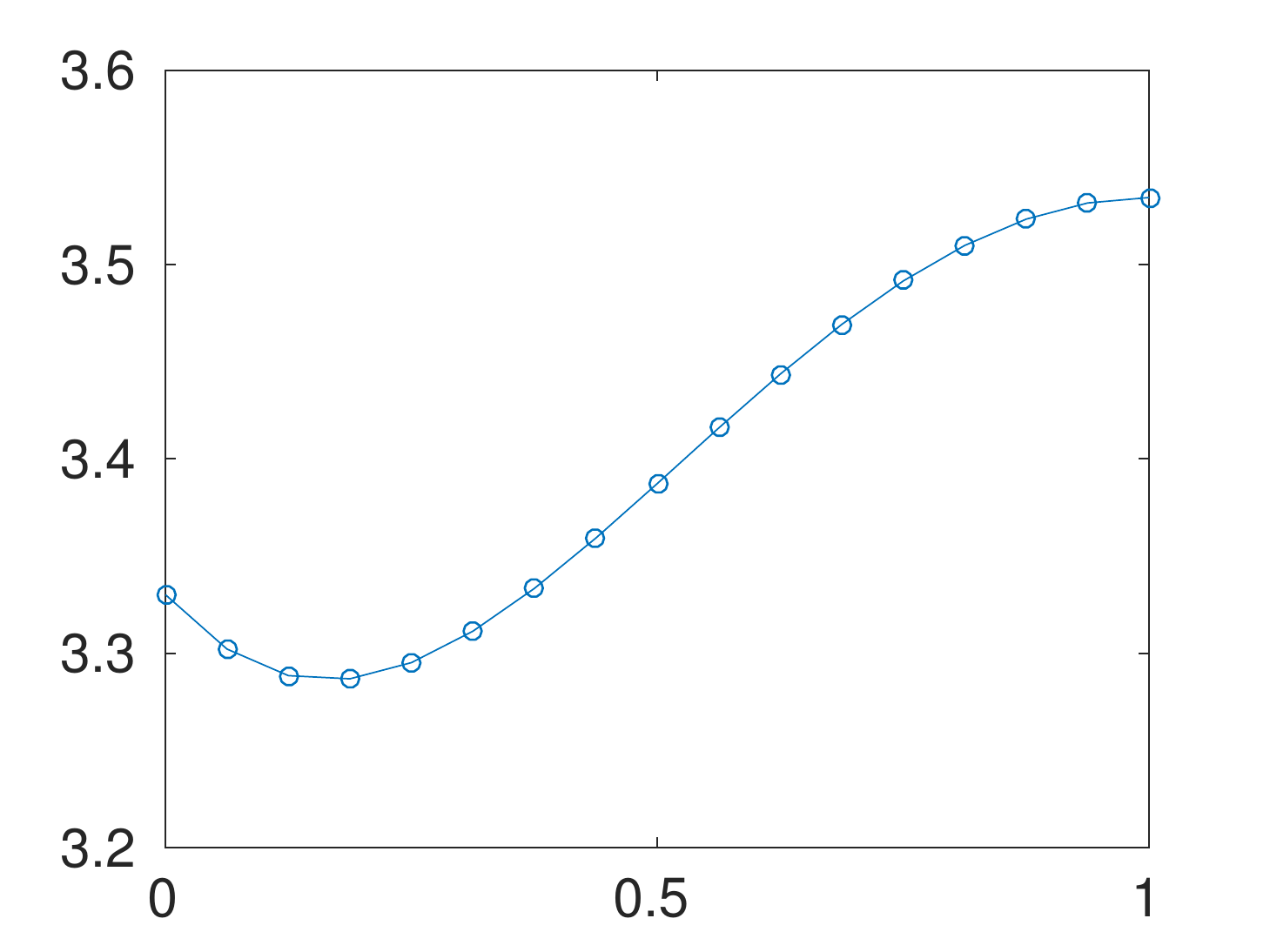}
  \caption*{$e_2$ at $t=1$}
  \label{fig:U2_t=1}
\end{subfigure}
\begin{subfigure}{.32\textwidth}
  \centering
  \includegraphics[width=.9\linewidth]{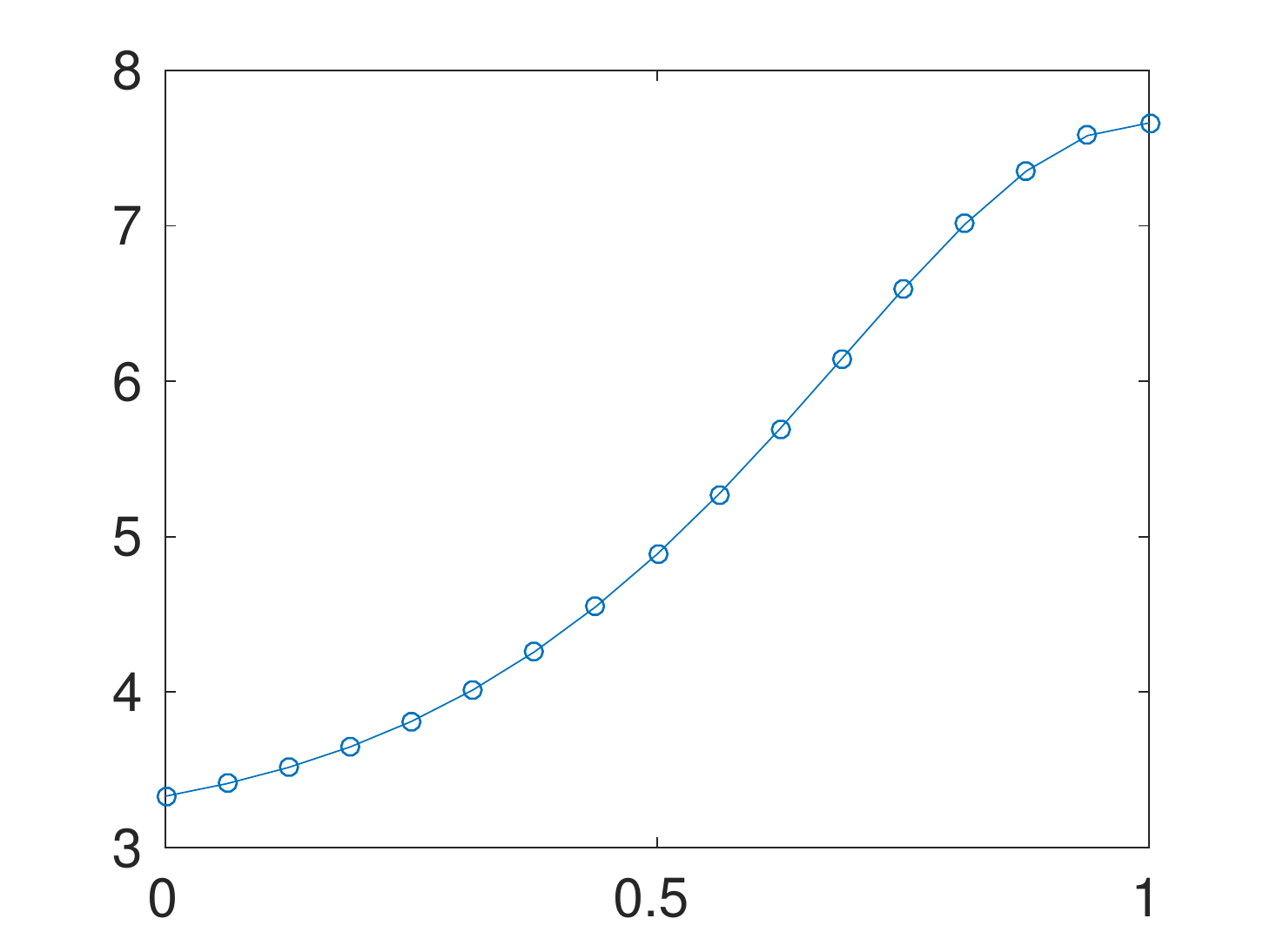}
  \caption*{$e_3$ at $t=1$}
  \label{fig:U3_t=1}
\end{subfigure}
\caption{Snapshots of the population density $u_h(t)$ for the pipes $e_i$ at times $t=0.5,1$ obtained with meshsize $h=2^{-4}$. At time $t=0$ (not shown) the solution has the constant value $4$.
\label{fig:evolutionU}}
\end{figure}

In order to verify the theoretical convergence rates obtained in section~\ref{sec:rates},
we proceed as follows: Since no analytical solution is available for this test case, we use the difference between the numerical solutions on two different meshes with meshsize $H$ and $h=H/2$ as an approximation for the actual error. Since the method has been proven to converge with a given rate, this yields accurate approximations for the true error. 
The results obtained in our numerical tests are presented in Table~\ref{tab:1} and \ref{tab:2}.
\begin{table}[ht!]
\begin{center}
\begin{tabular}{||c|c||c|c||c|c||c||}
$h$ & $\tau$ & $\|u_h-u_H\|_{L^{\infty}(0,1;L^2(\E))}$ & eoc & $\|u_h-u_H\|_{L^2(0,1;H^1(\E))}$ & eoc \\
\hline
\hline
$2^{-7}$&$2^{-10}$&0.096656&--&0.008712&-- \\
$2^{-8}$&$2^{-11}$&0.049195&0.97&0.004392&0.99 \\
$2^{-9}$&$2^{-12}$&0.024821&0.99&0.002205&0.99 \\
$2^{-10}$&$2^{-13}$&0.012467&1.00&0.001105&1.00 \\
$2^{-11}$&$2^{-14}$&0.006248&1.00&0.000553&1.00 \\
\end{tabular}
\smallskip
\caption{Numerically observed errors and estimated order of convergence for the population density $u_h$.\label{tab:1}}
\end{center}
\end{table}

\begin{table}[ht!]
\begin{center}
\begin{tabular}{||c|c||c|c||c|c||c||}
$h$ & $\tau$ &  $\|c_h-c_H\|_{L^{\infty}(0,1;L^2(\E))}$ & eoc & $\|c_h-c_H\|_{L^2(0,1; H^1(\E))}$ & eoc\\
\hline
\hline
$2^{-7}$&$2^{-10}$&0.027456&--&0.002500&-- \\
$2^{-8}$&$2^{-11}$&0.013808&0.99&0.001252&1.00 \\
$2^{-9}$&$2^{-12}$&0.006947&0.99&0.000627&1.00 \\
$2^{-10}$&$2^{-13}$&0.003499&0.99&0.000313&1.00 \\
$2^{-11}$&$2^{-14}$&0.001766&0.99&0.000156&1.00 \\
\end{tabular}
\smallskip
\caption{Numerically observed errors and estimated order of convergence for the concentration $c_h$ of the chemoattractant.\label{tab:2}}
\end{center}
\end{table}

As predicted by our theoretical results in Section~\ref{sec:rates}, 
we observe first order of convergence in both solution components and norms used for our analysis.

\subsection{Block network}
As a second test case, we consider an example proposed in \cite{BorscheGoettlichKlarSchillen14}, 
namely the block network depicted in Figure~\ref{fig:block}. 
\begin{figure}[ht!]
\hspace*{-2.5em}
\tiny
\begin{tikzpicture}[scale=0.7]
\node[circle,draw,inner sep=2pt,minimum size=0.75cm] (v0) at (-1.414,7.414) {$v_0$};
\node[circle,draw,inner sep=2pt,minimum size=0.75cm] (v1) at (0,6) {$v_1$};
\node[circle,draw,inner sep=2pt,minimum size=0.75cm] (v2) at (2,6) {$v_2$};
\node[circle,draw,inner sep=2pt,minimum size=0.75cm] (v3) at (4,6) {$v_3$};
\node[circle,draw,inner sep=2pt,minimum size=0.75cm] (v4) at (6,6) {$v_4$};
\node[circle,draw,inner sep=2pt,minimum size=0.75cm] (v5) at (0,4) {$v_5$};
\node[circle,draw,inner sep=2pt,minimum size=0.75cm] (v6) at (2,4) {$v_6$};
\node[circle,draw,inner sep=2pt,minimum size=0.75cm] (v7) at (4,4) {$v_7$};
\node[circle,draw,inner sep=2pt,minimum size=0.75cm] (v8) at (6,4) {$v_8$};
\node[circle,draw,inner sep=2pt,minimum size=0.75cm] (v9) at (0,2) {$v_9$};
\node[circle,draw,inner sep=2pt,minimum size=0.75cm] (v10) at (2,2) {$v_{10}$};
\node[circle,draw,inner sep=2pt,minimum size=0.75cm] (v11) at (4,2) {$v_{11}$};
\node[circle,draw,inner sep=2pt,minimum size=0.75cm] (v12) at (6,2) {$v_{12}$};
\node[circle,draw,inner sep=2pt,minimum size=0.75cm] (v13) at (0,0) {$v_{13}$};
\node[circle,draw,inner sep=2pt,minimum size=0.75cm] (v14) at (2,0) {$v_{14}$};
\node[circle,draw,inner sep=2pt,minimum size=0.75cm] (v15) at (4,0) {$v_{15}$};
\node[circle,draw,inner sep=2pt,minimum size=0.75cm] (v16) at (6,0) {$v_{16}$};
\node[circle,draw,inner sep=2pt,minimum size=0.75cm] (v17) at (7.414,-1.414) {$v_{17}$};
\draw[->,red,dashed,thick,line width=1.5pt] (v0) -- node[above,sloped] {$e_1$} ++(v1);
\draw[->,thick,line width=1.5pt] (v1) -- node[above] {$e_2$} ++(v2);
\draw[->,thick,line width=1.5pt] (v2) -- node[above] {$e_3$} ++(v3);
\draw[->,thick,line width=1.5pt] (v3) -- node[above] {$e_4$} ++(v4);
\draw[->,red,dashed,thick,line width=1.5pt] (v1) -- node[right] {$e_5$} ++(v5);
\draw[->,thick,line width=1.5pt] (v2) -- node[right] {$e_6$} ++(v6);
\draw[->,thick,line width=1.5pt] (v3) -- node[right] {$e_7$} ++(v7);
\draw[->,thick,line width=1.5pt] (v4) -- node[right] {$e_8$} ++(v8);
\draw[->,red,dashed,thick,line width=1.5pt] (v5) -- node[above] {$e_9$} ++(v6);
\draw[->,thick,line width=1.5pt] (v6) -- node[above] {$e_{10}$} ++(v7);
\draw[->,thick,line width=1.5pt] (v7) -- node[above] {$e_{11}$} ++(v8);
\draw[->,line width=1.5pt] (v5) -- node[right] {$e_{12}$} ++(v9);
\draw[->,red,dashed,thick,line width=1.5pt] (v6) -- node[right] {$e_{13}$} ++(v10);
\draw[->,thick,line width=1.5pt] (v7) -- node[right] {$e_{14}$} ++(v11);
\draw[->,thick,line width=1.5pt] (v8) -- node[right] {$e_{15}$} ++(v12);
\draw[->,thick,line width=1.5pt] (v9) -- node[above] {$e_{16}$} ++(v10);
\draw[->,red,dashed,thick,line width=1.5pt] (v10) -- node[above] {$e_{17}$} ++(v11);
\draw[->,thick,line width=1.5pt] (v11) -- node[above] {$e_{18}$} ++(v12);
\draw[->,thick,line width=1.5pt] (v9) -- node[right] {$e_{19}$} ++(v13);
\draw[->,thick,line width=1.5pt] (v10) -- node[right] {$e_{20}$} ++(v14);
\draw[->,red,dashed,thick,line width=1.5pt] (v11) -- node[right] {$e_{21}$} ++(v15);
\draw[->,thick,line width=1.5pt] (v12) -- node[right] {$e_{22}$} ++(v16);
\draw[->,thick,line width=1.5pt] (v13) -- node[above] {$e_{23}$} ++(v14);
\draw[->,thick,line width=1.5pt] (v14) -- node[above] {$e_{24}$} ++(v15);
\draw[->,red,dashed,thick,line width=1.5pt] (v15) -- node[above] {$e_{25}$} ++(v16);
\draw[->,red,dashed,thick,line width=1.5pt] (v16) -- node[above,sloped] {$e_{26}$} ++(v17);
\end{tikzpicture}
\caption{Block network. The shortest path between the vertices $v_0$ and $v_{17}$ is given by the edges $e_1,e_5,e_9,e_{13},e_{17},e_{21},e_{25},e_{26}$.\label{fig:block}} 
\end{figure}
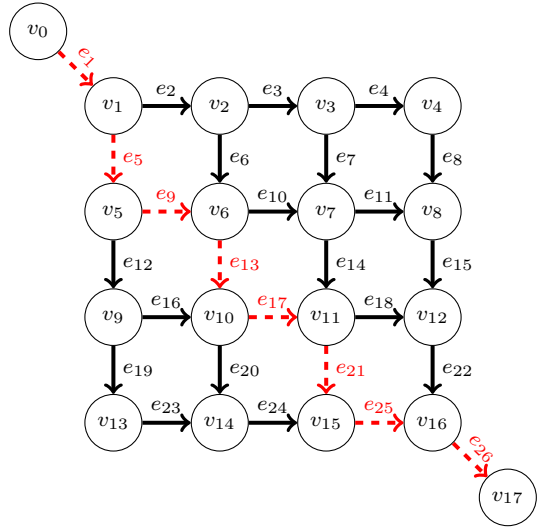
The set of edges $\E= \{e_1,\ldots,e_{26}\}$ is here partitioned into two disjoint subsets 
\begin {align*} 
\E_1 = \{ e_1,e_5,e_9,e_{13},e_{17},e_{21},e_{25},e_{26}\} 
\qquad \text{and} \qquad 
\E_2 = \E \setminus \E_1
\end {align*}
of pipes with different physical properties. The lengths of the pipes are chosen as $l_i = l_{e_i} = 1$ for all $i\in \E$, but the model parameters are chosen differently by
\begin{align*}
	\alpha_i,\beta_i,\chi_i = 100,\ e_i \in \E_1, \qquad \alpha_i,\beta_i,\chi_i = 1,\ e_i \in \E_2, \qquad \gamma_i,\delta_i = 0.1,\ \forall e_i \in \E.
\end{align*}
The initial values for the two solution components are simply chosen as
\begin{align*} 
u_{i,0}(x) = 1 \qquad \text{and} \qquad c_{i,0}(x) = 0, \qquad e \in \E. 
\end{align*}
In order to get an interesting behavior, the boundary conditions at the two ports of the network are chose as follows: 
\begin{alignat*}{4}
	\alpha \dn u(v_0,t) - \chi \dn c(v_0,t) u(v_0,t) &= \frac{2}{1+u(v_0,t)},
	       \qquad & \beta \dn c(v_0,t) &= 0, \\
	\alpha \dn u(v_{17},t) - \chi \dn c(v_{17},t) u(v_{17},t) &= 0, 
	       \qquad & \beta \dn c(v_{17},t) &= \frac{2}{1+c_{26}(v_{17},t)}.
\end{alignat*}
This means that the bacteria with density $u$ enter the network at node $v_0$ 
and they are expected to move towards $v_{26}$ where the chemoattractant $c$ is added. 
Due to the different properties of the individual pipes, we expect the bacteria to move along the red path highlighted in Figure~\ref{fig:block}. 
Some snapshots of the solution $u_h(t)$ and $c_h(t)$ computed with meshsize $h=2^{-5}$ and time step $\tau = 2^{-7}$ are shown in Figures~\ref{fig:blockU} and \ref{fig:blockC}.

\begin{figure}[ht!]
\begin{subfigure}{.35\textwidth}
  \centering
  \includegraphics[width=1\linewidth]{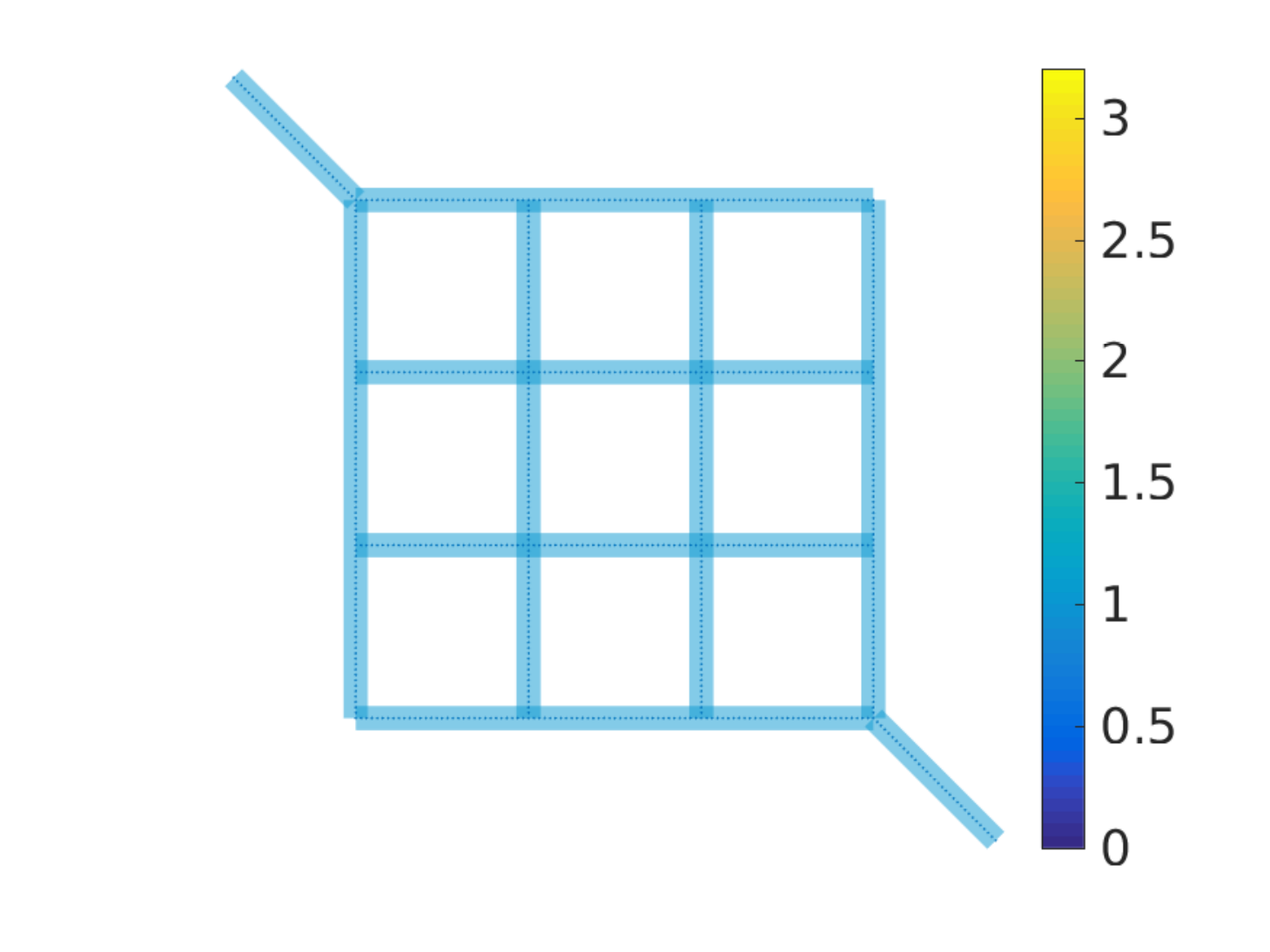}
  \vskip-1em
  \caption*{$t=0$}
  \vskip1em
  \label{fig:blockU_t=0}
\end{subfigure}%
\begin{subfigure}{.35\textwidth}
  \centering
  \includegraphics[width=1\linewidth]{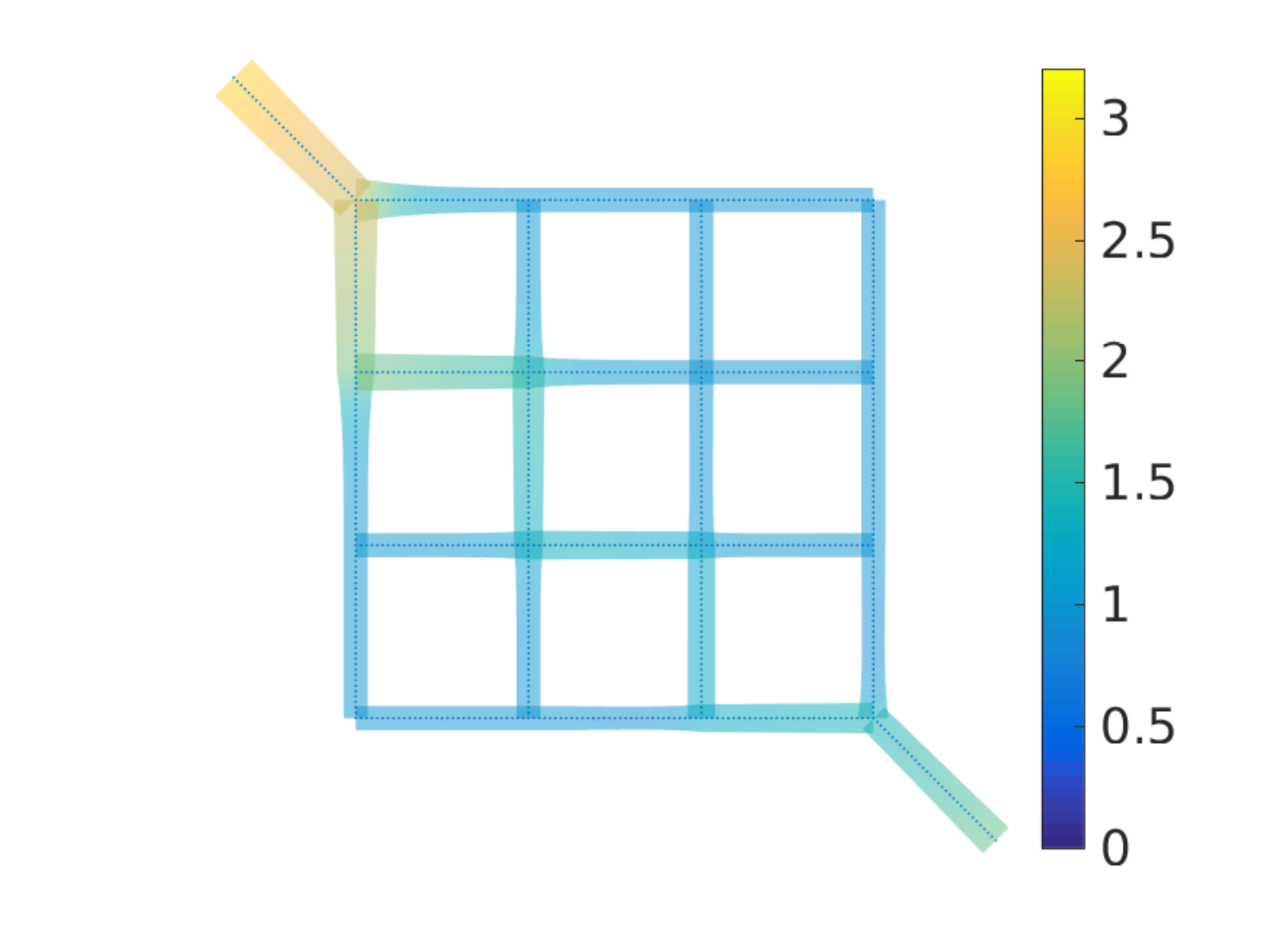}
  \vskip-1em
  \caption*{$t=10$}
  \vskip1em
  \label{fig:blockU_t=10}
\end{subfigure} \\
\begin{subfigure}{.35\textwidth}
  \centering
  \includegraphics[width=.9\linewidth]{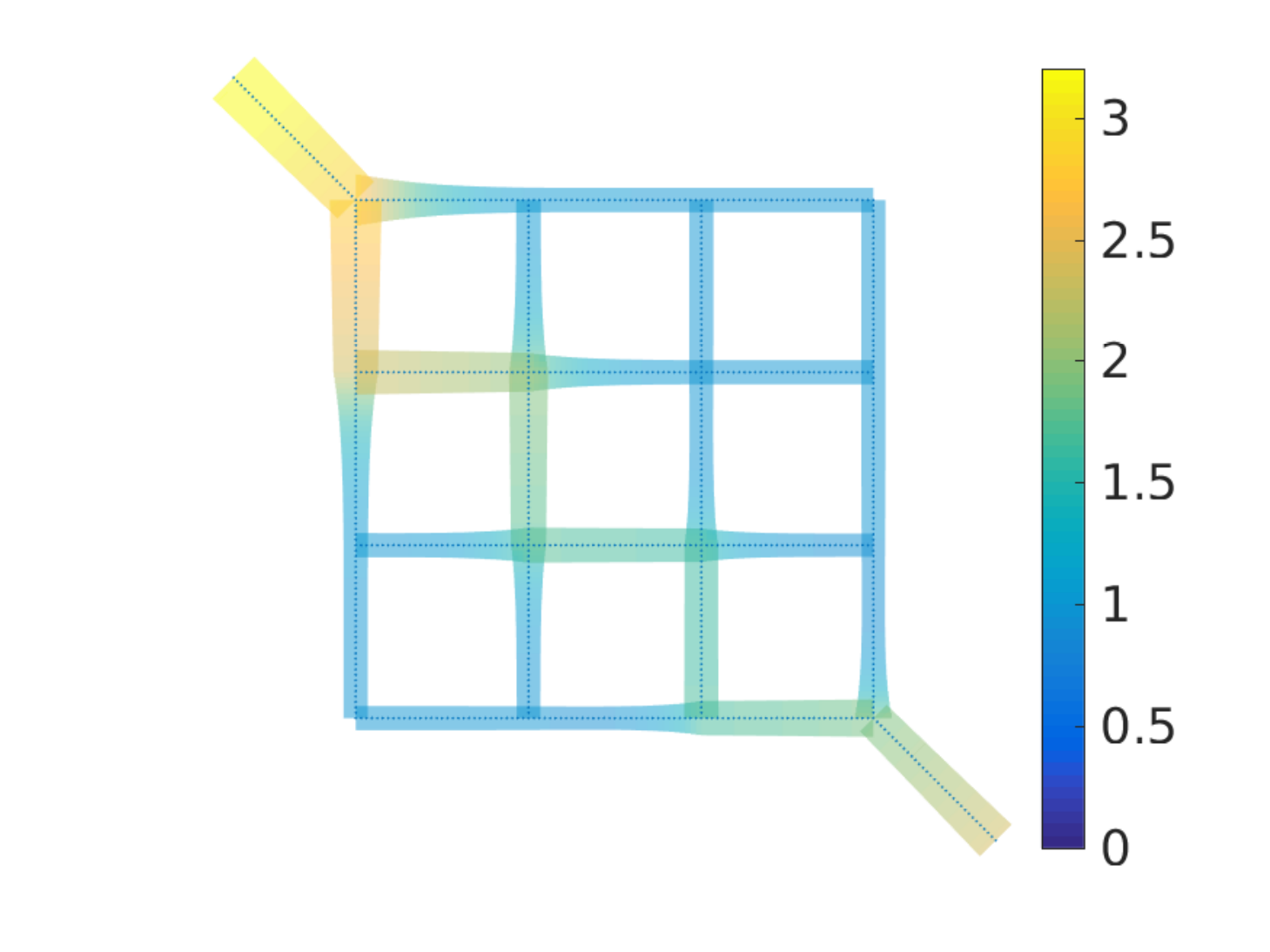}
  \vskip-1em
  \caption*{$t=20$}
  \vskip1em
  \label{fig:blockU_t=20}
\end{subfigure}
\begin{subfigure}{.35\textwidth}
  \centering
  \includegraphics[width=.9\linewidth]{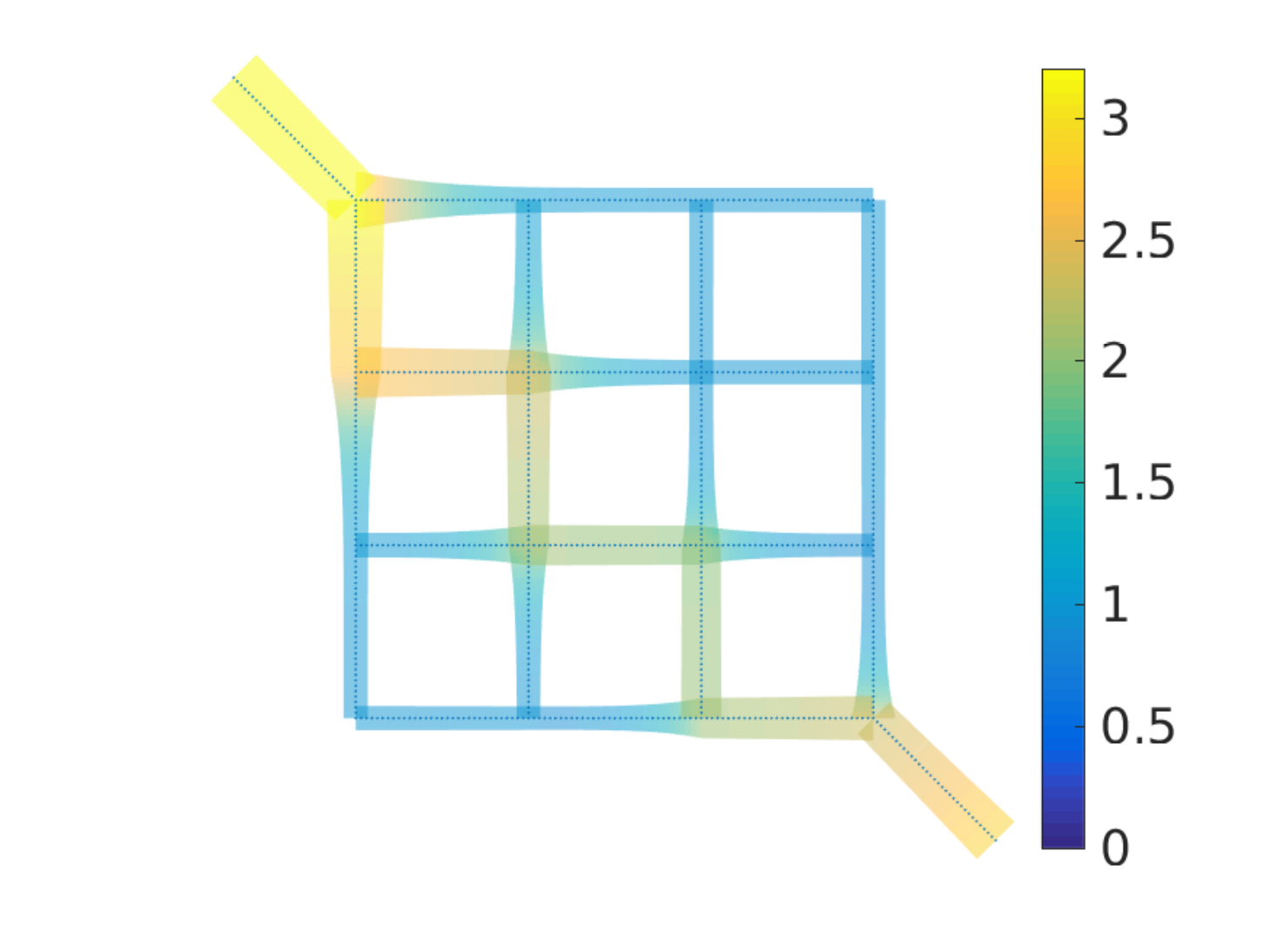}
  \vskip-1em
  \caption*{$t=30$}
  \vskip1em
  \label{fig:blockU_t=30}
\end{subfigure}
\caption{Snapshots of the population density $u_h(t)$ for $t=0,10,20,30$.}
\label{fig:blockU}
\end{figure}
\begin{figure}[ht!]
\begin{subfigure}{.35\textwidth}
  \centering
  \includegraphics[width=1\linewidth]{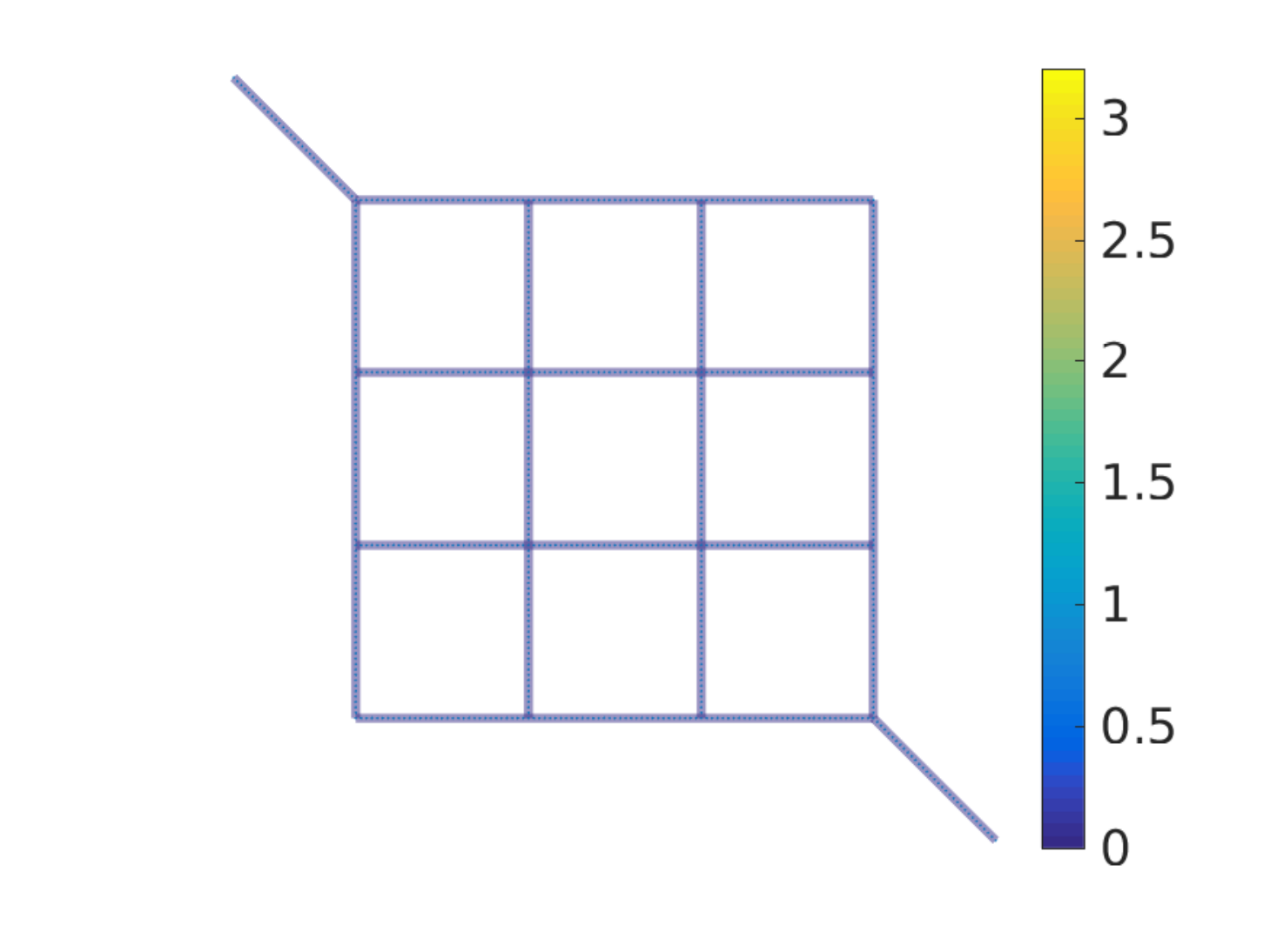}
  \vskip-2em
  \caption*{$t=0$}
  \vskip1em
  \label{fig:blockC_t=0}
\end{subfigure}%
\begin{subfigure}{.35\textwidth}
  \centering
  \includegraphics[width=1\linewidth]{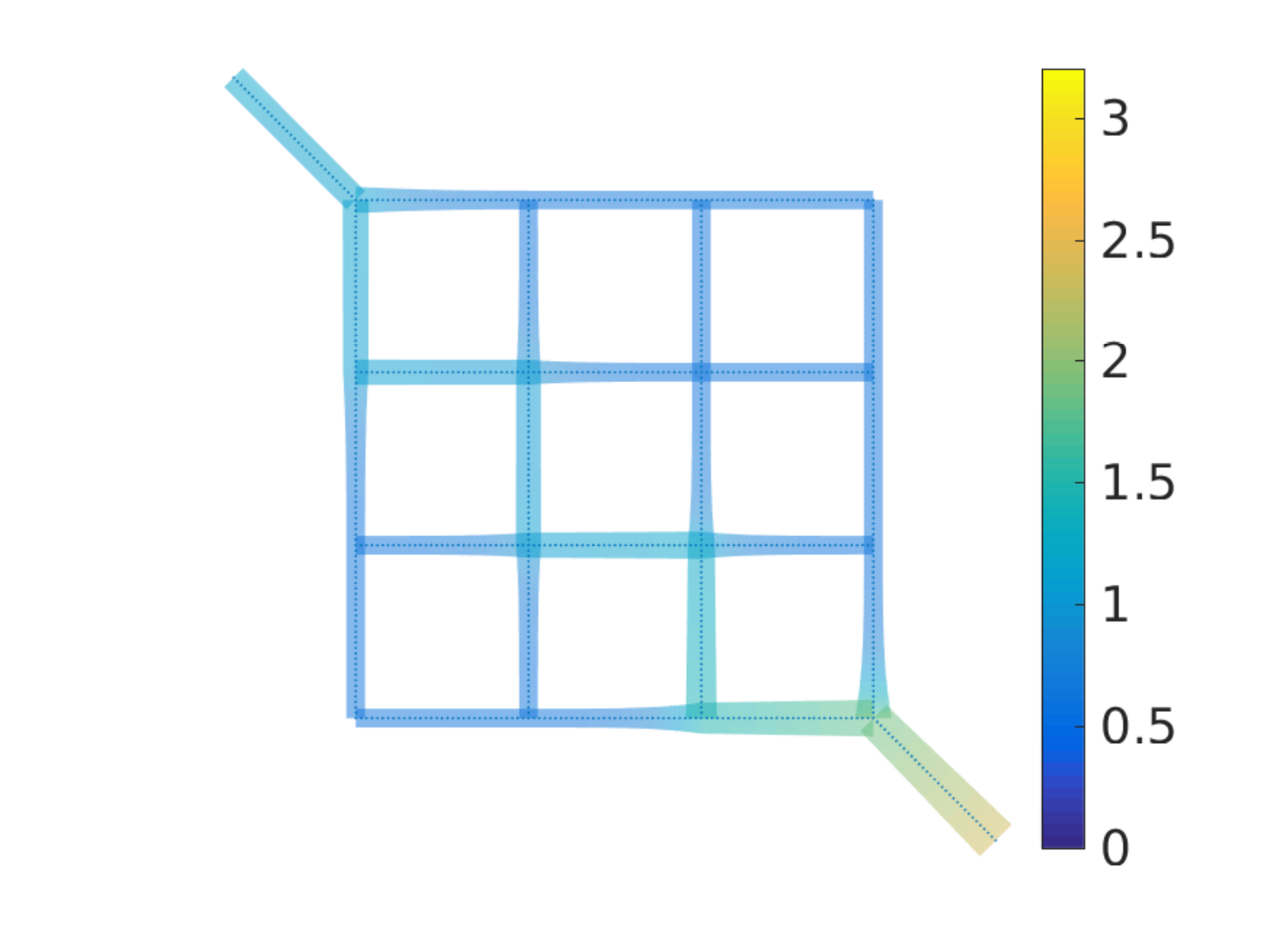}
  \vskip-2em
  \caption*{$t=10$}
  \vskip1em
  \label{fig:blockC_t=10}
\end{subfigure} \\
\begin{subfigure}{.35\textwidth}
  \centering
  \includegraphics[width=.9\linewidth]{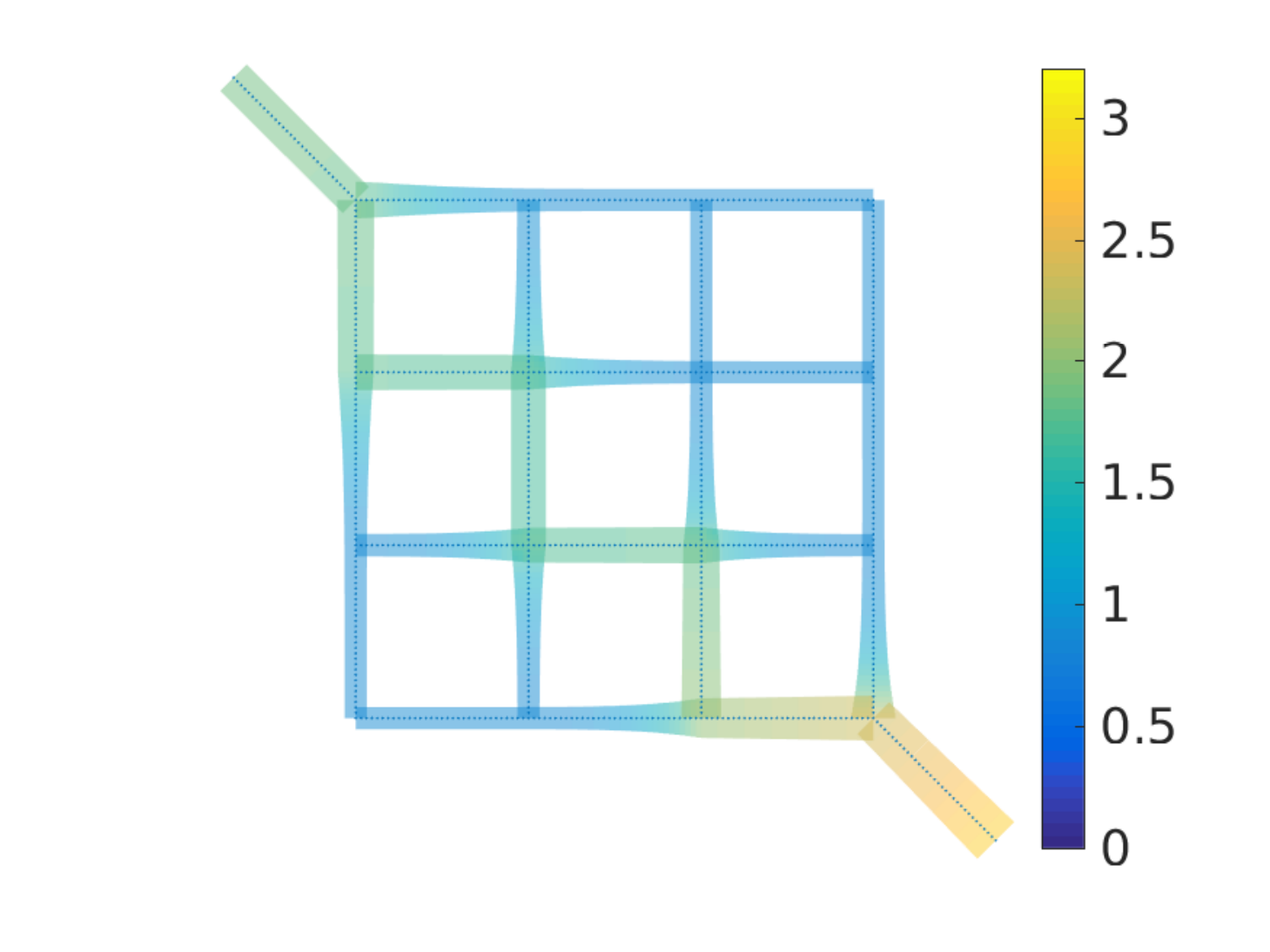}
  \vskip-2em
  \caption*{$t=20$}
  \vskip1em
  \label{fig:blockC_t=20}
\end{subfigure}
\begin{subfigure}{.35\textwidth}
  \centering
  \includegraphics[width=.9\linewidth]{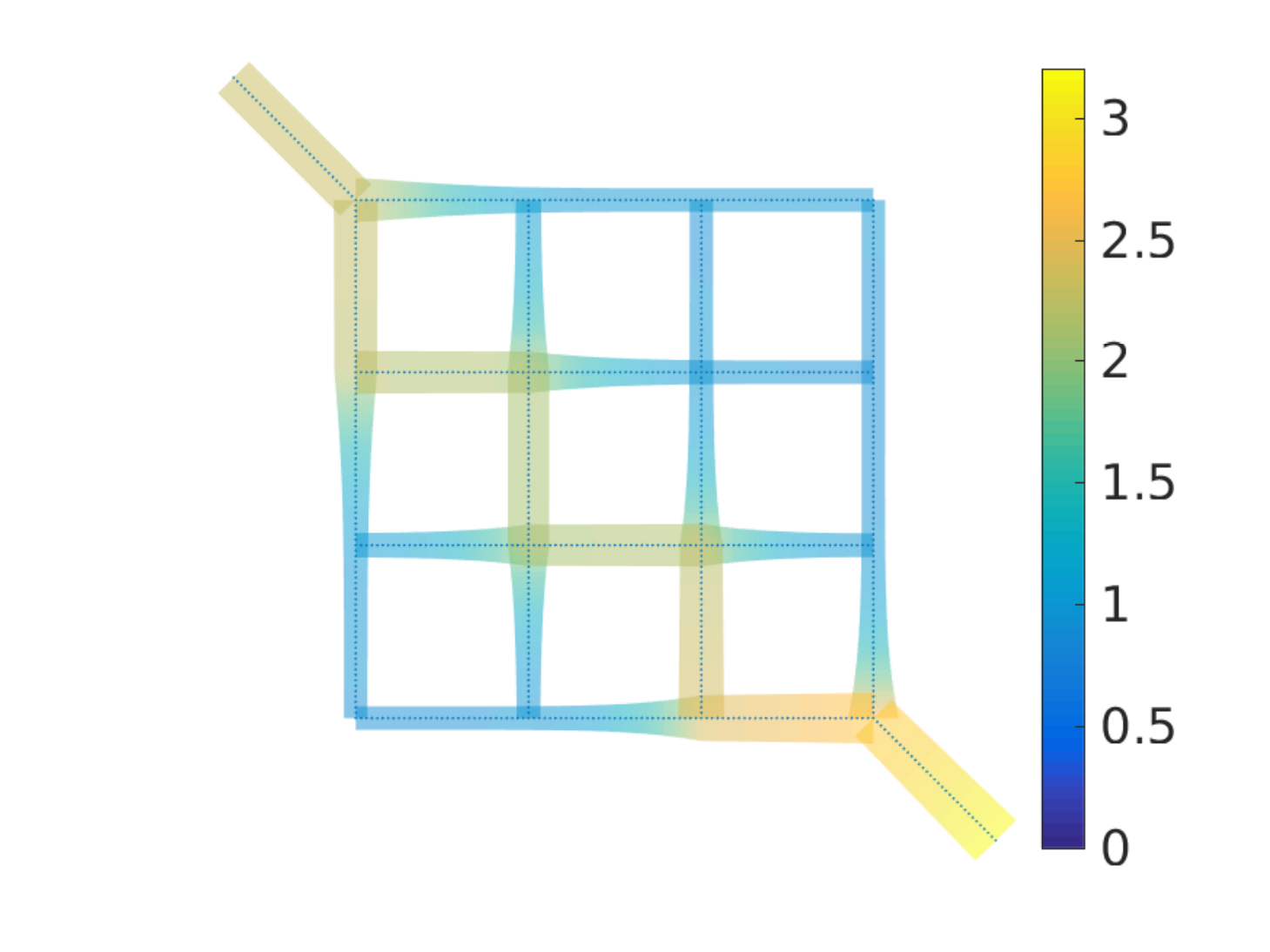}
  \vskip-2em
  \caption*{$t=30$}
  \vskip1em
  \label{fig:blockC_t=30}
\end{subfigure}
\caption{Snapshots of the concentration $c_h(t)$ for $t=0,10,20,30$.}
\label{fig:blockC}
\end{figure}

As can be seen from the images, both solution components are smooth and the system behavior seems to be computed correctly as expected.

\section{Summary and discussion} \label{sec:discussion}

In this paper, we considered a chemotaxis model on networks described by a system of partial differential-algebraic equations, which can be seen as a natural generalization of the minimal model to the network context.
By extending the results of Osaki and Yagi, we were able to establish global existence and uniqueness of solutions
based on perturbation arguments for semilinear evolution problems, conservation of mass, and positivity of the solutions. 
In addition, we derived regularity estimates for the solutions.

The arguments of the proofs were developed in a way that allowed us to prove the global existence and uniqueness of solutions also for numerical approximations obtained by a finite element discretization with mass lumping and upwinding and an implicit Euler method for time integration. 
The discrete approximations could be shown to converge to the unique global solution of the chemotaxis problem 
without artificial smoothness requirements on the solution. In addition, we could establish order optimal 
convergence rates under minimal smoothness assumptions.

The arguments used for the analysis of the finite element method presented in this paper 
may be used to improve also the convergence results of the method of Saito~\cite{Saito12},
in particular, to get rid of the strong stepsize restrictions needed there. 
Also a generalization to chemotaxis models with nonlinear coefficients seems possible to some extent. 
A class of problem that would certainly deserve further considerations, also from a numerical point of view, 
are models of haptotaxis, where no diffusion is present in the equation governing the chemoattractant.

\section*{Acknowledgements}

The authors are grateful for financial support by the German Research Foundation (DFG) via grants IRTG~1529, TRR~146, and TRR~154 and by the ``Excellence Initiative'' of the German Federal and State Governments via the Graduate School of Computational Engineering GSC~233 at Technische Universität Darmstadt.

\appendix

\section*{Appendix}

In the following sections, we present some auxiliary results that were used in our analysis and we 
provide complete proofs for some results mentioned in the paper. 

\section{Auxiliary results}

The following embedding inequalities are used several times in our proofs.
\begin{lemma} \label{lem:interpolation}
Let $(\V,\E,\ell)$ be a finite metric graph. Then 
for any $\eps>0$,
\begin{align}
\|f\|_{L^\infty(\E)} 
  &\le C_G (\|f\|_{L^1(\E)}^{1/3} \|\dx f\|_{L^2(\E)}^{2/3} + \|f\|_{L^1(\E)}) \notag \\
   &\le \eps \|\dx f\|_{L^2(\E)} + (C_G+\tfrac{4 C_G^3}{27\eps^2}) \|f\|_{L^1(\E)} 
 \label{interpolationL1}, \\
\|f\|_{L^\infty(\E)} 
   &\le C_G (\|f\|_{L^2(\E)}^{1/2} \|\dx f\|_{L^2(\E)}^{1/2} + \|f\|_{L^2(\E)}) \notag \\
   &\le \eps \|\dx f\|_{L^2(\E)} + (C_G+\tfrac{C_G^2}{4\eps}) \|f\|_{L^2(\E)}   \notag \\
    &\le \eps \|\dx f\|_{L^2(\E)} + C_G(1+\tfrac{1}{\eps}) \|f\|_{L^2(\E)} 
\label{interpolationL2}, 
\end{align}
for all $f \in H^1(\E)$ with a constant $C_G$ only depending on the geometry of the graph.
\end{lemma}
\begin{proof}
The assertions follow by applying the classical Gagliardo-Nirenberg inequality on every edge, 
summation over all edges, using the equivalence of the $l^{p}$-norms on finite sequences, 
and application of Young's inequality.
\end{proof}

We also require the following continuous and discrete version of Gronwall's inequality.
\begin{lemma}[Gronwall's lemma, see \cite{Evans10}]\label{lem:gronwall}
Assume $\eta$ is a nonnegative, absolutely continuous function and $\zeta, \phi, \psi$ are nonnegative, integrable functions on $[0,T]$ such that
\[ \eta'(t) + \zeta(t) \leq \phi(t) \eta(t) + \psi(t), \quad \text{a.a. } t\in [0,T]. \]
Then
\[ \eta(t) + \int_0^t \zeta(s)~ds \leq e^{\int_0^t \phi(s)ds}\left(\eta(0) + \int_0^t \psi(s)~ds \right), \quad \text{a.a. } t\in [0,T]. \]
\end{lemma}

\begin{lemma}[Discrete Gronwall's lemma, see \cite{Heywood90}]\label{lem:discrete_gronwall}
Let $\tau,B$ and $a_n,b_n,c_n,\gamma_n$ for $n=0,\ldots, N$ be nonnegative numbers such that
\begin{align*} 
a_k  + \tau \sum_{n=0}^k b_n \leq \tau\sum_{n=0}^k \gamma_n a_n + \tau \sum_{n=0}^k c_n + B, \quad k=0,\ldots,N. 
\end{align*}
Further suppose that $\tau\gamma_{n} <1$ for $n=1,\ldots,N$ and set $\sigma_{n} \coloneqq (1-\tau\gamma_n)^{-1}$. 
Then
\begin{align} \label{eq:discrete_gronwall} 
a_k + \tau \sum_{n=0}^k b_n \leq \exp(\tau\sum_{n=0}^k \sigma_n \gamma_n) \Big[ \tau \sum_{n=0}^k c_n + B\Big],\quad k=0,\ldots,N. 
\end{align}
\end{lemma}

\section{Proof of Theorem~\ref{thm:local}}

For convenience of the reader, we now present a complete proof of Theorem~\ref{thm:local}.
As a first step, we consider the following linearized variational equations
\begin{alignat}{2}
\la\dt u(t),v\ra_\E + \la\alpha \dx u(t),\dx v\ra_\E &= \la\chi u(t) \dx c(t), \dx v\ra_\E \quad && \forall v \in H^1(\E), \label{eq:lin1}\\
\la\dt c(t),q\ra_\E + \la\beta \dx c(t),\dx q\ra_\E + \la\gamma c(t), q\ra_\E &= \la\delta z(t),q\ra_\E \quad && \forall q \in H^1(\E), \label{eq:lin2}
\end{alignat}
where $z \in L^\infty(0,T;L^2(\E))$ is some given function.
%
In the following lemmas, we summarize the main results about well-posedness of the corresponding initial value problem and derive a-priori estimates for its solutions. 

\begin{lemma} \label{lem:local1}
Let (A1) hold and $T>0$. Then for any $z \in L^\infty(0,T;L^2(\E))$ and any $c_0 \in L^2(\E)$, there exists a unique weak solution $c \in L^2(0,T;H^1(\E)) \cap H^1(0,T;H^1(\E)')$ of equation \eqref{eq:lin2} with initial value $c(0)=c_0$. Moreover, the solution can be bounded by
\begin{align*} 
\|c\|_{L^\infty(0,t;L^2(\E))} + \|c\|_{L^2(0,t;H^1(\E))} \le C_1 (\|c_0\|_{L^2(\E)} + t \|z\|_{L^\infty(0,t;L^2(\E))}).
\end{align*}
for a..e $0 \le t \le T$ with $C_1=C_1(\underline \beta)$. If $c_0 \in H^1(\E)$, then additionally
\begin{align*}
\|\dx c\|_{L^\infty(0,t;L^2(\E))} + \|\dt c\|_{L^2(0,t;L^2(\E))} \le C_2 (\| c_0\|_{H^1(\E)} + t^{1/2}\|z\|_{L^\infty(0,t;L^2(\E))})
\end{align*}
with constant $C_2=C_2(\underline\beta,\overline\beta,\overline\gamma,\overline\delta)$ depending only on the bounds in (A1).
\end{lemma}
\begin{proof}
Existence of a unique weak solution follows by Galerkin approximation and standard arguments; see \cite{DL5,Evans10} for details. To keep track of the constants, we give a short proof of the a-priori estimates. 
Testing the variational equation \eqref{eq:lin2} with $q=c(t)$ yields 
\begin{align*}
\frac{1}{2} \frac{d}{dt} \|c(t)\|_{L^2(\E)}^2 + \underline \beta \|\dx c(t)\|^2_{L^2(\E)} 
&\le \overline \delta \|c(t)\|_{L^2(\E)} \|z(t)\|_{L^2(\E)} \\
&\le \frac{\eps}{2} \|c(t)\|^2_{L^2(\E)}  + \frac{\overline \delta^2}{2\eps} \|z(t)\|_{L^2(\E)}^2. 
\end{align*}
Here we used the bounds for the coefficients for the first step and Young's inequality with $\eps>0$ for the second. 
An application of Lemma~\ref{lem:gronwall} with $\eps=1/t$ further yields 
\begin{align*}
\|c(t)\|^2_{L^2(\E)} + \underline \beta \|\dx c\|^2_{L^2(0,t;L^2(\E))} 
&\le e^1 \left(\|c_0\|_{L^2(\E)}^2 + t \|z\|_{L^2(0,t;L^2(\E))}^2\right). 
\end{align*}
The first estimate then follows by noting that $\|z\|_{L^2(0,t;L^2(\E))}^2 \le t \|z\|_{L^\infty(0,t;L^2(\E))}^2$, which follows by the Cauchy-Schwarz inequality, and some further elementary computations. 
For the second bound of the lemma, we test \eqref{eq:lin2} with $q = \dt c(t)$ which yields 
\begin{align*}
&\|\dt c(t)\|^2_{L^2(\E)} + \frac{1}{2}\frac{d}{dt} \|\beta^{1/2} \dx c(t)\|^2_{L^2(\E)} + \frac{1}{2} \frac{d}{dt} \|\gamma^{1/2} c(t)\|^2_{L^2(\E)} \\
&\le \overline \delta \|z(t)\|_{L^2(\E)} \|\dt c(t)\|_{L^2(\E)}
\le \frac{\overline \delta^2}{2} \|z(t)\|_{L^2(\E)}^2 + \frac{1}{2} \|\dt c(t)\|_{L^2(\E)}^2.
\end{align*}
By rearranging the terms, integration in time, and using assumption (A1),
we obtain
\begin{align*}
\|\dt c\|^2_{L^2(0,t;L^2(\E))} + \underline \beta \|\dx c(t)\|^2_{L^2(\E)} 
\le \overline \beta \|\dx c_0\|^2_{L^2(\E)} + \overline \gamma \|c_0\|_{L^2(\E)}^2 + \overline \delta^2 \|z\|_{L^2(0,t;L^2(\E))}^2. 
\end{align*}
This yields the second estimate and completes the proof of the lemma.
\end{proof}

As a next step, we now derive a-priori estimates for the second solution component.
\begin{lemma} \label{lem:local2}
Let (A1) hold and $T>0$. Then for any $u_0 \in L^2(\E)$ and $c \in L^\infty(0,T;H^1(\E))$, there exists a unique weak solution $u \in L^2(0,T;H^1(\E)) \cap H^1(0,T;H^1(\E)')$ of \eqref{eq:lin1} with initial value $u(0)=u_0$. 
Moreover, the solution can be bounded by
\begin{align*} 
\|u\|_{L^\infty(0,T;L^2(\E))}^2 + \underline \alpha \|u\|_{L^2(0,T;H^1(\E)}^2 \le e^{C_3 T} \|u_0\|_{L^2(\E)}^2
\end{align*}
with constant $C_3=C_3(\overline \chi,\underline \alpha,\overline C,C_G)$ that only depends on the graph, on the bounds in assumption (A1), and on the norm $\overline C:=\|\dx c\|_{L^\infty(0,T;L^2(\E))}$ of the data.
\end{lemma}
\begin{proof}
Testing \eqref{eq:lin1} with $v=u(t)$, we obtain via assumption (A1) that 
\begin{align*}
\frac{1}{2} \frac{d}{dt} &\|u(t)\|^2_{L^2(\E)} + \underline \alpha \|\dx u(t)\|^2_{L^2(\E)}
 \le \overline \chi \|u(t)\|_{L^\infty(\E)} \|\dx c(t)\|_{L^2(\E)} \|\dx u(t)\|_{L^2(\E)} \\
&\le \overline \chi (\eps \|\dx u(t)\|_{L^2(\E)} + (C_G + \tfrac{C_G^2}{4\eps}) \|u(t)\|_{L^2(\E)})
  \overline C \|\dx u(t)\|_{L^2(\E)}=(*),
\end{align*}
where we used Lemma~\ref{lem:interpolation} to estimate $\|u(t)\|_{L^\infty(\E)}$.
Choosing $\eps = \underline \alpha/(4\overline \chi \overline C)$ and applying Young's inequality yields
\begin{align*}
(*) &\le\tfrac{\underline \alpha}{2} \|\dx u(t)\|^2_{L^2(\E)} + \tfrac{C}{2} \|u(t)\|_{L^2(\E)}^2
\end{align*}
with a constant $C=C(\overline \chi,\underline \alpha,\overline C,C_G)$.  
Inserting this expression in the above estimate, rearranging some of the terms, and applying Lemma~\ref{lem:gronwall} further yields
\begin{align*}
\|u(t)\|_{L^2(\E)}^2 + \underline \alpha \|\dx u\|^2_{L^2(0,t;L^2(\E))} 
\le e^{ C t} \|u_0\|_{L^2(\E)}^2. 
\end{align*}
The result then follows by taking the maximum over $t \in [0,T]$ on both sides.
\end{proof}

By combination of the two previous results, we directly obtain the following assertion.
\begin{lemma} \label{lem:local3}
Let (A1) hold and $T>0$. Then for any $u_0 \in L^2(\E)$, $c_0 \in H^1(\E)$, and for any $z \in L^\infty(0,T;L^2(\E))$, the linearized system \eqref{eq:lin1}--\eqref{eq:lin2} 
has a unique weak solution $u \in L^2(0,T;H^1(\E)) \cap H^1(0,T;H^1(\E)')$ and $c \in L^\infty(0,T;H^1(\E)) \cup H^1(0,T;L^2(\E))$ with initial values $u(0)=u_0$ and $c(0)=c_0$. 
Moreover, there holds
\begin{align*}
\|u\|_{L^\infty(0,T;L^2(\E))} \le e^{C_4 T} \|u_0\|_{L^2(\E)} 
\end{align*}
with $C_4$ depending only on the bounds in assumption (A1), on the geometry of the graph, and monotonically on $\|c_0\|_{H^1(\E)}$, $\|z\|_{L^\infty(0,T;L^2(\E))}$, and the time horizon $T$. 
\end{lemma}

Based on the previous results, we can define a mapping
\begin{align} \label{eq:fixedpointmap}
\Phi_T : X_T \to X_T, \qquad z \mapsto u,  
\end{align}
where $X_T=L^{\infty}(0,T;L^2(\E))$
with norm $\|v\|_{X_T}=\|v\|_{L^{\infty}(0,T;L^2(\E))}$ 
and where $u$ is the first component of the solution of \eqref{eq:lin1}--\eqref{eq:lin2} 
with initial values $u_0 \in L^2(\E)$ and $c_0 \in H^1(\E)$.
By application of Lemma~\ref{lem:local3}, we can immediately deduce the following result.
\begin{lemma} \label{lem:local4}
For any $T>0$, $u_0 \in H^1(\E)$, and $c_0 \in L^2(\E)$, 
the mapping $\Phi_T$ is well defined on $X_T = L^{\infty}(0,T;L^2(\E))$. 
Moreover, for any $R > \|u_0\|_{L^2(\E)}$ there exists $T(R)>0$ such that for all $0 < T \le T(R)$,
$\Phi_T$ maps $B_{T,R}= \{ z \in X_T : \|z\|_{X_T} \le R\}$ into itself.
\end{lemma}
\begin{proof}
The assertion follows directly from the previous lemmas. 
\end{proof}

With similar arguments as employed for the proof of the assertions stated in the previous section, 
we can show that $\Phi_T$ is Lipschitz continuous on $X_T=L^{\infty}(0,T;L^2(\E))$.
\begin{lemma} \label{lem:local5}
Let (A1) hold and let $u_0$, $c_0$ and $T(R)$ be given as in Lemma~\ref{lem:local4}.
Then for any $0 < T \le T(R)$, we have
\begin{align*}
\|\Phi_{T}(z) - \Phi_{T}(\widehat z) \|_{X_{T}} \le L(T) \|z - \widehat z\|_{X_{T}} \qquad \forall z,\widehat z \in B_{T,R}
\end{align*}
with Lipschitz constant $L(T)=C_5 T e^{C_6 T}$ and constants $C_5,C_6$ independent of $T$.
\end{lemma}
\begin{proof}
Let $(u,c)$ and $(\widehat u,\widehat c)$ denote the solutions of \eqref{eq:lin1}--\eqref{eq:lin2} 
with the same initial values $u_0 \in L^2(\E)$ and $c_0 \in H^1(\E)$ but with different data $z$ and $\widehat z \in B_{T,R}$.
Then 
\begin{align*}
\la\dt c(t) - \dt \widehat c(t),q\ra_\E + \la\beta \dx c(t) - \beta \dx \widehat c(t),\dx q\ra_\E + \la\gamma c(t) - \gamma \widehat c(t),q\ra_\E &= \la\delta z(t)-\delta \widehat z(t),q\ra_{\E}  
\end{align*}
for all suitable test functions $q$ and $0 \le t \le T$. In addition, $c(0)-\widehat c(0)=0$. 
With similar arguments as in the proof of Lemma~\ref{lem:local1}, one can see that
\begin{align} \label{eq:aux}
\|\dx c - \dx \widehat c\|_{X_T} \le \overline \delta T^{1/2} \|z-\widehat z\|_{X_T}. 
\end{align}
Next observe that $u - \widehat u$ satisfies $u(0)-\widehat u(0)=0$ and, in addition, there holds
\begin{align*}
\la\dt u(t) - \widehat u(t),v\ra_\E &+ \la\alpha \dx u(t)-\alpha \widehat u(t),\dx v\ra_\E \\
&= \la\chi (u(t) - \widehat u(t)) \dx c(t),\dx v\ra_\E + \la\chi \widehat u(t) (\dx c(t) - \dx \widehat c(t)), \dx v\ra_\E
\end{align*}
for any suitable test function $v$ and a.a. $0 \le t \le T$. 
By choosing $v=u(t) - \widehat u(t)$ and applying some elementary manipulations, one can deduce that 
\begin{align*}
\frac{1}{2} \frac{d}{dt} &\|u(t) - \widehat u(t)\|^2_{L^2(\E)} + \underline \alpha \|\dx u(t) - \dx \widehat u(t)\|_{L^2(\E)}^2 \\
&\le \overline \chi \|u(t) - \widehat u(t) \|_{L^\infty(\E)} \|\dx c(t)\|_{L^2(\E)} \|\dx u(t) - \dx \widehat u(t)\|_{L^2(\E)}\\
& \qquad \qquad \ + \overline \chi \|\widehat u(t)\|_{L^\infty(\E)} \|\dx c(t) - \dx \widehat c(t)\|_{L^2(\E)} \|\dx u(t) - \dx \widehat u(t)\|_{L^2(\E)}=(i)+(ii).
\end{align*}
With similar arguments as were used to bound the term $(*)$ in the proof of Lemma~\ref{lem:local2},
the first term can be further estimated by 
\begin{align*}
(i) &\le  \frac{\underline \alpha}{4} \|\dx u(t) - \dx \widehat u(t)\|^2_{L^2(\E)} + C \|u(t) - \widehat u(t)\|_{L^2(\E)}^2
\end{align*}
with $C$ depending on the graph, on the bounds in assumption (A1), on the norms $\|c_0\|_{H^1(\E)}$ and $\|u_0\|_{L^2(\E)}$ of the initial values, as well as on $R$ and $T(R)$ as required. 
Using Lemma~\ref{lem:local2}, Lemma~\ref{lem:interpolation}, and \eqref{eq:aux}, the second term can be estimated by
\begin{align*}
(ii) \le \frac{\underline\alpha}{4} \|\dx u(t) - \dx \widehat u(t)\|_{L^2(\E)}^2 + C' \overline \delta^2 T \|z - \widehat z\|_{L^\infty(0,T;L^2(\E))}
\end{align*}
with $C'$ depending on the problem data like $C$ as required. 
A combination of these estimates and an application of Lemma~\ref{lem:gronwall} finally yields the result.
\end{proof}

Problem~\eqref{eq:sys1}--\eqref{eq:sys6} is equivalent to the fixed-point problem 
$u=\Phi_T(u)$. As a direct consequence of the previous Lemmas, one can see that for all $0<T \le T'(R)$ with $T'(R)$ 
chosen sufficiently small, $\Phi_T$ maps $B_{R,T}=\{z \in L^\infty(0,T;L^2(\E)) : \|z\|_{L^\infty(0,T;L^2(\E))} \le R\}$ into itself and is a contraction. Hence, Banach's fixed-point theorem guarantees the existence of a unique fixed-point $u \in B_{R,T}$ with $u=\Phi_T(u)$. An application of Lemma~\ref{lem:local3} then yields the assertion of the theorem.

\section{Proof of Theorem~\ref{thm:reg}}

By formally differentiating the system \eqref{eq:sys1}--\eqref{eq:sys2} we obtain 
\begin{align}
\dt w - \dx(\alpha \dx w) &= \dx(\chi w \dx c) + \dx(\chi u \dx d), \label{eq:diff1}\\
\dt d - \dx(\beta \dx d) &= \delta w - \gamma d, \label{eq:diff2}
\end{align}
where $w = \dt u$ and $d = \dt c$ are the derivatives of the solution $(u,c)$.
Similarly as in the previous section, we will prove existence of local solutions via Banach's fixed-point theorem. 
To this end, we consider the following linearized system 
\begin{alignat}{1}
\la \dt w(t),v \ra_{\E} + \la \alpha \dx w(t), \dx v \ra_{\E} = \la \chi w(t) \dx c(t), \dx v \ra_{\E} + \la \chi u(t) \dx d(t), v \ra_{\E}, \label{eq:lindt1}\\
\la \dt d(t),q \ra_{\E} + \la \beta \dx d(t), \dx q \ra_{\E} + \la \gamma d(t), q\ra_{\E} = \la \delta z(t), q \ra_{\E},  \label{eq:lindt2}
\end{alignat}
for all $v,q \in H^1(\E)$ and $0 \le t \le T$ with $z\in L^{\infty}(0,T;L^2(\E))$ given.

\begin{lemma} \label{lem:reg1}
Let (A1)--(A2) hold and $T>0$. Then for any $w_0 \in L^2(\E)$, $d_0 \in H^1(\E)$, and for any $z \in L^\infty(0,T;L^2(\E))$, the linearized system \eqref{eq:lindt1}--\eqref{eq:lindt2} 
has a unique weak solution $w \in L^2(0,T;H^1(\E)) \cap H^1(0,T;H^1(\E)')$ and $d \in L^\infty(0,T;H^1(\E)) \cap H^1(0,T;L^2(\E))$ with initial values $w(0)=w_0$ and $d(0)=d_0$. Moreover, 
\begin{align*} 
&\|w\|_{L^{\infty}(0,T;L^2(\E))} + \|w\|_{L^2(0,T;H^1(\E))} \\
&\qquad \le e^{CT}\left(\|w_0\|_{L^2(\E)} 
+ C'(T^{1/2}\|z\|_{L^{\infty}(0,T;L^2(\E))} + \|\dx d_0\|_{L^2(\E)})\|u\|_{L^2(0,T;H^1(\E))}\right).
\end{align*}
\end{lemma}

\begin{proof}
It suffices to prove the a-priori estimate. 
Existence of a unique solution can then be obtained by Galerkin approximation.
According to Lemma~\ref{lem:local1}, we have 
\begin{align*} 
\|\dx d\|_{L^{\infty}(0,t;L^2(\E))} \le C(\|d_0\|_{H^1(\E)} + t^{1/2} \|z\|_{L^{\infty}(0,t;L^2(\E))}) 
\end{align*}
for a.a. $t\in [0,T]$. 
Testing \eqref{eq:lindt1} with $v = w(t)$ we obtain the differential inequality
\begin{align*}
\frac{1}{2} \frac{d}{dt} \|w(t)\|_{L^2(\E)}^2 + \underline\alpha \| \dx w(t) \|_{L^2(\E)}^2 
&\leq \overline\chi \left(\|w(t)\|_{L^{\infty}(\E)} \| \dx c(t) \|_{L^2(\E)} \|\dx w(t)\|_{L^2(\E)} \right. \\
&\qquad \left. + \|u(t)\|_{L^{\infty}(\E)} \|\dx d(t)\|_{L^2(\E)} \| \dx w(t) \|_{L^2(\E)}\right).
\end{align*}
Integrating between $0$ and $t$, using that $c\in L^{\infty}(0,T;H^1(\E)), u\in L^2(0,T;H^1(\E))$, applying the interpolation inequality \eqref{interpolationL2} and Gronwall's lemma, we arrive at
\begin{align*}
\|w(t)\|_{L^2(\E)}^2 + &\int_0^t \|\dx w(s)\|_{L^2(\E)}^2 ds \le e^{Ct} \big( \|w_0\|_{L^2(\E)}^2 \\
&+ C'(t \sup_{0\le s\le t} \|z(s)\|^2_{L^2(\E)} + \| \dx d_0 \|^2_{L^2(\E)}) \int_0^t \|u(s)\|_{H^1(\E)}^2 ds \big).
\end{align*}
The bound of the lemma is then obtained by taking the supremum over $t$ on both sides.
\end{proof}

Similarly to the previous section, we can now define the fixed-point map
\begin{align}
	\Psi_T: X_T \to X_T, \qquad z \mapsto w,
\end{align}
where $X_T = L^{\infty}(0,T;L^2(\E))$ is chosen as before and where $w$ denotes the first component of the  solution of system \eqref{eq:lindt1}--\eqref{eq:lindt2} with given initial values $w_0 \in L^2(\E)$ and $d_0 \in H^1(\E)$. 
By the previous lemma, $\Psi_T$ is well-defined and maps $X_T$ into itself. 
As a next step, we verify that $\Psi_T$ is a contraction, if $T$ is chosen sufficiently small.
\begin{lemma} \label{lem:reg2}
Let the assumptions of Lemma~\ref{lem:reg1} be valid. Then 
\begin{align*}
\|\Psi_{T}(z) - \Psi_{T}(\widehat z) \|_{X_{T}} \le L(T) \|z - \widehat z\|_{X_{T}} \qquad \forall z,\widehat z \in X_T
\end{align*}
with Lipschitz constant $L(T)=C' P(T) T^{1/2} e^{C Q(T)}$, where $C',C$ are constants and $P(T),Q(T)$ are polynomials of $T$ that only depend on the problem data and the geometry of the graph.
\end{lemma}
\begin{proof}
Let $(w,d)$ and $(\widehat w, \widehat d)$ be two solutions of \eqref{eq:lindt1}--\eqref{eq:lindt2} with the same initial values $w_0\in L^2(\E)$ and $d_0\in H^1(\E)$, but with different data $z, \widehat z \in X_T$. 
With the same arguments as in the proof of Lemma~\ref{lem:local5}, we obtain
\begin{align}
	\|\dx d - \dx \widehat d \|_{X_T} \leq \overline \delta T^{1/2} \|z - \widehat z\|_{X_T}. \label{eq:contr}
\end{align}
Moreover, the difference $w - \widehat w$ satisfies $w(0) - \widehat w(0) = 0$ and
\begin{align*}
\la\dt (w(t) - \widehat w(t)),v\ra_\E &+ \la\alpha \dx (w(t)-\widehat w(t)),\dx v\ra_\E \\
&= \la\chi (w(t) - \widehat w(t)) \dx c(t),\dx v\ra_\E + \la\chi \widehat u(t) \dx(d(t) - \widehat d(t)), \dx v\ra_\E.
\end{align*}
By choosing $v=w(t) - \widehat w(t)$, one can then deduce that 
\begin{align*}
\frac{1}{2} \frac{d}{dt} &\|w(t) - \widehat w(t)\|^2_{L^2(\E)} + \underline\alpha \|\dx w(t) - \dx \widehat w(t)\|_{L^2(\E)}^2 \\
&\le \overline\chi(\|w(t) - \widehat w(t)\|_{L^{\infty}(\E)} \|\dx c(t)\|_{L^2(\E)} \|\dx w(t) - \dx \widehat w(t)\|_{L^2(\E)} 
	\\&\quad+ \|u(t)\|_{L^{\infty}(\E)} \| \dx d(t) - \dx \widehat d(t) \|_{L^2(\E)} \|\dx w(t) - \dx \widehat w(t)\|_{L^2(\E)}),
\end{align*}
Applying the interpolation inequality \eqref{interpolationL2}, the estimate \eqref{eq:contr}, and using the polynomial bounds from Theorem~\ref{thm:global}, we can deduce the assertion of the theorem by integration and an application of Gronwall's lemma.
\end{proof}

\subsection*{Proof of Theorem~\ref{thm:reg}}
Lemma~\ref{lem:reg2} shows that $\Psi_T$ is a contraction on $X_T$ for $T$ sufficiently small. 
By Banach's fixed-point theorem, the system \eqref{eq:diff1}--\eqref{eq:diff2} thus has a unique local solution. 
Since the system is linear w.r.t. $w$ and $d$ we can extend the solution to arbitrary time-intervals via a bootstrap argument.  
The condition (A3) of Theorem~\ref{thm:reg} guarantees that $w_0=\dt u(0) \in L^2(\E)$ and $d_0=\dt c(0) \in H^1(\E)$. 
The piecewise $H^2$-bound of Theorem~\ref{thm:reg} is obtained from the strong form \eqref{eq:sys1}--\eqref{eq:sys2} 
of the system and the previous estimates.


\begin{thebibliography}{10}

\bibitem{Berge}
C.~Berge.
\newblock {\em Graphs. 2nd rev.}
\newblock North-Holland, Amsterdam, New~York, Oxford, 1985.

\bibitem{BorscheGoettlichKlarSchillen14}
R.~Borsche, S.~G\"ottlich, A.~Klar, and P.~Schillen.
\newblock The scalar {K}eller-{S}egel model on networks.
\newblock {\em Math. Model. Meth. Appl. Sci.}, 24:221--247, 2014.

\bibitem{BorscheKlarPham16}
R.~Borsche, J.~Kall, A.~Klar, and T.~N.~H. Pham.
\newblock Kinetic and related macroscopic models for chemotaxis on networks.
\newblock {\em Math. Models Methods Appl. Sci.}, 26:1219--1242, 2016.

\bibitem{BrennerScott08}
S.~C. Brenner and L.~R. Scott.
\newblock {\em The Mathematical Theory of Finite Element Methods}.
\newblock Springer, 2008.

\bibitem{BrettiNataliniRibot14}
G.~Bretti, R.~Natalini, and M.~Ribot.
\newblock A hyperbolic model of chemotaxis on a network: a numerical study.
\newblock {\em ESAIM: Math. Model. Numer. Anal.}, 48:231--258, 2014.

\bibitem{CamilliCorrias17}
F.~Camilli and L.~Corrias.
\newblock Parabolic models for chemotaxis on weighted networks.
\newblock {\em J. Math. Pures Appl.}, 108:459--480, 2017.

\bibitem{ChertockEpshteynHuKurganov17}
A.~Chertock, Y.~Epshteyn, H.~Hu, and A.~Kurganov.
\newblock High-order positivity-preserving hybrid
  finite-volume-finite-difference methods for chemotaxis systems.
\newblock {\em Adv. Comput. Math.}, 44:327--350, 2017.

\bibitem{ChildressPercus81}
S.~Childress and J.~K. Percus.
\newblock Nonlinear aspects of chemotaxis.
\newblock {\em Math. Biosciences}, 56:217--237, 1981.

\bibitem{Clement75}
P.~Cl\'ement.
\newblock Approximation by finite element functions using local regularization.
\newblock 9:77--84, 1975.

\bibitem{DL5}
R.~Dautray and J.-L. Lions.
\newblock {\em Mathematical analysis and numerical methods for science and
  technology. {V}ol. 5}.
\newblock Springer-Verlag, Berlin, 1992.
\newblock Evolution problems. I, With the collaboration of Michel Artola,
  Michel Cessenat and H{\'e}l{\`e}ne Lanchon, Translated from the French by
  Alan Craig.

\bibitem{Epshteyn09}
Y.~Epshteyn.
\newblock Discontinuous {G}alerkin methods for the chemotaxis and haptotaxis
  models.
\newblock {\em J. Comput. Appl. Math.}, 224:168--181, 2009.

\bibitem{Evans10}
L.~Evans.
\newblock {\em Partial differential equations}.
\newblock American Mathematical Society, 2010.

\bibitem{Filbet06}
F.~Filbet.
\newblock A finite volume scheme for the {P}atlak-{K}eller-{S}egel chemotaxis
  model.
\newblock {\em Numer. Math.}, 104:457--488, 2006.

\bibitem{Gosse12}
L.~Gosse.
\newblock Asymptotic-preserving and well-balanced scheme for the 1{D}
  {C}attaneo model of chemotaxis movement in both hyperbolic and diffusive
  regimes.
\newblock {\em J. Math. Anal. Appl}, pages 964--983, 2012.

\bibitem{HerreroVelazquez97}
M.~A. Herrero and J.~J.~L. Vel\'azquez.
\newblock A blow-up mechanism for a chemotaxis model.
\newblock {\em Ann. Scuola Normale Superiore}, 24:633--683, 1997.

\bibitem{Heuser82}
H.~G. Heuser.
\newblock {\em Functional analysis}.
\newblock John Wiley \&\ Sons, Ltd., Chichester, 1982.
\newblock Translated from the German by John Horv\'ath, A Wiley-Interscience
  Publication.

\bibitem{Heywood90}
J.~G. Heywood and R.~Rannacher.
\newblock Finite-element approximation of the nonstationary navier--stokes
  problem. part iv: Error analysis for second-order time discretization.
\newblock {\em SIAM J. Numer. Anal.}, 27:353--384, 1990.

\bibitem{HillenPainter01}
T.~Hillen and K.~Painter.
\newblock Global existence for a parabolic chemotaxis model with the
  preventiona of overcrowding.
\newblock {\em Adv. Appl. Math.}, 26:280--301, 2001.

\bibitem{HillenPainter09}
T.~Hillen and K.~J. Painter.
\newblock A user’s guide to pde models for chemotaxis.
\newblock {\em J. Math. Biol.}, 58:183, 2009.

\bibitem{HillenPotapov04}
T.~Hillen and A.~Potapov.
\newblock The one-dimensional chemotaxis model: global existence and asymptotic
  profile.
\newblock {\em Math. Meth. Appl. Sci.}, 27:1783--1801, 2004.

\bibitem{Horstmann03}
D.~Horstmann.
\newblock From 1970 util present: the {K}eller-{S}egel model in chemotaxis and
  its consequences {I}.
\newblock {\em Jahresberichte DMV}, 105, 2003.

\bibitem{JaegerLuckhaus92}
W.~J\"ager and S.~Luckhaus.
\newblock On explosions of solutions to a system of partial differential
  equations modelling chemotaxis.
\newblock {\em Trans. Am. Math. Soc.}, 329:817--824, 1992.

\bibitem{KellerSegel71}
E.~F. Keller and L.~A. Segel.
\newblock Model for chemotaxis.
\newblock {\em J. Theor. Biology}, 30:225--234, 1971.

\bibitem{Mugnolo14}
D.~Mugnolo.
\newblock {\em Semigroup methods for evolution equations on networks}.
\newblock Springer, 2014.

\bibitem{NagaiSenbaYoshida97}
T.~Nagai, T.~Senba, and K.~Yoshida.
\newblock Application of the {M}oser-{T}rudinger inequality to a parabolic
  system of chemotaxis.
\newblock {\em Funkc. Ekvacioj}, 40:411--433, 1997.

\bibitem{NakaguchiYagi01}
E.~Nakaguchi and A.~Yagi.
\newblock Full discrete approximations by {G}alerkin method for chemotaxis
  growth model.
\newblock {\em Nonlin. Anal.}, 47:6097--6107, 2001.

\bibitem{NataliniRibot12}
R.~Natalini and M.~Ribot.
\newblock Asymptotic high order mass-preserving schemes for a hyperbolic model
  of chemotaxis.
\newblock {\em SIAM J. Numer. Anal.}, 50:883--905, 2012.

\bibitem{OsakiYagi01}
K.~Osaki and A.~Yagi.
\newblock Finite eimensional attractor for one-dimensional {K}eller-{S}egel
  equations.
\newblock {\em Funkcialaj Ekvacioj}, 44:441--469, 2001.

\bibitem{Plemmons77}
R.~J. Plemmons.
\newblock {$M$}-matrix characterizations. {I}. {N}onsingular {$M$}-matrices.
\newblock {\em Linear Algebra and Appl.}, 18:175--188, 1977.

\bibitem{Roubicek13}
T.~s. Roub\'\i~\v cek.
\newblock {\em Nonlinear partial differential equations with applications},
  volume 153 of {\em International Series of Numerical Mathematics}.
\newblock Birkh\"auser/Springer Basel AG, Basel, second edition, 2013.

\bibitem{Saito12}
N.~Saito.
\newblock Error analysis of a conservative finite-element approximation for the
  keller-segel system of chemotaxis.
\newblock {\em Commun. Pure Appl. Anal}, 11:339--364, 2012.

\bibitem{StrehlSokolovKuzminHorstmannTurek13}
R.~Strehl, A.~Sokolov, D.~Kuzmin, D.~Horstmann, and S.~Turek.
\newblock A positivity-preserving finite element method for chemotaxis problems
  in 3{D}.
\newblock {\em J. Comput. Appl. Math.}, 239:290--303, 2013.

\bibitem{Thomee06}
V.~Thom\'ee.
\newblock {\em Galerkin finite element methods for parabolic problems},
  volume~25 of {\em Springer Series in Computational Mathematics}.
\newblock Springer-Verlag, Berlin, second edition, 2006.

\bibitem{Varga71}
R.~S. Varga.
\newblock {\em Functional Analysis and Approximation Theory in Numerical
  Analysis}.
\newblock CBMS-NSF Regional Conference Series in Applied Mathematics. SIAM,
  Philadelphia, 1971.

\bibitem{Wheeler73}
M.~F. Wheeler.
\newblock A priori ${L_2}$ error estimates for {G}alerkin approximations to
  parabolic partial differential equations.
\newblock {\em SIAM J. Numer. Anal.}, 10:723--759, 1973.

\bibitem{Yagi97}
A.~Yagi.
\newblock Norm behavior of solutions to the parabolic model of chemotaxis.
\newblock {\em Mathematica Japonica}, 45:241--265, 1997.

\end{thebibliography}

\end{document}